\documentclass{amsart}

\usepackage{amsmath, amssymb, amsthm, amsfonts}

\input xy
\xyoption{all}

\usepackage{tikz}
\usepackage{dynkin-diagrams}
\usepackage{array}   
\newcolumntype{C}{>{$}c<{$}} 
\newcolumntype{L}{>{$}l<{$}} 
\usepackage{diagbox}
\usepackage{pdflscape} 
\usepackage{float} 
\usepackage{makecell}
\usepackage{adjustbox}
\usepackage[utf8]{inputenc}

\usepackage{hyperref}

\swapnumbers
\numberwithin{equation}{section}

\theoremstyle{plain}
\newtheorem{theorem}[subsection]{Theorem}
\newtheorem{lemma}[subsection]{Lemma}
\newtheorem{prop}[subsection]{Proposition}
\newtheorem{cor}[subsection]{Corollary}

\theoremstyle{definition}
\newtheorem{defn}[subsection]{Definition}

\newtheorem{remark}[subsection]{Remark}

\def\AA{\mathbb{A}}

\def\CC{\mathbb{C}}

\def\GG{\mathbb{G}}

\def\NN{\mathbb{N}}

\def\PP{\mathbb{P}}
\def\QQ{\mathbb{Q}}
\def\RR{\mathbb{R}}

\def\XX{\mathbb{X}}

\def\ZZ{\mathbb{Z}}

\newcommand\cB{\mathcal{B}}

\newcommand\cE{\mathcal{E}}
\newcommand\cF{\mathcal{F}}

\newcommand\cH{\mathcal{H}}

\newcommand\cK{\mathcal{K}}
\newcommand\cL{\mathcal{L}}
\newcommand\cM{\mathcal{M}}
\newcommand\cN{\mathcal{N}}
\newcommand\cO{\mathcal{O}}
\newcommand\cP{\mathcal{P}}
\newcommand\cQ{\mathcal{Q}}

\newcommand\cT{\mathcal{T}}

\newcommand\cX{\mathcal{X}}

\def\bI{\mathbf{I}}

\def\bP{\mathbf{P}}
\def\bQ{\mathbf{Q}}

\newcommand\frb{\mathfrak{b}}

\newcommand\frg{\mathfrak{g}}
\newcommand\frh{\mathfrak{h}}

\newcommand\frl{\mathfrak{l}}
\newcommand\fm{\mathfrak{m}}
\newcommand\frn{\mathfrak{n}}
\newcommand\frp{\mathfrak{p}}
\newcommand\frq{\mathfrak{q}}

\newcommand\frt{\mathfrak{t}}

\newcommand\frz{\mathfrak{z}}

\newcommand\aff{\textup{aff}}

\newcommand\alg{\textup{alg}}

\newcommand{\codim}{\textup{codim}}

\newcommand\diag{\textup{diag}}

\newcommand\ev{\textup{ev}}

\newcommand{\Fl}{\textup{Fl}}

\newcommand\Gal{\textup{Gal}}

\newcommand{\Gr}{\textup{Gr}}

\newcommand\IC{\textup{IC}}
\newcommand\id{\textup{id}}
\renewcommand{\Im}{\textup{Im}}
\newcommand{\Ind}{\textup{Ind}}

\newcommand\Irr{\textup{Irr}}

\newcommand\Lie{\textup{Lie}\ }

\newcommand\nil{\textup{nil}}
\newcommand{\Nm}{\textup{Nm}}

\newcommand{\Orb}{\textup{Orb}}

\newcommand\pt{\textup{pt}}

\newcommand{\rat}{\textup{rat}}
\newcommand{\red}{\textup{red}}
\newcommand{\reg}{\textup{reg}}

\newcommand{\Res}{\textup{Res}}

\newcommand\sgn{\textup{sgn}}

\newcommand\Span{\textup{Span}}
\newcommand\Spec{\textup{Spec}\ }

\newcommand\std{\textup{std}}

\newcommand\Sym{\textup{Sym}}

\newcommand{\Tr}{\textup{Tr}}

\newcommand\triv{\textup{triv}}

\newcommand\Aut{\textup{Aut}}

\newcommand{\Part}{\textup{Part}}

\newcommand\GL{\textup{GL}}
\newcommand\gl{\mathfrak{gl}}

\newcommand\SL{\textup{SL}}
\renewcommand\sl{\mathfrak{sl}}

\newcommand\SO{\textup{SO}}

\newcommand\so{\mathfrak{so}}
\newcommand\Sp{\textup{Sp}}
\renewcommand\sp{\mathfrak{sp}}

\newcommand{\Gm}{\GG_m}
\def\Ga{\GG_a}

\newcommand{\ad}{\textup{ad}}
\newcommand{\Ad}{\textup{Ad}}
\renewcommand\sc{\textup{sc}}
\newcommand{\der}{\textup{der}}

\newcommand\xcoch{\mathbb{X}_*}

\newcommand{\incl}{\hookrightarrow}
\newcommand{\isom}{\stackrel{\sim}{\to}}

\newcommand{\surj}{\twoheadrightarrow}

\newcommand{\Qlbar}{\overline{\QQ}_\ell}

\newcommand{\const}[1]{\overline{\QQ}_{\ell,#1}}

\renewcommand{\j}[1]{\langle{#1}\rangle}
\newcommand{\wt}[1]{\widetilde{#1}}

\newcommand\quash[1]{}
\newcommand\un{\underline}

\newcommand{\ov}{\overline}
\newcommand{\bs}{\backslash}
\newcommand{\tl}[1]{[\![#1]\!]}
\newcommand{\lr}[1]{(\!(#1)\!)}
\newcommand\sss{\subsubsection}
\newcommand\xr{\xrightarrow}
\newcommand\op{\oplus}
\newcommand\ot{\otimes}

\newcommand{\sslash}{\mathbin{/\mkern-6mu/}}
\renewcommand\c{\circ}

\newcommand{\cohog}[2]{\textup{H}^{#1}({#2})}     

\newcommand\upH{\textup{H}}

\renewcommand\a\alpha
\renewcommand\b\beta
\newcommand\G\Gamma
\newcommand\g\gamma
\renewcommand\d\delta
\newcommand\D\Delta
\newcommand{\e}{\epsilon}
\newcommand{\io}{\iota}

\renewcommand{\th}{\theta}
\newcommand{\ph}{\varphi}
\renewcommand{\r}{\rho}

\renewcommand{\t}{\tau}

\newcommand{\z}{\zeta}
\newcommand{\ep}{\epsilon}

\renewcommand{\l}{\lambda}
\renewcommand{\L}{\Lambda}
\newcommand{\om}{\omega}

\newcommand\sh{\sharp}

\newcommand\vn{\varnothing}

\newcommand\nb{\nabla}

\newcommand\Hod{\textup{Hod}}
\newcommand\dR{\textup{dR}}
\newcommand\RT{\textup{RT}}
\newcommand\cle{\preccurlyeq}
\newcommand\cge{\succcurlyeq}
\newcommand\Cone{\textup{Cone}}

\title{A Deligne-Simpson problem for irregular $G$-connections over $\PP^{1}$}

\author{Konstantin Jakob}
\thanks{K.J. acknowledges support by the Deutsche Forschungsgemeinschaft (DFG, German Research Foundation) Individual Research Grant 566801746 and (through Timo Richarz) by the European Research Council (ERC) under Horizon Europe (grant agreement no 101040935), by the Deutsche Forschungsgemeinschaft (DFG, German Research Foundation) TRR 326 \textit{Geometry and Arithmetic of Uniformized Structures}, project number 444845124 and the LOEWE professorship in Algebra, project number LOEWE/4b//519/05/01.002(0004)/87}

\author{Zhiwei Yun}
\thanks{Z.Y. is partially supported by the Simons Foundation.}
\date{}

\begin{document}

\begin{abstract} We give an algebraic and a geometric criterion for the existence of $G$-connections on $\PP^{1}$ with prescribed irregular type with equal slope at $\infty$ (isoclinic) and with regular singularity of prescribed residue at $0$. The algebraic criterion is in terms of an irreducible module of the rational Cherednik algebra, and the geometric criterion is in terms of affine Springer fibers. We use these criteria to give complete solutions to the isoclinic Deligne-Simpson problem for classical groups, and for arbitrary $G$ when the slope at $\infty$ has Coxeter number as the denominator. Among our solutions, we classify the cohomologically rigid connections, and obtain new cases in types $B,D$ and $F_4$. 
\end{abstract}

\maketitle
\tableofcontents


\section{A Deligne-Simpson problem and its solution}\label{s:cond on conn}

\subsection{Deligne-Simpson problem} In general, the Deligne-Simpson problem asks for criteria for the existence of a local system (in the topological, $\ell$-adic or de Rham setting) on an algebraic curve with prescribed monodromy/ramification at finitely many given punctures. The problem was first phrased for $G=\GL_n$ and the curve $\PP^1_{\CC}$ in the topological setting by P. Deligne in the eighties, and C. Simpson was the first to systematically study it \cite{Simpson}. Aside from the existence of solutions to the Deligne-Simpson problem, one can ask for uniqueness of the solution, i.e. (physical) rigidity, and for the existence and properties of moduli spaces classifying local systems with prescribed ramification. 

Rigid local systems have been studied prominently by N. Katz, who devised a construction algorithm for them \cite{Katz}. This algorithm was used by W. Crawley-Boevey \cite{CB}, to give a complete solution to the Deligne-Simpson problem for irreducible topological local systems. In his solution, the existence problem is reinterpreted in terms of representations of a quiver attached to the conjugacy classes of local monodromy. Crawley-Boevey's solution is given in terms of the Kac-Moody root system attached to a quiver, and he also classifies the cases in which the solution is unique (physically rigid). 

Following the construction of moduli spaces of irregular singular differential equations with prescribed behaviour at the singularities due to P. Boalch \cite{Boalch2} and K. Hiroe and D. Yamakawa \cite{HY}, Hiroe gave a complete solution to the Deligne-Simpson problem for irregular differential equations of integral slope \cite{Hiroe}. Again, his solution reinterprets the problem in terms of representations of quivers attached to the fixed ramification behaviour. 

For non-integral slope, little is known about the Deligne-Simpson problem. In addition, all of the above results deal with the case $G=\GL_n$, and next to nothing was known about the Deligne-Simpson problem for general connected reductive groups. One of the problems when trying to generalize is the absence of an analogue of the quiver theoretic formulation outside of type $A$. For a more thorough exposition of the history of the Deligne-Simpson problem we refer to V. Kostov's survey article \cite{Kostov} (who formulated the additive version of the Deligne-Simpson problem), and to D. Sage's article \cite{Sage}. 

Recently, M. Kulkarni, N. Livesay, J. Matherne, B. Nguyen and D. Sage solved the Deligne-Simpson problem for $\GL_n$ over $\PP^{1}$ in a special case with non-integral slope, called maximally ramified \cite{KLMNS}, with two singularities.  In this paper, we broadly generalize their setup, keeping the two singularities.

\subsection{Main result (special case)}
Let $G$ be an arbitrary connected reductive group. This paper gives a solution to the Deligne-Simpson problem for $G$-connections on $\PP^{1}$ with two punctures, $0$ and $\infty$, and specific ramification constraints. We give a purely algebraic criterion for the existence of such connections. Our strategy revolves around the geometry of non-abelian Hodge moduli spaces first considered in \cite{BBAMY}. We define moduli spaces of connections with prescribed local behaviour, and, using ideas from non-abelian Hodge theory, relate them to moduli spaces of Higgs bundles. On the other hand, these spaces are homotopy equivalent to affine Springer fibers, and their cohomology carries an action of the rational Cherednik algebra of type $W$ (the Weyl group of $G$) by \cite{OY}.

We state here a special and weaker version of our main results, leaving the precise definitions and the stronger versions to the rest of the introduction.
\begin{theorem}[Special case of Theorems \ref{th:DS} and \ref{th:DS Gr}]\label{th:intro main} Let $\nu=d/m\in \QQ_{>0}$ (in lowest terms) where $m$ is a regular elliptic number of $W$. Let $\cO$ be a nilpotent orbit in $\frg$. Then the following are equivalent:
\begin{itemize}
\item There exists an algebraic $G$-connection $\nb$ on $\PP^{1}\setminus\{0,\infty\}$ that is isoclinic of slope $\nu$ at $\infty$ and has regular singularity at $0$ with residue in $\cO$;
\item The representation $E_{\cO}$ of the Weyl group $W$ (attached to the nilpotent orbit $\cO$ via the Springer correspondence) appears in the irreducible module $L_{\nu}(\triv)$ for the rational Cherednik algebra of type $W$.
\item There exists an element $\psi\in \ov\cO+t\frg\tl{t}\subset \frg\tl{t}$ that is homogeneous of slope $\nu$ (here $\ov{\cO}$ is the closure of $\cO$). 
\end{itemize}
\end{theorem}

We also illustrate by examples that our criterion can be used to effectively compute those pairs $(\nu, \cO)$ for which the above Deligne-Simpson problem has a positive solution. We fully solve the case when the slope $\nu$ has denominator $h$ (the Coxeter number) for arbitrary $G$, see Theorem \ref{th:h main}. For classical types, we give complete solutions for arbitrary slope $\nu$ in Theorem \ref{th:cl}. We partially prove \cite[Conjecture 5.7.]{Sage} in classical types, and generalize it to arbitrary slopes $\nu$, see Theorem \ref{th:hq main}. Additionally, we illustrate how to solve the Deligne-Simpson problem for more general slopes in exceptional types, by explicating the solution in type $F_4$, see Section \ref{s:F4soln}. In conclusion, we solve the isoclinic Deligne-Simpson problem up to finitely many cases in the exceptional types.

Among our solutions, we classify the cohomologically rigid cases in Section  \ref{s:rig}, which include most of the rigid connections that appear in recent works influenced by the work of Frenkel and Gross \cite{FG}, such as \cite{Ch}, \cite{JKY} and \cite{KS}. Our classification proves \cite[Conjecture 5.9.]{Sage}, and additionally includes new rigid connections in types $B,D$ and $F_4$;  it would be interesting to construct their $\ell$-adic counterparts.

\subsection{Setup}
Let $G$ be a connected reductive group over $\CC$. Let $T\subset G$ be a maximal torus. Let $\frg$ and $\frt$ be the Lie algebras of $G$ and $T$, and $W$ be the Weyl group. 

For the curve $X:=\PP^{1}$ over $\CC$ with affine coordinate $t$, we use $\t=t^{-1}$ as the affine coordinate on $\PP^{1}-\{\infty\}$, so that $\t=0$ corresponds to $t=\infty$. We will consider meromorphic $G$-connections $(\cE,\nb)$ on $X$ that are regular over $X\bs \{0,\infty\}$ and have specific types of singularities at $0$ and $\infty$.

\subsection{Isoclinic irregular type}\label{ss:isoc irr type}
Let $D^{\times}_{\infty}=\Spec \CC\lr{\t}$ be the formal punctured disk at $\infty\in X$. For $m\in\NN$, let 
$D^{(m),\times}_{\infty}=\Spec \CC\lr{\t^{1/m}}$ be the degree $m$ ramified cover of $D^{\times}_{\infty}$.  

For a $G$-connection $(\cE,\nb)$ over $D^{\times}_{\infty}$, it is well-known that for some $m\in\NN$,  the pullback $(\cE, \nb)|_{D_{\infty}^{(m),\times}}$ is gauge equivalent to 
\begin{equation*}
d+\left(A(\t^{1/m})+\cdots\right)\frac{d\t}{\t},
\end{equation*}
where
\begin{equation}\label{irr type A}
A(\t^{1/m})=\frac{A_{-d/m}}{\t^{d/m}}+\frac{A_{-(d-1)/m}}{\t^{(d-1)/m}}+\cdots+\frac{A_{-1/m}}{\t^{1/m}}\in \frt[\t^{-1/m}]
\end{equation}
and the $\cdots$ part lies in $\frg\tl{\t^{1/m}}$. 

We call $(\cE,\nb)$ {\em isoclinic of slope $\nu=d/m$} if the leading coefficient $A_{-d/m}$ above is regular semisimple (this condition is independent of the choice of $m$ and the gauge transformation). Isoclinic refers to the fact that for the adjoint connection $\Ad(\cE,\nb)$, all nonzero slopes are equal to $\nu$. In this case, we say that $A=A(\t^{1/m})$ is an {\em isoclinic irregular type} of slope $\nu$, and that $(\cE,\nb)$ is {\em isoclinic of irregular type $A$}. 

In Lemma \ref{l:A} we will give an intrinsic characterization of isoclinic irregular types of slope $\nu$. In particular, we shall see that the denominator of $\nu$ in lowest terms is a {\em regular number}, i.e., the order of a regular element $w$ in the Weyl group $W$ in the sense of Springer \cite{Spr}.

\begin{defn} Let $A$ be an isoclinic irregular type of slope $\nu>0$. Let $\cO\subset\frg$ be an adjoint orbit. 
\begin{enumerate}
\item A meromorphic algebraic $G$-connection $(\cE,\nb)$ on $X$ is of {\em type $(A,\cO)$} if:
\begin{itemize}
\item $(\cE, \nb)$ has no singularities on $X\bs \{0,\infty\}$;
\item $(\cE, \nb)|_{D_{\infty}^{\times}}$ is isoclinic of irregular type $A$;
\item $(\cE, \nb)|_{D_{0}^{\times}}$ has regular singularities with residue in the adjoint orbit $\cO$.
\end{itemize}
We say $(\cE,\nb)$ is of {\em type $(\nu,\cO)$} if it is of type $(A,\cO)$ for some isoclinic irregular type of slope $\nu$.
\item The {\em isoclinic Deligne-Simpson problem} of type $(A,\cO)$ (resp. type $(\nu,\cO)$) asks whether there exists a meromorphic algebraic $G$-connection $(\cE,\nb)$ on $X$ of type $(A,\cO)$ (resp. type $(\nu,\cO)$). We abbreviate these problems as $DS(A,\cO)$ and $DS(\nu,\cO)$, or $DS_{G}(A,\cO)$ and $DS_{G}(\nu,\cO)$ if we want to emphasize the group $G$. 
\end{enumerate}
\end{defn}

\subsection{Nilpotent orbit from $\cO$}\label{ss:Onil} Let $\cO\subset\frg$ be an adjoint orbit.  Let $x\in \cO$ with Jordan decomposition $x=x_{s}+x_{n}$. Let $L=C_{G}(x_{s})$, then $x_{n}$ is a nilpotent element in $\frl=\Lie L$. Let $\cO^{L}_{x_{n}}$ be the nilpotent orbit in $\frl$ of $x_{n}$. Let $\cO^{\nil}$ be the Lusztig-Spaltenstein induction $\Ind_{\frl}^{\frg}(\cO^{L}_{x_{n}})$, see \cite{LS}. By definition, $\cO^{\nil}$ is characterized as follows. Let $P$ be a parabolic subgroup of $G$ with Levi $L$. Let $\frp=\Lie  P$ with nilpotent radical $\frn_{P}$. Then $\cO^{\nil}$ is the unique nilpotent orbit of $\frg$ such that $\cO^{\nil}\cap \frp$ is dense in $\cO^{L}_{x_{n}}+\frn_{P}$. It is shown in \cite{LS} that  $\Ind_{\frl}^{\frg}(\cO^{L}_{x_{n}})$ is independent of the choice of $P$ with Levi $L$.

In Lemma \ref{l:limit nil} we will give another description of $\cO^{\nil}$ in terms of the algebraic asymptotic cone of $\cO$ in the sense of Adams-Vogan \cite[Def. 3.6]{AV}. 

Let $E_{\cO^{\nil}}$ be the irreducible representation of $W$ attached to the trivial local system on the nilpotent orbit $\cO^{\nil}$ under the Springer correspondence. Our convention is that when $\cO^{\nil}$ is the regular orbit (which happens if and only if $\cO$ consists of regular elements in $\frg$), $E_{\cO^{\nil}}$ is the trivial representation of $W$.

\subsection{An irreducible module of the rational Cherednik algebra} To state the algebraic solution to the Deligne-Simpson problem for connections of type $(A,\cO)$, we need to consider a simple module over the rational Cherednik algebra.

For the Weyl group $W$ and $\nu\in \QQ$, let $\cH^{\rat}_{\nu}(W)$ be the rational Cherednik algebra with central charge $\nu$ introduced in \cite{EG}.  When the denominator of $\nu$ in lowest terms is an {\em elliptic regular number} for $W$ (i.e., the order of an elliptic regular element in $W$), $\cH^{\rat}_{\nu}(W)$ has a unique finite-dimensional irreducible module $L_{\nu}(\triv)$ that is the quotient of the Verma module $M_{\nu}(\triv)$ attached to the trivial representation of $W$. Since  $\cH^{\rat}_{\nu}(W)$ contains the group algebra of $W$ as a subalgebra, $L_{\nu}(\triv)$ is in particular a $W$-module.

We can now state our algebraic solution to the Deligne-Simpson problem under the ellipticity assumption.

\begin{theorem}[Algebraic criterion, elliptic case]\label{th:DS} Let $A$ be an isoclinic irregular type of slope $\nu>0$ whose denominator is an elliptic regular number of $W$. Let $\cO$ be an adjoint orbit of $\frg$. Then $DS(A,\cO)$ has an affirmative answer if and only if $E_{\cO^{\nil}}$ appears in $L_{\nu}(\triv)$ as a $W$-module. 
\end{theorem}
In particular, the answer to $DS(A,\cO)$ depends only on the slope $\nu$ of $A$, not $A$ itself, i.e., $DS(A,\cO)$ has the same answer as $DS(\nu,\cO)$.

\subsection{The non-elliptic case}\label{ss:non ell} Now we consider the more general case where the denominator $m$ of $\nu$ is only assumed to be a regular number of $W$ but not necessarily elliptic. Then there is a parabolic subgroup $W'\subset W$ such that $m$ is an elliptic regular number of $W'$. Indeed, if $w$ is a regular element in $W$ of order $m$, we let $W'$ be the pointwise stabilizer of $\frt^{w}$ under $W$, then $w$ is a regular elliptic element in $W'$. Let $M$ be the Levi subgroup of $G$ containing $T$ with Weyl group $W'$.

\begin{theorem}[Reduction to the elliptic case]\label{th:DS general} Let $A$ be an isoclinic irregular type of slope $\nu>0$.   Let $\cO$ be an adjoint orbit of $\frg$. Then $DS_{G}(A,\cO)$ has an affirmative answer if and only if for some nilpotent orbit $\cO'$ of $M$ such that $\cO'\subset \ov{\cO^{\nil}}$, $DS_{M}(\nu,\cO')$ has an affirmative answer.
\end{theorem}

Again, the answer to $DS(A,\cO)$ depends only on the slope $\nu$ of $A$, not $A$ itself.

\subsection{Affine Springer fiber} The proof of the main results relies on constructing moduli spaces of connections with prescribed singularity types and relating such moduli spaces to affine Springer fibers.  Write $\nu=d/m$ in lowest terms. The leading term $A_{-\nu}\t^{-\nu}=A_{-d/m}t^{d/m}\in \frt[t^{1/m},t^{-1/m}]$ of an isoclinic irregular type $A$ of slope $\nu$ can be $G^{\ad}\lr{t^{1/m}}$-conjugated to a regular semisimple element $\psi\in \frg[t,t^{-1}]$. The element $\psi$ is {\em homogeneous of slope $\nu$} in the sense that for $s\in \CC^{\times}$, $\psi(s^{m}t)$ is in the same $G^{\ad}\lr{t}$-orbit of $s^{d}\psi(t)$. We consider its affine Springer fiber in the affine Grassmannian
\begin{equation}\label{ASF}
\Gr_{\psi}=\{gG\tl{t}\in \Gr_{G}:=G\lr{t}/G\tl{t}|\Ad(g^{-1})\psi\in\frg\tl{t}\}.
\end{equation}

The following geometric criterion for the isoclinic Deligne-Simpson problem is an intermediate step towards the proofs of Theorems \ref{th:DS} and \ref{th:DS general}. 
 \begin{theorem}[Geometric criterion]\label{th:DS Gr} Let $A$ be an isoclinic irregular type of slope $\nu>0$. Let $\psi\in \frg\lr{t}$ be in the same $G^{\ad}\lr{t^{1/m}}$-orbit as the leading term $A_{-\nu}t^{\nu}$. Let $\cO$ be an adjoint orbit of $\frg$.  Then $DS(A,\cO)$ has an affirmative answer if and only if there exists a point $gG\tl{t}\in\Gr_{\psi}$ such that $\Ad(g^{-1})\psi\in \ov{\cO^{\nil}}+t\frg\tl{t}$.
 \end{theorem}

This geometric criterion has some immediate consequences.

\begin{cor}\label{c:larger O} Let $A$ be an isoclinic irregular type of slope $\nu>0$. Let $\cO'$ be a nilpotent orbit and $\cO$ an arbitrary adjoint orbit of $\frg$. Suppose $\cO'$ is in the closure of $\Gm\cdot\cO$ (scalings of elements in $\cO$), and $DS(A,\cO')$ has an affirmative answer, then so does $DS(A,\cO)$.
\end{cor}

\begin{cor}\label{c:sp case} Let $A$ be an isoclinic irregular type of slope $\nu>0$, and $\cO$ be an adjoint orbit of $\frg$. Then $DS(A,\cO)$ has an affirmative answer in any of the following cases:
\begin{enumerate}
\item $\nu$ arbitrary and $\cO$ is a regular adjoint orbit. 
\item $\nu\ge1$ and $\cO$ arbitrary.
\end{enumerate}
\end{cor}

The proofs of the above results are carried out in Sections \ref{s:mod} and \ref{s:proof}.

\subsection{Complete solutions in the Coxeter cases}
Let $G$ be almost simple and $h$ be its Coxeter number. In Section \ref{s:ex} we use Theorem \ref{th:DS} to determine for which pairs $(d/h,\cO)$, where $\gcd(d,h)=1$ and $\cO$ is an adjoint orbit, $DS(d/h,\cO)$ has an affirmative answer. The result can be summarized as:

\begin{theorem}[See Theorem \ref{th:h main}]\label{th:intro h} Let $d\in \NN$ be coprime to $h$. Then there is a nilpotent orbit $\cO_{d/h}$ of $\frg$, explicitly determined in all cases, such that $DS(d/h,\cO)$ has an affirmative answer if and only if $\cO_{d/h}\subset \ov{\Gm\cdot \cO}$. 
\end{theorem}

The proof uses the algebraic criterion (Theorem \ref{th:DS}) and information about $L_{d/h}(\triv)$ proved in \cite{BEG}.

\subsection{Complete solutions for classical groups} For $G$ of classical type, we solve the isoclinic Deligne-Simpson problem for any slope $\nu$ completely:

\begin{theorem}[See Theorem \ref{th:cl}]\label{th:intro cl} Let $G$ be an almost simple classical group. For any slope $\nu$ with regular number $m$ as denominator,   there is a nilpotent orbit $\cO_{\nu}$ of $\frg$, explicitly determined in all cases,  such that for any adjoint orbit $\cO$, $DS(\nu,\cO)$ has an affirmative answer if and only if $\cO_{\nu}\subset \ov{\Gm\cdot \cO}$. 
\end{theorem}

The proof uses the geometric criterion (Theorem \ref{th:DS Gr}) and the {\em skeleton} of affine Grassmannian introduced by one of us in \cite{Yun-KL}.

In Theorem \ref{th:hq main} we make the condition $\cO_{\nu}\subset \ov{\Gm\cdot \cO}$ more precise in terms of the generalized eigenspaces of $\cO$, recovering results of \cite{KLMNS} and we partially solve a conjecture in \cite{Sage}. 

\subsection{Other results}
In Section  \ref{s:F4soln}, we give the explicit answer to $DS(\nu,\cO)$ for  more cases of $\nu$ in type $F_{4}$. The solutions follow the same pattern as in Theorem \ref{th:intro h}. For other exceptional types, analagous calculations are doable in principle but we do not carry them out.

In Section \ref{s:rig}, for classical groups, we list those pairs $(\nu,\cO)$ that give rise to (cohomologically) rigid connections; for exceptional groups, we list pairs $(\nu,\cO)$ that would give rigid connections if $DS(\nu,\cO)$ is affirmative.

\subsection*{Acknowledgment}  The moduli spaces of connections that are used in the proof of the main result  came from joint work of Z.Y. with R. Bezrukavnikov, Pablo Boixeda Alvarez and Michael McBreen \cite{BBAMY}. Z.Y. would like to thank his coauthors for inspiring discussions. We also thank P. Etingof for answering our questions on representations of rational Cherednik algebras. K.J. would like to thank T. Richarz for comments that helped improve a previous version of this article. We thank an anonymous referee for various helpful comments. Moreover, we thank Masoud Kamgarpour and Bailey Whitbread for pointing out an error in Table \ref{t:cl index rig} in a previous version.

\section{Moduli spaces of connections and Higgs bundles}\label{s:mod}

\subsection{Isoclinic irregular types}
Let $\CC[\t^{\QQ_{<0}}]$ be the set of finite $\CC$-linear combinations of $\t^{a}$ where $a\in \QQ_{<0}$. Let $\frt[\t^{\QQ_{<0}}]=\frt\ot_{\CC}\CC[\t^{\QQ_{<0}}]$. An element $A=\sum_{a\in \QQ_{<0}}A_{a}\t^{a}\in\frt[\t^{\QQ_{<0}}]$ is called an {\em isoclinic irregular type} of slope $\nu\in \QQ_{>0}$, if
\begin{itemize}
\item The lowest degree term of $A$ is $A_{-\nu}\t^{-\nu}$, with $A_{-\nu}$ regular in $\frt$.
\item There exists $n\in\NN$ and a connection $(\cE,\nb)$ over $D^{\times}_{\infty}$ such that $A\in \frt[\t^{-1/n}]$ and the pullback of $(\cE,\nb)$ to $D^{(n),\times}_{\infty}$ (the punctured formal disk with coordinate $\t^{1/n}$) is $G\lr{\t^{1/n}}$-gauge equivalent to a connection of the form
\begin{equation*}
d+(A+\frg\tl{\t^{1/n}})\frac{d\t}{\t}.
\end{equation*}
\end{itemize}

\begin{lemma}\label{l:A}
Let $A=\sum_{a\in \QQ_{<0}}A_{a}\t^{a}\in \frt[\t^{\QQ_{<0}}]$ with lowest degree $-\nu$ and coefficient $A_{-\nu}$ regular in $\frt$. Then 
\begin{enumerate}
\item $A$ is an isoclinic irregular type of slope $\nu$ if and only if there exists $w\in W$ (necessarily unique) such that $wA_{a}=e^{2\pi ia}A_{a}$ for all $a\in \QQ_{<0}$. 
\item Let $m$ be the minimal positive integer such that $A\in \frt[\t^{-1/m}]$. Then the element $w$ is regular of order $m$. Moreover, $m$ is the denominator of $\nu$ in lowest terms. 
\end{enumerate}
\end{lemma}
\begin{proof}
(1) The uniqueness of $w$ is clear: comparing lowest coefficients we have $wA_{-\nu}=e^{-2\pi i\nu}A_{-\nu}$. Since $A_{-\nu}$ is regular, $w$ is unique if it exists.
 
Suppose $A$ is an isoclinic irregular type of slope $\nu$. By definition, for some $n\in\NN$ such that $A\in \frt[\t^{-1/n}]$, there exists a $G$-connection $\nb^{(n)}$ of the form $d+(A+\frg\tl{\t^{1/n}})d\t/\t$  over $D^{(n),\times}_{\infty}$ that descends to a $G$-connection $\nb$ on $D_{\infty}$. It is standard to see that $\nb^{(n)}$ is gauge equivalent to a $T$-connection, i.e., one of the form $d+(A+\frt\tl{\t^{1/n}})d\t/\t$, which we assume from now (see for example \cite[\S 9]{BV}). Let $\G$ (resp. $\G^{(n)}$) be the differential Galois group of the disk $D^{\times}_{\infty}$ (resp. $D^{(n),\times}_{\infty}$), then $\nb^{(n)}$ corresponds to a homomorphism of pro-algebraic groups $\r: \G^{(n)}\to T$. Since $A_{-\nu}$ is regular in $\frt$, the image of $\r$ is regular (not lying in $\ker(\a)$ for any root $\a$). The descent datum of $\nb^{(n)}$ to $D^{\times}_{\infty}$ gives an extension of $\r$ to  $\wt\r: \G\rtimes\mu_{n}\to G$. Note that $\G/\G^{(n)}=\mu_{n}=\Aut(D^{(n),\times}_{\infty}/D^{\times}_{\infty})$.  Since $\r$ has regular image in $T$, the image of $\wt\r$ must normalize $T$. The generator $\z_{n}=e^{2\pi i/n}\in \mu_{n}$ then gives an element $w\in W$ by taking the image of $\wt\r(\wt\z_{n})$ in $N_{G}(T)/T=W$, where $\wt\z_{n}$ is any lifting of $\z_{n}$ to $\G$. By construction, we have $\r\c\Ad(\wt\z_{n})=w\r: \G^{(n)}\to T$. The homomorphism $\r\c\Ad(\wt\z_{n}): \G^{(n)}\to T$ gives a $T$-connection with irregular part $A(\z_{n}\t^{1/n})d\t/\t$, while $w\r$ gives a $T$-connection with irregular part $wAd\t/\t$. Comparing irregular parts we conclude that $wA_{a}=e^{2\pi ia}A_{a}$ for each coefficient $A_{a}$ of $A$. 

Conversely, if $A\in \frt[\t^{-1/n}]$ satisfies $wA(\t^{1/n})=A(e^{2\pi i/n}\t^{1/n})$, we consider the connection $\nb^{(n)}=d+Ad\t/\t$ on $D^{(n),\times}_{\infty}$. Multiply $n$ by another positive integer if necessary, we may assume there exists a lifting $\dot w\in N_{G}(T)$ of $w$ such that $\dot w^{n}=1$. Then gauge by the constant element $\dot w$ gives an isomorphism of $G$-connections $\z_{n}^{*}\nb^{(n)}\cong \nb^{(n)}$ which gives a descent datum for $\nb^{(n)}$ from $D^{(n),\times}_{\infty}$ to $D^{\times}_{\infty}$. This shows that $A$ is the irregular part of the pullback of some $G$-connection on $D^{\times}_{\infty}$, hence an isoclinic irregular type of slope $\nu$.

(2) Let $d=m\nu\in \NN$. Since $wA_{-\nu}=e^{-2\pi i\nu}A_{-\nu}$, $A_{-\nu}\in\frt$ is a regular eigenvector of $w$, $w$ is regular in the sense of Springer \cite{Spr}. Let $m_{1}$ be the order of $w$. By \cite[Theorem 4.2(i)]{Spr}, $m_{1}$ is the order of $e^{2\pi id/m}$ (as a root of unity), hence $m_{1}|m$. On the other hand, for $1\le j\le d$, if $A_{-j/m}\ne0$, then it is an eigenvector of $w$ with eigenvalue $e^{2\pi ij/m}$, which means that $e^{2\pi ij/m}$ has order divisible by $m_{1}$ whenever $A_{-j/m}\ne0$. This implies $A\in \frt[\t^{-1/m_{1}}]$. Therefore the minimality of $m$ implies $m_{1}=m$. Since $e^{2\pi id/m}$ has order $m_{1}=m$, we conclude that $\gcd(d,m)=1$.
\end{proof}

\subsection{Regular gradings and parahoric subgroups}
Fix a Borel subgroup $B$ containing $T$. Let $\Phi$ be the set of roots of $G$ with respect to $T$, and for $\a\in \Phi$, let $\frg_{\a}$ be the corresponding root space. The affine real roots of the loop Lie algebra $\frg\lr{\t}$ with respect to $T$ are $\a+n$ for $\a\in \Phi$ and $n\in\ZZ$. Let $\bI_{\infty}\subset G\tl{\t}$ be the standard Iwahori subgroup determined by $B$. 

A point $x\in \xcoch(T)_{\QQ}$ determines a {\em Moy-Prasad grading} 
$$\frg[\t,\t^{-1}]=\op_{r\in \QQ}\frg\lr{\t}_{x,r}$$
on $\frg[\t,\t^{-1}]$ defined as follows. For $r\in \QQ$, $\frg\lr{\t}_{x,r}$ is the span of $\frg_{\a}t^{n}$ for those affine roots $\a+n$ such that $\a(x)+n=r$, together with $\frt \t^{r}$ if $r\in \ZZ$. 

Let $\frg\lr{\t}_{x,\ge r}$ be the $\t$-adic completion of $\op_{r'\ge r}\frg\lr{\t}_{x,r'}$. Similarly define $\frg\lr{\t}_{x,>r}$. Then $\frg\lr{\t}_{x,\ge 0}$ is a parahoric subalgebra of $\frg\lr{\t}$, which corresponds to a parahoric subgroup $\bP_{x}\subset G\lr{\t}$; $\frg\lr{\t}_{x,> 0}$ is the Lie algebra of the pro-unipotent radical $\bP_{x}^{+}$ of $\bP_{x}$. The Lie algebra of the Levi quotient $L_{x}=\bP_{x}/\bP^{+}_{x}$ is identified with $\frg\lr{\t}_{x,0}=\frg_{x,\un 0}$.

Evaluating at $\t=1$, we get an embedding $\frg\lr{\t}_{x,r}\incl \frg$ whose image depends only on $\un r=r\mod \ZZ\in \QQ/\ZZ$. We denote this image by $\frg_{x,\un\r}\subset \frg$. We easily see that $\frg_{x,\un\r}$ for various $\un\r\in\QQ/\ZZ$ give a $\QQ/\ZZ$-grading of $\frg$
\begin{equation*}
\frg=\op_{\un r\in \QQ/\ZZ}\frg_{x,\un r}.
\end{equation*}
Let $\frac{1}{m}\ZZ/\ZZ$ (where $m\in\NN$) be the smallest subgroup of $\QQ/\ZZ$ containing those $\un r$ such that $\frg_{x,\un r}\neq 0$. Then $m$ is the smallest positive integer such that $mx$ lies in the coweight lattice (i.e., $\a(x)\in \frac{1}{m}\ZZ$ for all $\a\in \Phi$). We call $m$ is {\em order} of $x$.
  
\begin{defn} Let $m\in \NN$. A point $x\in \xcoch(T)_{\QQ}$ is called $m$-regular, if the order of $x$ is $m$, and $\frg_{x,\un{1/m}}$ contains a regular semisimple element of $\frg$.
\end{defn} 

Identify $\frt$ with $\xcoch(T)_{\CC}$ such that the exponential map $\exp_{T}: \frt\cong\xcoch(T)_{\CC}\to T$ has kernel $\xcoch(T)$.

\begin{lemma}\label{l:mreg} Let $x\in \xcoch(T)_{\QQ}$ and let $m$ be the order of $x$. The following are equivalent:
\begin{enumerate}
\item $x$ is $m$-rergular.
\item For some $d$ prime to $m$, $\frg_{x,\un{d/m}}$ contains a regular semisimple element of $\frg$.
\item For all $d$ prime to $m$, $\frg_{x,\un{d/m}}$ contains a regular semisimple element of $\frg$.
\item $\exp_{T}(x)$ is $G$-conjugate to a lifting $\dot w\in N_{G}(T)$ of a regular element $w\in W$ of order $m$. 
\end{enumerate}
\end{lemma}
\begin{proof}
(3)$\Rightarrow$ (1)$\Rightarrow$ (2) are clear.

(2)$\Rightarrow$ (4). Let $g\in \exp_{T}(x)$, then $\frg_{x,\un r}$ is the eigenspace of $\Ad(g)$ with eigenvalue $e^{2\pi i r}$. Let $s\in \frg_{x,\un{d/m}}$ be regular semisimple for some $d$ prime to $m$. Let $H=C_{G}(s)$, a maximal torus of $G$, and $\frh=\Lie H$, a Cartan subalgebra. Since $s$ is an eigenvector of $\Ad(g)$,  $\Ad(g)$ normalizes $H$, hence $g\in N_{G}(H)$. 
The subalgebra $\frh$ inherits a $\frac{1}{m}\ZZ/\ZZ$-grading $\frh=\op_{\un r} \frh_{\un r}$ where $\frh=\frh\cap\frg_{x,\un r}$. Now $\Ad(g)$ has a regular eigenvector on $\frh$ with eigenvalue a primitive $m$th root of unity, hence its image in $W(H,G)$ is a regular element $w$ of order $m$.   

(4) $\Rightarrow$ (3): When (4) holds, there exists a maximal torus $H\subset G$ such that $g:=\exp_{T}(x)\in N_{G}(H)$ and maps to a regular element $w\in W(H,G)$ of order $m$. Since the Galois action on $\frh\cong \xcoch(H)\ot\CC$ permutes eigenspaces of $w$ with eigenvalues of the same order, $w$ has regular eigenvectors with eigenvalue $e^{2\pi id/m}$ for all $d$ prime to $m$, which give regular semisimple elements in $\frg_{x, \un{d/m}}$ for all $d$ prime to $m$.
\end{proof}  
  
\begin{remark} By \eqref{l:mreg}, if $x$ is $m$-regular, then $m$ must be a regular number of $W$. Conversely, if $m$ is a regular number of $W$, then there exists $x\in \xcoch(T)_{\QQ}$ that is $m$-regular. Indeed, by \cite[Theorem 3.2.5]{OY}, one can take $x=\r^{\vee}/m$.
\end{remark}  
 
\begin{remark} When $G$ is semisimple and $m$ is an elliptic regular number of $W$, the $m$-regular points $x\in \xcoch(T)_{\QQ}$ are in the same orbit for the affine Weyl group action, because any two liftings of an elliptic element in $W$ are in the same $G^{\sc}$-orbit. In this case, the Moy-Prasad filtration on $\frg\lr{\t}$ corresponding to a $m$-regular $x$ was first considered by Reeder and Yu in their construction of epipelagic representations \cite{RY}.
 
When $m$ is not elliptic, there may be $m$-regular points $x,x'\in \xcoch(T)_{\QQ}$ in different orbits under the extended affine Weyl group. For example, when $\frg=\so_{2n}$ and $m$ is an odd divisor of $n-1$ (letting $n-1=m\ell$), the isomorphic type of $\frg_{x,\un 0}$ can be either $\gl(2\ell)^{(m-1)/2}\times \so(2\ell+2)$, or $\gl(2\ell+1)\times \gl(2\ell)^{(m-3)/2}\times \so(2\ell)$. 
\end{remark}

For $x\in \xcoch(T)_{\QQ}$ of order $m$, let $\Ad(\t^{x})$ denote the adjoint action of $\l(\t^{1/m})$ in $G\lr{\t^{1/m}}$, where we write $x=\l/m$ for some $\l$ in the coweight lattice. Note that $\Ad(\t^{x})$ restricts to a linear isomorphism
\begin{equation*}
\Ad(\t^{x}): \frg\lr{\t}_{x,r}\isom \frg_{x,\un r}\t^{r}.
\end{equation*}

\begin{defn}\label{d:A adapt} Let $m$ be a regular number of $W$ and let $x\in \xcoch(T)_{\QQ}$ be $m$-regular. Let $A$ be an isoclinic irregular type of slope $\nu=d/m$ (in lowest terms). An element $A'=\sum_{j=1}^{d}A'_{-j/m}\in \op_{j=1}^{d}\frg\lr{\t}_{x, -j/m}$ is said to be {\em adapted to $A$}, if there exists $g\in G$ such that
\begin{equation*}
\Ad(g\t^{x})A'_{-j/m}=A_{-j/m}\t^{-j/m}, \quad j=1,2,\cdots, d.
\end{equation*}
\end{defn}

\begin{lemma}\label{l:A adapt} Let $m$ be a regular number of $W$ and let $x\in \xcoch(T)_{\QQ}$ be $m$-regular. Let $A$ be an isoclinic irregular type of slope $\nu=d/m$ (in lowest terms). Then there exists $A'\in \op_{j=1}^{d}\frg\lr{\t}_{x, -j/m}$ adapted to $A$.
\end{lemma}
\begin{proof}
Choose any $s\in \frg\lr{\t}_{x,1/m}$ such that its image in $\frg_{x,\un{1/m}}$ is regular semisimple. Let $\frh\lr{\t}$ be the centralizer of $s$ in $\frg\lr{\t}$. Then $\frh\lr{\t}$ inherits a Moy-Prasad grading from $\frh\lr{\t}_{x,r}:=\frh\lr{\t}\cap \frg\lr{\t}_{x,r}$ so that $\frh\lr{\t}$ is the $\t$-adic completion of $\op\frh\lr{\t}_{x,r}$. Let $\frh_{x,\un r}\subset \frg_{x,\un r}$ be the image of $\frh\lr{\t}_{x,r}$ under the evaluation at $\t=1$. Then $\frh:=\op_{\un r\in \frac{1}{m}\ZZ/\ZZ} \frh_{x,\un r}$ is a Cartan subalgebra of $\frg$. The map $\Ad(\t^{x})$ maps $\frh\lr{\t}_{x,r}$ isomorphically to $\frh_{x,\un r}\t^{r}\subset \frh \t^{r}$. 

By construction, $\frh_{x,\un r}$ is the eigenspace of $\Ad(\exp_{T}(x))$ on $\frh$ with eigenvalue $e^{2\pi ir}$. The proof of Lemma \ref{l:mreg} shows that $\exp_{T}(x)$ maps to a regular element in $W(H,G)$ of order $m$ (where $H$ is the maximal torus of $G$ with Lie algebra $\frh$). In particular, there exists $g\in G$ such that 
\begin{equation}\label{Adg}
\Ad(g): (H, \exp_{T}(x))\isom (T, w)
\end{equation}
namely $\Ad(g)H=T$ compatibly with their automorphisms given by $\exp_{T}(x)$ and $w$. For $j\in \ZZ/m\ZZ$, let $\frt^{w,[j]}$ be the eigenspace of $w$ on $\frt$ with eigenvalue $e^{2\pi ij/m}$. Then $\Ad(g)$ restricts to a linear isomorphism $\frh_{x,\un{j/m}}\isom\frt^{w, [j]}$ for all $j\in \ZZ/m\ZZ$. Composing with $\Ad(\t^{x})$ we get a linear isomorphism
\begin{equation*}
\Ad(g\t^{x}): \frh\lr{\t}_{x,j/m}\isom \frh_{x,\un{j/m}}\isom \frt^{w, [j]}.
\end{equation*}
Finally take $A'_{-j/m}\in \frh\lr{\t}_{x,j/m}$ to be $\Ad(g\t^{x})^{-1}A_{-j/m}$ for $j=1,\cdots,d$. 
\end{proof}

\begin{lemma}\label{l:irr type A gauge} Let $m$ be a regular number of $W$ and let $x\in \xcoch(T)_{\QQ}$ be $m$-regular. Let $A$ be an isoclinic irregular type of slope $\nu=d/m$ (in lowest terms). Let $A'\in \op_{j=1}^{d}\frg\lr{\t}_{x,-j/m}$ be adapted to $A$.  Then a $G$-connection $(\cE,\nb)$ on $D^{\times}_{\infty}$ is isoclinic of irregular type $A$ if and only if $\nb$ is $G\lr{\t}$-gauge equivalent to a connection of the form
\begin{equation*}
d+(A'+\frg\lr{\t}_{x,\ge0})d\t/\t.
\end{equation*}
\end{lemma}
\begin{proof} By definition, there exists $g\in G$ such that $\Ad(g\t^{x})(A'_{-j/m})=A_{-j/m}$ for $j=1,2,\cdots, d$. Let $\nb_{A'}$ be the $G$-connection on $D^{\times}_{\infty}$ given by $d+A'd\t/\t$. Gauging by $g\t^{x}$ we transform $\nb_{A'}$ to $d+(A-\Ad(g)x)d\t/\t$ (here we view $x\in \xcoch(T)_{\QQ}$ as a $\QQ$ form of $\frt$), which has irregular part $A$. We conclude that $\nb_{A'}$ is isoclinic of irregular type $A$.

Now consider the other implication: if $\nb$ is isoclinic of irregular type $A$, then $\nb$ is $G\lr{\t}$-gauge equivalent to a connection of the form $d+(A'+\frg\lr{\t}_{x,\ge0})d\t/\t$. Let $w$ be the regular element attached to $A$ as in Lemma \ref{l:A}(1).

We first consider the case $w$ is elliptic (i.e. $\frt^{w}$ is the center of $\frg$). Let $G'=G^{\der}\times Z(G)^{\c}$ which is isogenous to $G$. We can view the connection form of $\nb$ as giving a $G'$-connection on $D^{\times}_{\infty}$. It is therefore sufficient to show that $\nb$ is $G'\lr{\t}$-gauge equivalent to a connection of the form $d+(A'+\frg\lr{\t}_{\ge0})d\t/\t$. This allows us to replace $G$ by $G'$, and write $\nb=\nb_{G^{\der}}+\nb_{Z(G)^{\c}}$ according to the decomposition $\frg=\frg^{\der}\op \frz$ (where $\frz=\Lie Z(G)^{\c}$). The condition on $\nb$ implies $\nb_{Z(G)^{\c}}$ has the form $d+(A_{\frz}+\frz\tl{\t})d\t/\t$, where $A_{\frz}\in \frz[\t^{-1}]$ is the projection of $A$ to $\frz[\t^{-1/m}]$, which lies in $\frz[\t^{-1}]$. Note that $A_{\frz}+\frz\tl{\t}\subset \frg\lr{\t}_{x,\ge -\nu}$. We therefore reduce to the same problem for the $G^{\der}$-connection $\nb_{G^{\der}}$.  Henceforth we may assume $G$ is semisimple, so that $\frt^{w}=0$.

We claim that when $G$ is semisimple and $w$ is elliptic, there is up to isomorphism only one isoclinic $G$-connection on $D^{\times}_{\infty}$ of the given irregular type $A$. Indeed, let $\nb$ be such a connection, then for some $n$ divisible by $m$, its pullback $\nb^{(n)}$ to $D^{(n),\times}_{\infty}$ is gauge equivalent to one of the form $d+(A+\frt\tl{\t^{1/n}})d\t/\t$. It is easy to use gauge transformations in $T$ to eliminate the part in $\t^{1/n}\frt\tl{\t^{1/n}}$, hence we may assume $\nb^{(n)}$ is gauge equivalent to $d+(A+A_{0})d\t/\t$ for some $A_{0}\in\frt$. The same argument as in Lemma \ref{l:A}(1) shows that $wA_{0}=A_{0}$, hence $A_{0}\in \frt^{w}=0$. Therefore $\nb^{(n)}$ is isomorphic to $d+Ad\t/\t$ over $D^{(n),\times}_{\infty}$. The descent datum of $\nb^{(n)}$ to $D^{\times}_{\infty}$ is also unique: they correspond to extensions of $\r:\G^{(n)}\to T$ (monodromy representation of $\nb^{(n)}$) to $\wt\r: \G\to G$. Let $\g\in \G$ be a lifting of $\z_{n}\in \mu_{n}=\G/\G^{(n)}$, then $\wt\r(\g)\in wT$. Since $w$ is elliptic, all elements in $wT$ are conjugate under $T$, hence the uniqueness of $\wt\r$ up to $G$-conjugacy. This proves the uniqueness of the descent of $d+Ad\t/\t$ to $D^{\times}_{\infty}$, hence the isomorphism class of $\nb$ is unique.

With the uniqueness statement above, and the fact that $\nb_{A'}$ is isoclinic of irregular type $A$, we see that any isoclinic $G$-connection of irregular type $A$ on $D^{\times}_{\infty}$ is isomorphic to $\nb_{A'}$. This proves the lemma when $w$ is elliptic.

Now consider the general case. Let $M\subset G$ be the centralizer of $\frt^{w}$, so that $w$ is an {\em elliptic} regular element in $W_{M}=W(T,M)$. Since the image of the monodromy representation $\r: \G\to G$ of $\nb$ lies in $Tw^{\ZZ}$ (up to $G$-conjugacy), in particular, it lies in $M$. Therefore we get an $M$-reduction $\nb_{M}$ of $\nb$ that is isoclinic of irregular type $A$ (viewed as an irregular type for $M$). 

Let $\frh\lr{\t}\subset \frg\lr{\t}$ be the centralizer of $A'_{-\nu}$, and let $\frh\subset\frg$ be the centralizer of the image of $A'_{-\nu}$ in $\frg_{x,-\un\nu}$. We freely use the notations from the proof of Lemma \ref{l:A adapt}.  Then $\frh\lr{\t}_{x,0}\cong \frh_{x,\un 0}\subset \frg_{x,\un 0}$ is an abelian subalgebra consisting of semisimple elements. Up to $G_{x,\un 0}=L_{x}$-action, we may assume $\frh_{x,\un 0}$ is contained in $\frt$. Let $M'=C_{G}(\frh_{x,\un 0})$, a Levi subgroup of $G$ containing $T$. Since $\Ad(g)\frh_{x,\un 0}=\frt^{w}$, we have $\Ad(g)M'=M$. Using $\Ad(g^{-1})$ we transport $\nb_{M}$ to a $M'$-connection $\nb_{M'}$ that is isoclinic of irregular type $\Ad(g^{-1})A$. The same point $x$ gives a Moy-Prasad grading of $\fm'\lr{\t}$ and $\frh\lr{\t}\subset \fm'\lr{\t}$ compatible with the gradings. The element $A'$ is adapted to $\Ad(g^{-1})A$ when the ambient group is considered to be $M'$. By the elliptic case that is already proved, $\nb_{M'}$ is $M'\lr{\t}$-gauge equivalent to an $M'$-connection of the form $d+(A'+\fm'\lr{\t}_{x,\ge0})d\t/\t$. Since $\nb$ is isomorphic to the induced $G$-connection of $\nb_{M'}$, $\nb$ is $G\lr{\t}$-gauge equivalent to a $G$-connection of the form $d+(A'+\fm'\lr{\t}_{x,\ge0})d\t/\t$, and we are done since $\fm'\lr{\t}_{x,\ge0}\subset \frg\lr{\t}_{x,\ge0}$.
\end{proof}

\subsection{The moduli space}
From now on fix a regular number $m$ of $W$ and an $m$-regular point $x\in \xcoch(T)_{\QQ}$. We let $\bP_{\infty}=\bP_{x}$ be the parahoric subgroup of $G\lr{\t}$ with Lie algebra $\frg\lr{\t}_{x,\ge0}$, $\bP^{+}_{\infty}$ its pro-unipotent radical, and $L_{\bP}$ its Levi subgroup containing $T$.

{\bf Convention:} Below we shall write $\frg\lr{\t}_{r}, \frg\lr{\t}_{\ge r}$ and $\frg\lr{\t}_{>r}$ for $\frg\lr{\t}_{x,r}, \frg\lr{\t}_{x,\ge r}$ and $\frg\lr{\t}_{x,>r}$.

\begin{remark} Below we will use results from \cite{BBAMY}, in which we study geometric properties of moduli spaces adapted to the level structure $\bP_{\infty}$ for the fixed choice $x=\rho^{\vee}/m$. However, all arguments in {\em loc.cit.} work equally well for any $m$-regular point $x\in \xcoch(T)_{\QQ}$.
\end{remark}
  
Let $A$ be an isoclinic irregular type of slope $\nu=d/m$, where $d\in \NN$ is prime to $m$. Let $A'=\sum_{j=1}^{d}A'_{-j/m}\in \op_{j=1}^{d}\frg\lr{\t}_{-j/m}$ be adapted to $A$. By Lemma \ref{l:irr type A gauge}, the existence of connections with isoclinic irregular type $A$ at $\infty$ is the same as the existence of connections of the form $d+(A'+\frg\lr{\t}_{\ge0})d\t/\t$ at $\infty$. We shall define a moduli space of such connections. 

Let $\psi=A'_{-\nu}\in \frg\lr{\t}_{-\nu}$. Then $\psi$ is regular semisimple and is homogeneous of slope $-\nu$ as an element in $\frg\lr{\t}$. We refer to \cite[\S3]{OY} for detailed discussion of homogeneous elements in the loop Lie algebra. It will turn out to be convenient to consider all $A'$ with the same lowest term $\psi$. Therefore we fix $\psi\in \frg\lr{\t}_{-\nu}$ that is regular semisimple as an element in $\frg\lr{\t}$ (equivalently, its image in $\frg_{x,-\un\nu}$ is regular semisimple).

In \cite{BBAMY} we introduced the de Rham moduli space $\cM_{\dR, \psi}$. It is the $1$-fiber of a one-parameter family $\l: \cM_{\Hod, \psi}\to \AA^{1}$ whose $0$-fiber is a symplectic moduli space of Higgs bundles. Here we consider a slightly large moduli space that is Poisson but not symplectic, which suffices for our purposes.

Let $\cN_{\Hod, \psi}$ be the moduli stack of triples $(\l, \cE, \nb)$ where
		\begin{itemize}
			\item $\l\in \AA^{1}$.
			\item $\cE$ is a $G$-bundle over $X$ with $\bP^{+}_{\infty}$-level structure at $\infty$. 
			\item $\nb$ is a $\l$-connection on $\cE|_{X\bs \{0,\infty\}}$ satisfying the following conditions: 
			\begin{enumerate}
				\item[(i)] Under any (equivalently, some) trivialization of $\cE|_{D_{\infty}}$ together with its $\bP^{+}_{\infty}$-level structure, $\nb|_{D_{\infty}^{\times}}$ takes the form 
				\begin{equation*}
					\nb|_{D_{\infty}^{\times}}\in \l d+(\psi+\frg\lr{\t}_{>-\nu})d\t/\t.
				\end{equation*}
				\item[(ii)] $(\cE,\nb)|_{D_{0}^{\times}}$ has at most a regular singularity (i.e. at most a simple pole at $0$). 
							\end{enumerate}
		\end{itemize}

Let $\l: \cN_{\Hod,\psi}\to \AA^{1}$ be the function given by the $\l$-coordinate of the data. Let
\begin{equation*}
\cN_{\psi}=\l^{-1}(0), \quad \cN_{\dR, \psi}=\l^{-1}(1).
\end{equation*}
Then $\cN_{\psi}$ is a moduli space of Higgs bundles. We use $(\cE,\ph)$ rather than $(\cE,\nb)$ to denote a point in $\cN_{\psi}$. Let $(f_{1},\cdots, f_{r})$ be homogeneous generators of the polynomial ring $\CC[\frg]^{G}$. Define the Hitchin base
\begin{equation*}
\cB_{\psi}\subset\prod_{i=1}^{r} \cohog{0}{X, \cO([d_{i}\nu]\cdot \infty)}
\end{equation*}
consisting of sections $(a_{1},\cdots, a_{r})$ such that the Laurent expansion of $a_{i}$ at $\infty$ takes the form $f_{i}(\psi)+\t^{-[\nu d_{i}]+1}\CC\tl{\t}$. Then  we have the Hitchin map
\begin{equation*}
f: \cN_{\psi}\to \cB_{\psi}
\end{equation*}
sending a Higgs bundle $(\cE,\ph)$ to $(f_{i}(\ph))_{1\le i\le r}$. Note $f_{i}(\ph)$ is a meromorphic section of $\om_{X}$ with pole order $1$ at $0$ and order $[d_{i}\nu]+1$
at $\infty$. Using $dt$ as a canonical trivialization of $\om(0+\infty)$, we view $f_{i}(\ph)$ as a section of $\cO([d_{i}\nu]\cdot\infty)$.

\subsection{The $\Gm$-action} There is a $\Gm$-action on $\cN_{\Hod,\psi}$ with weight $d$ on the function $\l$.

When restricted to the Higgs bundle moduli space $\cN_{\psi}$, the Hitchin map $f$ is $\Gm$-equivariant with respect to the action on $\cB_{\psi}$ given by $s\cdot a_{i}(\t)=s^{d_{i}d}a_{i}(s^{m}\t)$. The action on $\cB_{\psi}$ contracts to the point $a_{\psi}:=(f_{i}(\psi))_{1\le i\le r}\in \cB_{\psi}$. Let $\cN_{a_{\psi}}=f^{-1}(a_{\psi})$ be the corresponding Hitchin fiber.

\begin{prop}\label{p:Gm} The $\Gm$-action on $\cN_{\Hod, \psi}$ contracts every point to a point in $\cN_{a_{\psi}}$.
\end{prop}
\begin{proof}
Same argument as in \cite[Lemma 3.1.3]{BBAMY}.
\end{proof}

Recall the affine Springer fiber of $\psi$ in \eqref{ASF}.

\begin{prop}\label{p:central fiber} There is a canonical homeomorphism $\Gr_{\psi}\to \cN_{a_{\psi}}$.
\end{prop}
\begin{proof}
Same proof as \cite[Theorem 2.8.1(4)]{BBAMY}.
\end{proof}

\subsection{Residue maps} 
We have a residue map 
\begin{equation*}
\r_{0}: \cN_{\Hod,\psi}\to [\frg/G]
\end{equation*}
taking $(\l,\cE,\nb)$ to the residue of $\nb$ at $0$. This map is $\Gm$-equivariant with respect to the weight $d$ $\Gm$-action on $\frb$.

On the other hand, let $V_{\psi}$ be the affine space $\psi+\op_{j=1}^{d-1}\frg\lr{\t}_{-j/m}$. Let $\bP_{\infty,\ge\nu}\subset \bP^{+}_{\infty}$ be the Moy-Prasad subgroup with Lie algebra $\frg\lr{\t}_{\ge\nu}$, and  $K=\bP^{+}_{\infty}/\bP_{\infty,\ge\nu}$. We have a map
\begin{equation*}
\r_{\infty}: \cN_{\Hod,\psi}\to [V_{\psi}/K]
\end{equation*}
taking $(\l,\cE,\nb)$ to the polar terms of $\nb|_{D^{\times}_{\infty}}$. This action is equivariant with respect to the $\Gm$-action on $V_{\psi}$ and on $K$ such that $s\in \Gm$ acts on $\frg\lr{\t}_{j/m}$ by $s^{d+j}$.

\begin{prop}\label{p:res sm}
The map $\r:=(\l, \r_{0},\r_{\infty}): \cN_{\Hod,\psi}\to \AA^{1}\times[\frg/G]\times[V_{\psi}/K]$ is smooth and equidimensional.
\end{prop}
\begin{proof}
The map $\r$ is  $\Gm$-equivariant. Let $Z\subset \cN_{\Hod,\psi}$ be the non-smooth locus of $\r$, then $Z$ is closed and $\Gm$-stable. By Proposition \ref{p:Gm}, if $Z\ne\vn$, then $Z$ contains a fixed point, hence in particular $Z\cap \cN_{\psi}\ne\vn$. It therefore suffices to show that $\r$ is smooth at every geometric point $(\cE, \ph)\in \cN_{\psi}$. 

The relative tangent complex of $\l$ at $(\cE,\ph)\in \cN_{\psi}$ is $R\G(X, \cK^{\sh}_{(\cE,\ph)})$ where $\cK^{\sh}_{(\cE,\ph)}$ is the two-term complex in degrees $-1$ and $0$:
\begin{equation*}
\Ad(\cE; \frg\lr{\t}_{\ge 1/m})\xr{[-,\ph]}\Ad(\cE; \frg\lr{\t}_{\ge -(d-1)/m})\ot \om_{X}(\{0\}+\{\infty\}).
\end{equation*}
Here, for a lattice $\L\subset \frg\lr{\t}$ that is stable under $\bP^{+}_{\infty}$, we denote by $\Ad(\cE; \L)$ the subsheaf of $j_{*}\Ad(\cE)$ ($j:\AA^{1}\incl \PP^{1}$) consisting of sections that lies in $\L$ near $\infty$ under any trivialization of $\cE$. We can use $dt/t$ to trivialize $\om_{X}(\{0\}+\{\infty\})$.

Similarly, the relative tangent complex of $\r$ at $(\cE,\ph)\in \cN_{\psi}$ is $R\G(X, \cK_{(\cE,\ph)})$ where $\cK_{(\cE,\ph)}$ is the two-term complex in degrees $-1$ and $0$:
\begin{equation*}
\Ad(\cE; \frg\lr{\t}_{\ge d/m})\ot \cO_{X}(-\{0\})\xr{[-,\ph]}\Ad(\cE; \frg\lr{\t}_{\ge 0})\ot \om_{X}(\infty).
\end{equation*}
Here the twisting by $\cO_{X}(-\{0\})$ on both terms correspond to the relative tangent over $[\frg/G]$; changing the lattices at $\infty$ to $\frg\lr{\t}_{\ge d/m}$ and $\frg\lr{\t}_{\ge 0}$ correspond to the relative tangent over $[V_{\psi}/K]$. The Serre dual of $\cK_{(\cE,\ph)}$ is the following two step complex concentrated in degrees $-1$ and $0$
\begin{equation*}
\cK^{\vee}_{(\cE,\ph)}: \Ad(\cE; \frg\lr{\t}_{\ge 1/m})\ot \cO_{X}(-1)\xr{[-,\ph]}\Ad(\cE; \frg\lr{\t}_{\ge -(d-1)/m})\ot \cO_{X}(-1).
\end{equation*}
To show $\r$ is smooth at $(\cE,\ph)$ it suffices to show that the obstruction group  $\cohog{1}{X,\cK_{(\cE,\ph)}}=0$. By Serre duality, it is equivalent to showing $\cohog{-1}{X,\cK^{\vee}_{(\cE,\ph)}}=0$. Note that $\cK_{\cE,\ph}^{\vee}\cong \cK^{\sh}_{(\cE,\ph)}\ot\cO(-1)$, hence it suffices to show $\cohog{-1}{X,\cK^{\sh}_{(\cE,\ph)}}=0$. Equivalently,  it is equivalent to saying that $\Lie\Aut(\cE,\ph)=0$. By \cite[Cor 4.11.3]{NgoFL}, $\Aut(\cE,\ph)$ is isomorphic to a subgroup of a maximal torus of $G$, hence diagonalizable (this is proved in the case without level structure but the argument works with level structure). On the other hand, restricting to $D_{\infty}$, $\Aut(\cE,\ph)$ is a subgroup of the pro-unipotent group $\bP^{+}_{\infty}$, hence itself unipotent. Therefore, $\Aut(\cE,\ph)$ is the trivial algebraic group over $\CC$. This implies that $\r$ is smooth. 

The paragraph above shows that $\cN_{\psi}$ is an algebraic space. We show that $\cN_{\Hod,\psi}$ is also an algebraic space. Indeed, let $Z'\subset \cN_{\Hod,\psi}$ be the locus where the automorphism group is nontrivial. Then $Z'$ is a closed substack stable under $\Gm$. Using Proposition \ref{p:Gm}, if $Z'\ne\vn$, then $Z'$ contains a fixed point, hence in particular $Z'\cap \cN_{\psi}\ne\vn$. However we have shown that $\cN_{\psi}$ is an algebraic space. Therefore $Z'=\vn$, and hence  $\cN_{\Hod, \psi}$ is an algebraic space. 

It remains to show that $\r$ is equidimensional. The relative tangent complex of $\r$ at an arbitrary geometric point $(\l, \cE,\nb)\in \cN_{\Hod,\psi}$ is of the form $R\G(X,\cK_{(\cE,\nb)})$ where $\cK_{(\cE,\nb)}$ is the two step de Rham complex in degrees $-1$ and $0$:
\begin{equation*}
\Ad(\cE; \frg\lr{\t}_{\ge d/m})\ot \cO_{X}(-\{0\})\xr{\nb_{\Ad}}\Ad(\cE; \frg\lr{\t}_{\ge 0})\ot \om_{X}(\infty).
\end{equation*}
Here $\nb_{\Ad}$ is the connection on the adjoint bundle $\Ad(\cE)$ induced from $\nb$. By the smoothness of $\r$ and the vanishing of automorphism groups, $R\G(X,\cK_{(\cE,\nb)})$ is concentrated in degree $0$. Therefore the relative dimension of $\r$ at $(\l,\cE,\nb)$ is
\begin{equation*}
-\chi(\Ad(\cE; \frg\lr{\t}_{\ge d/m})\ot \cO_{X}(-1))+\chi(\Ad(\cE; \frg\lr{\t}_{\ge 0})\ot \cO_{X}(-1))=\dim(\frg\lr{\t}_{\ge 0}/\frg\lr{\t}_{\ge d/m})
\end{equation*}
which does not depend on $(\l,\cE,\nb)$. This shows that $\r$ is equidimensional.
 \end{proof}

The following fact will be used in the proof of the main results.
\begin{lemma}\label{l:good quotient Vpsi}
Let $\frq_{\psi}:=V_{\psi}\sslash K$ and let $\ov f: V_{\psi}\to \frq_{\psi}$ be the natural quotient map. Then for any $\ov a\in \frq_{\psi}(\CC)$, the preimage $\ov f^{-1}(\ov a)$ is a single $K$-orbit. 
\end{lemma}
\begin{proof}
The argument is given in the proof of \cite[Proposition 2.13.1]{BBAMY}, see the paragraphs after equation (2.13) in {\em loc.cit}.
\end{proof}

\section{Proof of the main results}\label{s:proof}

For two nilpotent orbits $\cO_{1}$ and $\cO_{2}$ of $\frg$, we write $\cO_{1}\cle\cO_{2}$ (equivalently $\cO_{2}\cge\cO_{1}$) to mean $\cO_{1}\subset \ov\cO_{2}$.

\subsection{$\cO^{\nil}$ as a limit}
In \S\ref{ss:Onil} we defined a nilpotent orbit $\cO^{\nil}$ out of any adjoint orbit $\cO$ using Lusztig-Spaltenstein induction. There is another characterization of $\cO^{\nil}$ that we use directly in the proof which we spell out here.

Let $\Cone(\cO)\subset \frg$ be the closure of $\Gm\cdot \cO$ (action by scaling). Consider the intersection $\Cone(\cO)\cap \cN$. This is the same as the algebraic asymptotic cone $\mathrm{Cone}_{\alg}(\cO)$ as defined in \cite[Definition 3.6]{AV}.

\begin{lemma}\label{l:limit nil} Let $\cN\subset \frg$ be the nilpotent cone.  Then $\Cone(\cO)\cap \cN=\ov{\cO^{\nil}}$.
\end{lemma}
\begin{proof}
Let $P$ be a parabolic subgroup containing $L$ as a Levi subgroup, with unipotent radical $N_{P}$ and nilpotent radical $\frn_{P}$. 

We first show that $x+\frn_{P}\subset \Ad(N_{P})x$. Indeed, let $B_{L}\subset L$ be a Borel subgroup with maximal torus $T$, and let $B=B_{L}N_{P}$ be the Borel subgroup in $P$.  We then have positive roots in $\Phi^{+}$ defined using $(T,B)$. Let $\frn_{P}(i)$ be the direct sum of root spaces of $\frn_{P}$ of height $i$, and $\frn_{P}(\ge i)=\op_{i'\ge i}\frn_{P}(i')$. Consider the (non-linear) map $\ph: \frn_{P}\to \frn_{P}$ given by $y\mapsto \Ad(e^{y})x-x$. We claim $\ph$ is an isomorphism. Indeed, $\ph$ restricts to $\ph_{\ge i}: \frn_{P}(\ge i)\to \frn_{P}(\ge i)$, and the induced map $\ph_{i}: \frn_{P}(i)\to \frn_{P}(i)$ is $y\mapsto [y,x_{s}]$, which is a linear isomorphism because $\a(x_{s})\ne0$ for any root $\a$ that appears in $\frn_{P}$. Given $z=z_{1}+z_{2}+\cdots\in \frn_{P}$ where $z_{i}\in \frn_{P}(i)$, we can solve for $y_{1}, y_{2},\cdots$ (where $y_{i}\in \frn_{P}(i)$) inductively such that $\ph(y_{1}+y_{2}+\cdots)=z_{1}+z_{2}+\cdots$ using that $\ph_{i}$ is an isomorphism.

The above shows that $\Ad(N_{P})x=x+\frn_{P}$, hence 
\begin{equation}\label{Porb}
\Ad(P)x=\Ad(N_{P})(\Ad(L)x)=x_{s}+\cO^{L}_{x_{n}}+\frn_{P}.
\end{equation}
Let $\wt\frg_{P}=\{(gP, y)\in G/P\times \frg|\Ad(g^{-1})y\in \frp\}$. Let $\pi_{P}:\wt\frg_{P}\to \frg$ be the second projection, which is proper. Let $\wt\cO_{P}\subset \wt\frg_{P}$ be the locally closed subscheme consisting of $(gP,y)$ such that $\Ad(g^{-1})y\in x_{s}+\cO^{L}_{x_{n}}+\frn_{P}$. Then \eqref{Porb} implies $\pi_{P}(\wt\cO_{P})= \cO$. 

Consider the $\Gm$-action on $\wt\frg_{P}$ that scales $y$. Let $D=\{(gP,y)\in  G/P\times \frg|\Ad(g^{-1})y\in \CC\cdot x_{s}+\ov\cO^{L}_{x_{n}}+\frn_{P}\}$. Then  $D\subset \wt\frg_{P}$ is the closure of the $\Gm$-orbit of $\wt\cO_{P}$. Now $\pi_{P}$ is proper, $\Gm$-equivariant and maps $\wt\cO_{P}$ onto $\cO$, we conclude that $\pi_{P}(D)=\Cone(\cO)$. In particular, $\Cone(\cO)\cap \cN=\pi_{P}(D_{0})$ where $D_{0}=\{(gP,y)\in  G/P\times \frg|\Ad(g^{-1})y\in \ov\cO^{L}_{x_{n}}+\frn_{P}\}$ which is irreducible. By the definition of Lusztig-Spaltenstein induction, $\cO^{\nil}$ is the open dense orbit in $\pi_{P}(D_{0})$, hence $\Cone(\cO)\cap \cN=\pi_{P}(D_{0})=\ov{\cO^{\nil}}$.
\end{proof}

\subsection{Proof of Theorem \ref{th:DS Gr}}
Recall that $\frq_{\psi}=V_{\psi}\sslash K$. Let $\r^{\flat}_{\infty}$ be the composition $\cN_{\Hod, \psi}\xr{\r_\infty}[V_{\psi}/K]\xr{\ov f} \frq_{\psi}$. Let 
\begin{equation*}
\r^{\flat}=(\l, \r^{\flat}_{\infty}, \r_{0}): \cN_{\Hod, \psi}\to  \AA^{1}_{\l}\times \frq_{\psi}\times [\frg/G].
\end{equation*}
For $A'\in V_{\psi}$, let $[A']\in \frq_{\psi}$ be its image. Let $\cN_{\dR, [A'], \cO}$ be the preimage of $\{1\}\times \{[A']\}\times \cO$ under $\r^{\flat}$. Similarly define $\cN_{\dR, [A'], \ov\cO}$. We need to show $\cN_{\dR, [A'], \cO}\ne\vn$. 

The $\Gm$-action on $V_{\psi}$ induces an action on $\frq_{\psi}$ which is contracting to $[\psi]$. 

On the other hand, let $\cN_{[\psi]}$ be the preimage of $\{0\}\times\{[\psi]\}\times[\frg/G]$ under $\r^{\flat}$. Then $\cN_{[\psi]}$ is closely related to what is denoted by $\cM_{\psi}$ in \cite{BBAMY}: the only difference is that $\cN_{[\psi]}$ does not have Iwahori level structure at $0$. Let $\cN_{[\psi],\cO^{\nil}}$ be the preimage of $\{0\}\times\{[\psi]\}\times\cO^{\nil}$ under $\r^{\flat}$. Similarly define a closed subspace $\cN_{[\psi],\ov{\cO^{\nil}}}\subset \cN_{[\psi]}$.

Finally, we have an evaluation map 
\begin{equation}\label{ev}
\ev_{\psi}: \Gr_{\psi}\to [\cN/G]
\end{equation}
that sends $g G\tl{t}\in \Gr_{\psi}$ to the image of $\Ad(g^{-1})\psi$ under the projection $\frg\tl{t}\to \frg$ (up to the adjoint action of $G$). For a union of nilpotent orbits $U\subset \cN$, let $\Gr_{\psi, U}\subset \Gr_{\psi}$ be the preimage of $[U/G]$ under $\ev_{\psi}$. In particular, $\Gr_{\psi,\cO^{\nil}}$ and $\Gr_{\psi,\ov{\cO^{\nil}}}$ are defined.

Recall from \cite{Yun-KL} that the {\em reduction type} $\RT(\psi)$ of $\psi$ is the set of nilpotent orbits that are in the image of the map $\ev_{\psi}$. The {\em minimal reduction type} $\RT_{\min}(\psi)$ of $\psi$ are the minimal elements of $\RT(\psi)$ under the closure partial order of nilpotent orbits.

The following lemma implies Theorem \ref{th:DS Gr}.
\begin{lemma}\label{l:N eq Gr} The following are equivalent:
\begin{enumerate}
\item $\cN_{\dR, [A'], \cO}\ne\vn$;
\item $\cN_{\dR, [A'], \ov\cO}\ne\vn$;
\item $\cN_{[\psi],\cO^{\nil}}\ne\vn$;
\item $\cN_{[\psi], \ov{\cO^{\nil}}}\ne\vn$;
\item $\Gr_{\psi, \ov{\cO^{\nil}}}\ne \vn$;
\item $\cO^{\nil}\cge\cO$ for some $\cO\in \RT_{\min}(\psi)$.
\end{enumerate}
\end{lemma}
\begin{proof}
We construct a one-parameter  family $\cN_{\Hod,C}$ degenerating $\cN_{\dR, [A'], \ov\cO}$ to $\cN_{[\psi], \ov{\cO^{\nil}}}$.

Let $C\subset \AA^{1}_{\l}\times \frq_{\psi}\times \frg$ be the closure of the $\Gm$-orbit of $\{1\}\times\{[A']\}\times \cO$. This is a $\Gm\times G$-stable (where $\Gm$ acts on each factor and $G$ only acts on $\frg$) closed subset of $\AA^{1}_{\l}\times \frq_{\psi}\times \frg$. Let $\cN_{\Hod, C}$ be the preimage of $[C/G]\subset \AA^{1}_{\l}\times \frq_{\psi}\times [\frg/G]$ under $\r^{\flat}$. 

Let $C_{\l}$ be the fiber of $C$ over a point $\l\in \AA^{1}_{\l}$. We have $C_{1}=\{1\}\times \{[A']\}\times \ov \cO$. Since  the action of $\Gm$ on $\frq_{\psi}$ is contracting to $[\psi]$, and by Lemma \ref{l:limit nil}, $C_{0}=\{0\}\times \{[\psi]\}\times \ov{\cO^{\nil}}$.


Let $\l_{C}: \cN_{\Hod, C}\to \AA^{1}_{\l}$ be the $\l$-coordinate. Then from the above descriptions of $C_{1}$ and $C_{0}$ we have
$$\l_{C}^{-1}(1)=\cN_{\dR, [A'], \ov \cO}, \quad  \l_{C}^{-1}(0)=\cN_{[\psi], \ov {\cO^{\nil}}}.$$ 

Let $\cO_{A'}\subset V_{\psi}$ (resp. $\cO_{\psi}\subset V_{\psi}$) be the $K$-orbit of $A'$ (resp. $\psi$). By Lemma \ref{l:good quotient Vpsi}, $\cO_{A'}$ is the preimage of $[A']\in\frq_{\l}$.  Let $\wt C\subset \AA^{1}\times V_{\psi}\times \frg$ be the preimage of $C$. Then $\wt C$ is a closed $\Gm\times K\times G$-stable subset of $\AA^{1}\times V_{\psi}\times \frg$ with fibers $\wt C_{1}=\{1\}\times \cO_{A'}\times \ov\cO$ and $\wt C_{0}=\{0\}\times \cO_{\psi}\times \ov{\cO^{\nil}}$. It is easy to see that $\wt C$ is the closure of $\Gm\cdot \wt C_{1}$ (using that $\wt C_{0}$ contains a dense $K\times G$-orbit that is clearly in the closure of $\Gm\cdot \wt C_{1}$), hence the projection $\wt C\to \AA^{1}_{\l}$ is flat.

(1)$\Rightarrow$ (2) and (3)$\Rightarrow$ (4) are clear. 

(2)$\Rightarrow$ (1): By the smoothness of $\r$ proved in Proposition \ref{p:res sm}, $\cN_{\dR, [A'], \ov\cO}$ is smooth over $[\wt C_{1}/(K\times G)]=[\cO_{A'}/K]\times [\ov\cO/G]$, hence its image is open. Now  $[\cO_{A'}/K]$ is a classifying space, so the image of $\r_{0}: \cN_{\dR, [A'], \ov\cO}\to [\ov\cO/G]$ is open. If the image is non-empty, it must intersect the dense subset $[\cO/G]$, hence $\cN_{\dR, [A'], \cO}\ne\vn$.

(4) $\Rightarrow$ (3) is proved in the similar way as  (2) $\Rightarrow$ (1), using the smoothness of $\cN_{[\psi], \ov{\cO^{\nil}}}$ over $[\wt C_{0}/(K\times G)]=[\cO_{\psi}/K]\times [\ov{\cO^{\nil}}/G]$.

(2) $\Rightarrow$ (4): For a point in $\cN_{\dR, [A'], \ov\cO}$, viewed as the fiber of $\cN_{\Hod,C}$ over $\l=1$, has a limit point under the $\Gm$-action by Proposition \ref{p:Gm}. This limit point must lie in the $\l=0$ fiber of $\cN_{\Hod,C}$, which is $\cN_{[\psi], \ov{\cO^{\nil}}}$.

(4) $\Rightarrow$ (2): By Proposition \ref{p:res sm}, $\cN_{\Hod, C}\to [\wt C/K\times G]$ is smooth hence flat. On the other hand, the projection $\wt C\to \AA^{1}_{\l}$ is flat, therefore $\l_{C}: \cN_{\Hod, C}\to \AA^{1}_{\l}$ is flat, hence it has open image. By assumption the image of $\l_{C}$ contains $0$ and it is $\Gm$-stable because $\l_{C}$ is $\Gm$-equivariant, hence the image of $\l_{C}$ also contains $1$, therefore $\l_{C}^{-1}(1)=\cN_{\dR, [A'],\ov\cO}\ne\vn$.  

(4) $\Rightarrow$ (5) because under the map in Proposition \ref{p:central fiber}, $\Gr_{\psi,\ov{\cO^{\nil}}}$ maps homeomorphically into $\cN_{a_{\psi}}\cap \cN_{[\psi], \ov{\cO^{\nil}}}$.

(5)  $\Rightarrow$ (4): If $\cN_{[\psi], \ov{\cO^{\nil}}}\ne\vn$, then its limit under the $\Gm$ action (which exists by Proposition \ref{p:Gm}) is a $\Gm$-fixed point hence lying in $\cN_{a_{\psi}}\cap \cN_{[\psi], \ov{\cO^{\nil}}}$, which is homeomorphic to $\Gr_{\psi,\ov{\cO^{\nil}}}$  by Proposition \ref{p:central fiber}. Hence $\Gr_{\psi,\ov{\cO^{\nil}}}\ne\vn$. 

(5) $\iff$ (6) by the definition of $\RT_{\min}(\psi)$.
\end{proof}

\subsection{Proof of Corollaries}
Corollary \ref{c:larger O} clearly follows from Theorem \ref{th:DS Gr}. For Corollary \ref{c:sp case}, we need to show that in both cases $\Gr_{\psi,\ov{\cO^{\nil}}}\ne\vn$:

(1) When $\cO$ is regular, $\cO^{\nil}$ is the regular nilpotent orbit. The condition $\Gr_{\psi,\ov{\cO^{\nil}}}=\Gr_{\psi}\ne\vn$ is automatic. 

(2) When $\nu\ge1$, $t^{-1}\psi$ is an integral element in $\frg\lr{t}$, hence there existd $g\in G\lr{t}$ be such that $\Ad(g^{-1})t^{-1}\psi\in \frg\tl{t}$. This implies then $gG\tl{t}\in \Gr_{\psi,\{0\}}$. In particular, $\Gr_{\psi,\ov{\cO^{\nil}}}\ne\vn$.

\qed

\subsection{Proof of Theorem \ref{th:DS general} assuming Theorem \ref{th:DS}}

Let $W'\subset W$ and $M\subset G$ be defined as in \S\ref{ss:non ell}. Let $\psi$ be homogeneous of slope $\nu$ adapted to $A$. We may assume $\psi\in \fm\lr{t}_{\nu}$. 

Suppose $\cO'$ is a nilpotent orbit of $\fm$ such that $DS_{M}(\nu,\cO')$ is affirmative and $\cO'\subset \ov{\cO^{\nil}}$. Then by Theorem \ref{th:DS}, $DS_{M}(A,\cO')$ is affirmative for any isoclinic irregular type $A$  for $M$ with slope $\nu$. By Lemma \ref{l:N eq Gr}, this implies $\Gr_{M,\psi, \ov\cO'}\ne\vn$. The embedding $\Gr_{M,\psi}\subset \Gr_{\psi}$ restricts to an embedding $\Gr_{M,\psi, \ov\cO'}\subset \Gr_{\psi, G\cdot \ov\cO'}$. Therefore $\Gr_{\psi, G\cdot \ov\cO'}\ne\vn$. Since $G\cdot\ov\cO'\subset \ov{\cO^{\nil}}$, we see that $\Gr_{\psi, \ov{\cO^{\nil}}}\ne\vn$, which implies $DS(A, \cO)$ is affirmative by Lemma \ref{l:N eq Gr}.

Conversely, suppose $DS(A, \cO)$ is affirmative, hence  for $\psi$ adapted to $A$, $\Gr_{\psi, \ov{\cO^{\nil}}}\ne\vn$ by Lemma \ref{l:N eq Gr}. Let $A_{M}$ be the neutral component of the center of $M$. Since $\Gr_{\psi, \ov{\cO^{\nil}}}$ is ind-proper and non-empty, the fixed point locus  $(\Gr_{\psi, \ov{\cO^{\nil}}})^{A_{M}}\ne\vn$.  Note that $(\Gr_{\psi})^{A_{M}}=\Gr_{M, \psi}$. Hence 
$$\vn\ne (\Gr_{\psi, \ov{\cO^{\nil}}})^{A_{M}}=\Gr_{M,\psi}\cap \Gr_{\psi, \ov{\cO^{\nil}}}=\Gr_{M,\psi, \fm\cap \ov{\cO^{\nil}}}. $$  
This implies that for some nilpotent orbit $\cO'$ of $\fm$ such that $\cO'\subset \ov{\cO^{\nil}}$, we have $\Gr_{M,\psi,\cO'}\ne\vn$. Apply Lemma \ref{l:N eq Gr} to $M$, we see that $DS_{M}(\nu,\cO')$ is an affirmative.

It remains to prove Theorem \ref{th:DS}, which follows from the following theorem combined with Lemma \ref{l:N eq Gr}.

\begin{theorem} Let $G$ be almost simple. Let $\cO$ be a nilpotent orbit of $\frg$, and let $\nu=d/m \in \QQ$ written in lowest terms with $m$ a regular elliptic number for $G$. For homogeneous $\psi\in \frg\lr{t}$ of slope $\nu$, $\Gr_{\psi, \ov\cO}\ne \vn$ if and only if $[L_{\nu}(\triv):E_{\cO}]\ne0$.
\end{theorem}
\begin{proof} 
Let $V\subset \upH^{*}_{\Gm}(\Fl_{\psi})$ be the part invariant under the action of the centralizer $C_{G\lr{t}}(\psi)$ and the monodromy action as $\psi$ varies in the regular semisimple open subset of $\frg\lr{t}_{\nu}$. This is a free module over  $\upH^{*}_{\Gm}(\pt)=\Qlbar[\e]$ (with $\ep$ in  degree $2$). Let $V_{0}$ and $V_{1}$ be the specialization of $V$ to $\ep=0$ and $\ep=1$.  Now $V$ carries a Springer action of $W$ as $\Qlbar[\ep]$-module automorphisms, hence $V_{0}$ and $V_{1}$ are isomorphic as $W$-modules. Note that $V_{0}\subset \cohog{*}{\Fl_{\psi}}$ is a graded $\Qlbar$-subalgebra. 

By \cite[Theorem 1.2.7.]{OY}, there is a $W$-stable ``perverse filtration'' $P_{\le i}V_{1}$ on $V_{1}$ such that $\Gr^{P}_{*}V_{1}\cong L_{\nu}(\triv)$ as $W$-modules. Therefore
\begin{equation*}
[L_{\nu}(\triv):E_{\cO}]\ne0\iff [V_{1}:E_{\cO}]\ne0\iff [V_{0}:E_{\cO}]\ne0.
\end{equation*}
Therefore we need to show that 
\begin{equation*}
\Gr_{\psi, \ov\cO}\ne \vn \iff [V_{0}:E_{\cO}]\ne0.
\end{equation*}

We have the tautological line bundles $\cL(\xi)$ on $\Fl$ indexed by $\xi\in \xcoch(T)$. The Chern polynomials of these line bundles give a map of graded $\Qlbar$-algebras
\begin{equation*}
\Sym(\frt^{*}[-2])\to V_{0}.
\end{equation*}

First assume $\Gr_{\psi, \ov\cO}\ne \vn$. Let $y\in \Gr_{\psi}$ be a point whose image in $[\cN/G]$ lies in a nilpotent orbit $\cO'\subset \ov \cO$. Let $\pi: \Fl_{\psi}\to \Gr_{\psi}$ be the projection. Then $\pi^{-1}(y)\cong \cB_{e'}$ for $e'\in \cO'$. Let $i_{y}^{*}: V_{0}\subset \cohog{*}{\Fl_{\psi}}\to\cohog{*}{\cB_{e'}}$ be the restriction map. Then we have a commutative diagram
\begin{equation*}
\xymatrix{ \Sym(\frt^{*}[-2]) \ar@{->>}[d]\ar[r] & V\ar[d]^{i_{y}^{*}}\\
\cohog{*}{\cB}\ar[r]^{i^{*}_{e'}} & \cohog{*}{\cB_{e'}}
}
\end{equation*} 
Here, $i^{*}_{e'}$ is the restriction map induced by the inclusion $i_{e'}: \cB_{e'}\incl \cB$, and the left vertical map is given by Chern polynomials of tautological line bundles, which is a surjection.
Let $V_{e'}=\Im(i^{*}_{e'}: \cohog{*}{\cB}\to \cohog{*}{\cB_{e'}})$. The above diagram shows that $V_{e'}$ is a subquotient of $V$ as a $W$-module. Therefore it suffices to show $[V_{e'}: E_{\cO}]\ne0$. This follows from Lemma \ref{l:Spr}(2).

Conversely, if $[V_{0}: E_{\cO}]\ne0$, hence $[\upH^{*}(\Fl_{\psi}): E_{\cO}]\ne0$, we show $\Gr_{\psi, \ov\cO}\ne0$. This is true more generally for any topologically nilpotent element $\g\in \frg\lr{t}$ in place of $\psi$. The proof is given in Lemma \ref{l:GrO}.
\end{proof}

\begin{lemma}\label{l:Spr} Let $\cO,\cO'$ be nilpotent orbits of $\frg$. Let $e'\in \cO'$ and let $\cB_{e'}$ be the corresponding Springer fiber. 
\begin{enumerate}
\item If $E_{\cO}$ appears in $\cohog{*}{\cB_{e'}}$, then $\cO'\subset \ov\cO$.
\item If $\cO'\subset \ov\cO$, then $E_{\cO}$ appears in $V_{e'}:=\Im(i^{*}_{e'}: \cohog{*}{\cB}\to \cohog{*}{\cB_{e'}})$.
\end{enumerate}
\end{lemma}
\begin{proof}
Let $\nu: \wt \cN\to \cN$ be the Springer resolution. Let $\cF=R\nu_{*}\Qlbar$, then $\cF[N]$ is a perverse sheaf ($N=\dim\cN$). Let $\cF=\op_{E\in \Irr(W)}\cF_{E}\ot E$ be the decomposition of $\cF$ under the $W$-action, such that $\cF_{E}$ is a shifted the intermediate extension of a local system $\cL_{E}$ on a nilpotent orbit $\cO_{E}$. Then $\cF_{E_{\cO}}=\IC(\ov\cO)[-N]$.

(1) If $E_{\cO}$ appears in $\cohog{*}{\cB_{e'}}=\cF_{e'}=\op_{E\in \Irr(W)}\cF_{E,e}\ot E$, then $\cF_{E,e}=\IC(\ov\cO)_{e}\ne0$, hence $e'\in \ov\cO$.

(2) Note that $\cohog{*}{\cN,\cF}=\cohog{*}{\wt\cN}$ is isomorphic to $\cohog{*}{\cB}$ via pullback along $\wt\cN\to \cB$. Therefore the restriction map $i^{*}_{e'}: \cohog{*}{\cB}\to \cohog{*}{\cB_{e'}}$ can be identified with map restricting global sections of $\cF$ to its stalk at $e'$:
\begin{equation*}
R_{e'}: \cohog{*}{\cN,\cF}\to \cF_{e'}
\end{equation*}
Decompose into isotypic components under $W$, it is the direct sum $\op_{E\in \Irr(W)}\id_{E}\ot r_{E}$of
\begin{equation*}
R_{E,e'}: \cohog{*}{\cN,\cF_{E}}\to \cF_{E,e'}.
\end{equation*}
We need to show that $R_{E_{\cO},e'}\ne0$ whenver $e'\in \ov\cO$. Let $d_{\cO}=\dim \cB_{e'}=(\codim\cO)/2$. We have a map $c: \const{\ov\cO}\to \IC(\ov\cO)[-\dim \cO]=\cF_{E_{\cO}}[2d_{\cO}]$ that induces an isomorphism on the lowest cohomology in degree $2d_{\cO}$. We have a commutative diagram
\begin{equation*}
\xymatrix{ \cohog{0}{\ov\cO, \Qlbar}\ar[d]^{r_{e'}} \ar[r]^-{\upH^{0}(c)} & \cohog{2d_{\cO}}{\ov\cO, \IC}\ar[d]^{R_{E_{\cO},e'}}\\
\Qlbar \ar[r]^-{\upH^{0}(c_{e'})} & \upH^{0}(\IC(\ov\cO)_{e'})
}
\end{equation*}
where $r_{e'}$ is the restriction of the constant sheaf to the stalk at $e'$. Now $\upH^{0}(c)$ and $r_{e'}$ are isomorphisms (since $e'\in \ov\cO$), $\upH^{0}(c_{e'})$ is injective (again using $e'\in \ov\cO$). Therefore $R_{E_{\cO},e'}\ne0$.
\end{proof}

\begin{lemma}\label{l:GrO} Let $\g\in\frg\lr{t}$ be a topologically nilpotent element. 
Let $\cO$ be a nilpotent orbit of $\frg$.
If $E_{\cO}$ appears as a direct summand of the $W$-module $\upH^{*}(\Fl_{\g})$, then $\Gr_{\g,\ov\cO}\ne\vn$.
\end{lemma}
\begin{proof} According to the image of $\ev_{\g}$, $\Gr_{\g}$ is decomposed into the disjoint of locally closed subschemes $\Gr_{\g,\cO'}$ for various nilpotent orbits $\cO'$. Let $e'\in \cO'$, and $A_{e'}=\pi_{0}(C_{G}(e'))$. Then $\ev_{\g}$ induces a map
\begin{equation*}
\ev_{\g, e'}: \Gr_{\g,\cO'}\to [\cO'/G]\cong [\pt/C_{G}(e')]\to [\pt/A_{e'}].
\end{equation*}
For any representation $\chi$ of $A_{e'}$, viewed as a local system on $\pt/A_{e'}$, let $\cL_{\chi}$ be the local system on $\Gr_{\g,\cO'}$ that is the pullback of $\chi$ by $\ev_{\g, e'}$.

Let
\begin{equation*}
\pi_{\cO'}: \Fl_{\g,\cO'}:=\pi^{-1}(\Gr_{\g,\cO'})\to \Gr_{\g,\cO'}
\end{equation*}
be the projection. This is a fibration with fibers isomorphic to $\cB_{e'}$. The sheaves $R^{i}\pi_{\cO'*}\Qlbar$ are local systems with $W$-actions. Decompose the cohomology of $\cB_{e'}$ under the $A_{e'}$-action
\begin{equation*}
\cohog{i}{\cB_{e'}}=\op_{\chi\in \Irr(A_{e'})}\chi\ot V^{i}_{\cO', \chi}.
\end{equation*}
Then each $V^{i}_{\cO', \chi}$ is a $W$-module. We have a $W_{\bQ_{0}}$-equivariant decomposition
\begin{equation*}
R^{i}\pi_{\cO'*}\Qlbar\cong \op_{\chi\in \Irr(A_{e'})}\cL_{\chi}\ot V^{i}_{\cO',\chi}.
\end{equation*}
The Leray spectral sequence for the fibration $\pi_{\cO'}$ abutting to $\cohog{*}{\Fl_{\g,\cO'}}$ has $E_{2}$ page 
$$E_{2}^{j,i}=\cohog{j}{\Gr_{\g,\cO'}, R^{i}\pi_{\cO'*}\Qlbar}=\op_{\chi\in \Irr(A_{e'})}\cohog{j}{\Gr_{\g,\cO'}, \cL_{\chi}}\ot V^{i}_{\cO',\chi}.$$ 
This implies that any irreducible representations $E\in \Irr(W)$ that appear in $\cohog{*}{\Fl_{\g}}$ has to appear in one of the following $W$-modules
\begin{equation*}
V^{i}_{\cO',\chi} \mbox{ for $\Gr_{\g,\cO'}\ne\vn$, $i\in \ZZ_{\ge0}$ and $\chi\in\Irr(A_{e'})$.}
\end{equation*}
Since $E_{\cO}$ appears in $\cohog{*}{\Fl_{\g}}$, it must appear in $V^{i}_{\cO',\chi}$ for some triple $(\cO',\chi, i)$ above, i.e., $[\cohog{*}{\cB_{e'}}:E_{\cO}]\ne0$ for $e'$ in a nilpotent orbit $\cO'$ such that $\Gr_{\g,\cO'}\ne\vn$. By Lemma \ref{l:Spr}(1), $[\cohog{*}{\cB_{e'}}:E_{\cO}]\ne0$ implies $\cO'\subset\ov\cO$. Hence $\Gr_{\g,\ov\cO}\ne\vn$.
\end{proof}

\section{Complete solutions in the Coxeter cases}\label{s:ex}

Let $G$ be almost simple.  In this section, using the criterion in Theorem \ref{th:intro main}, we solve the Deligne-Simpson problem for isoclinic $G$-connections with slope $\nu=d/h$ (where $h$ is the Coxeter number and $d$ is coprime to $h$) and regular with residue in $\cO$ at $0$. The answer is summarized as follows.

\begin{theorem}\label{th:h main} Let $d\in\NN$ be coprime to the Coxeter number $h$ of $G$, and $\nu=d/h$. Then there is a (necessarily unique) nilpotent orbit $\cO_{\nu}$ in $\frg$ such that $DS(d/h,\cO)$ has an affirmative answer if and only if $\cO_{\nu}\preccurlyeq \cO^{\nil}$ (equivalently, $\cO_{\nu}\subset \ov{\Gm\cdot \cO}$).

For $d>h$, we have $\cO_{\nu}=\{0\}$. For $1\le d\le h-1$, the nilpotent orbit $\cO_{\nu}$ is given by:

\begin{enumerate}
\item For classical types, see Table \ref{t:clCox}. Here, for $m\ge s\in \NN$, $\l^{m,s}$ denotes the partition of $m$ with $s$ parts such that each part is either $\lfloor \frac{m}{s}\rfloor$ or $\lceil \frac{m}{s}\rceil$ (most evenly distributed). In Table \ref{t:clCox}, we give the Jordan types of $\cO_{\nu}$ as a partition, which uniquely determines $\cO_{\nu}$ in these cases. 

\item For exceptional types, see Table \ref{t:excCox}. The nilpotent orbits in exceptional types are given using the Bala-Carter notation.
\end{enumerate}
The column labeled by $\D_{\nu}$ records the index of rigidity, to be defined in \eqref{Dnu}.
\end{theorem}

\begin{table}
\begin{center}
\begin{tabular}[t]{|C|C|C|C|}
\hline
\mbox{type} & \cO_{\nu} 	& \D_{\nu} &  \mbox{notation} \\ 	
\hline

A_{n-1} & \l^{n,d} & (d'-1)(d-d'-1)  & n=dk+d', 0\le d'\le d-1\\
\hline
B_{n} & \l^{2n+1,d} & \begin{cases} \frac{1}{4}(d'-1)(d-d') & k \mbox{ even} \\  \frac{1}{4}d'(d-d'-1) & k \mbox{ odd} \end{cases} & 2n+1=dk+d', 0\le d'\le d-1\\
\hline
C_{n}  & \l^{2n,d} & \begin{cases} \frac{1}{4}d'(d-d'-1) & k \mbox{ even} \\  \frac{1}{4}(d'-1)(d-d') & k \mbox{ odd} \end{cases}  & 
2n=dk+d', 0\le d'\le d-1\\
\hline
D_{n}  & \l^{2n-1,d}\cup\{1\} & \begin{cases} \frac{1}{4}(d'-1)(d-d') & k \mbox{ even} \\  \frac{1}{4}d'(d-d'-1) & k \mbox{ odd} \end{cases}  & 
2n-1=dk+d', 0\le d'\le d-1\\
\hline
 \end{tabular} 
\end{center}
\caption{Coxeter solutions in the classical types}
\label{t:clCox}
\end{table}

\begin{table}
\begin{center}
\begin{tabular}[t]{|C|C|C|C|}
\hline
$type$ 	& d 	& \cO_{\nu} 	& \D_{\nu}\\ 
\hline

E_6		& 1	& E_6	& 0 \\
		& 5 	& A_2+2A_1	& 1	\\
		& 7	& 3A_1	&  2	\\
 		& 11 	& A_1	& 5\\
		\hline
E_7		& 1& E_7 & 0 \\
		& 5	& A_3+A_2+A_1	& 1	\\
		& 7 	& A_2+3A_1		& 0	\\
		& 11	& (3A_1)'			& 4	\\
		& 13	& 2A_1			&  5	\\
		& 17	& A_1			& 10 \\
		\hline
 \end{tabular} 	\quad	
 \begin{tabular}[t]{|C|C|C|C|}
 \hline
 $type$ 	& d 	& \cO_{\nu} 	& \D_{\nu}\\

\hline
E_8		& 1		& E_8 			& 0	\\
		& 7		& A_4+A_2+A_1	& 2	\\	
		& 11		& 2A_2+2A_1		& 4	\\
		& 13		& A_2+3A_1		& 5	\\
		& 17		& 4A_1			& 8	\\
		& 19		& 3A_1			& 8	\\
		& 23		& 2A_1 			& 14 	\\
		& 29		& A_1			&  21 	\\
\hline
G_2 		& 1		& G_2		& 0	\\
		& 5		& A_1		& 1	\\
\hline
F_4		& 1		& F_4		& 0 \\
		& 5		& A_2+\tilde{A}_1 & 1 \\
		& 7		& A_1+\tilde{A}_1 &	2\\
		& 11	        & A_1 	& 4 \\
		\hline
		
 \end{tabular} 
\end{center}
\caption{Coxeter solutions in the exceptional types}
 \label{t:excCox}
\end{table}

To prove the theorem, by Theorem \ref{th:DS}, we need to determine which $W$-representations $E_{\cO}$ appear in $L_{d/h}(\triv)$. We will prove a general result in this direction that works for more general slopes (Corollary \ref{cor:commonsummand}). In the case of $\nu=d/h$, this boils down to listing certain subsets of simple roots called {\em $d$-allowable} (Corollary \ref{c:crit h}). We then carry out the case-by-case calculations to determine these $d$-allowable subsets.

\subsection{Preliminaries on $L_{\nu}(\triv)$}
For a graded $W$-representation $M=\op_{n}M_{n}$, where $\dim M_{n}<\infty$ and $M_{n}=0$ for $n\ll0$, we write
\begin{equation*}
\chi_{M}(w,t)=\sum_{n\in \ZZ}\Tr(w,M_{n})t^{n}\in \CC\lr{t}.
\end{equation*}
Similarly, for a finite dimensional $W$-module $M$, we write $\chi_{M}(w)$ for $\Tr(w,M)$.
 
Let $m$ be a regular number for $G$, and $d$ a positive integer coprime to $m$. To compute $L_{d/m}(\triv)$ as a $W$-module, we recall the following graded character formula proved by Rouquier. 
\begin{prop}[Rouquier {\cite[Proposition 5.14.]{Rouq}}]\label{p:Ld/m from L1/m} Let $\nu=d/m>0$ be in lowest terms. Let $\frh$ be the reflection representation of $W$. 
Then we have
\[\chi_{L_{d/m}(\triv)}(w,t) = \frac{\det_{\frh}(1-wt^d)}{\det_{\frh}(1-wt)}\chi_{L_{1/m}(\triv)}(w,t^{r}).\]
\end{prop}

Recall from \cite[Definition 3.5]{Som} that $d$ is called {\em good} for $G$ (or rather the root system of $G$) if $d$ is prime to the coefficients of the highest root in terms of simple roots; $d$ is called {\em very good} for $G$ if it is good for $G$ and it is prime to the order of $\pi_{1}(G^{\ad})$.  Let $\L$ be the coroot lattice of $G$, and view $\CC[\L/d\L]$ as the $W$-module induced by the permutation action of $W$ on $\L/d\L$. Sommers proved:

\begin{prop}[Sommers {\cite[Proposition 3.9.]{Som}}]\label{p:Sommers char} When $d$ is very good for $G$, 
\begin{equation*}
\chi_{\CC[\L/d\L]}(w)=d^{\dim(\frh^{w})}, \quad \forall w\in W. 
\end{equation*}
\end{prop}

\begin{cor}\label{cor:charform} Assume that $d$ is very good for $G$, then we have
\[\chi_{L_{d/m}(\triv)}(w) = \chi_{\CC[\L/d\L]}(w) \chi_{L_{1/m}(\triv)}(w), \quad \forall w\in W.\]
\end{cor}
\begin{proof} Note that we always have 
\[\lim_{t\to 1}\frac{\det_{\frh}(1-wt^d)}{\det_{\frh}(1-wt)} = d^{\dim(\frh^{w})}.\]
The statement then follows from Proposition \ref{p:Sommers char} and Proposition \ref{p:Ld/m from L1/m}.
\end{proof}

\subsection{The $W$-module $\CC[\L/d\L]$}\label{ss:stab}
Our next task is to understand which irreducible representations of $W$ appear in $\CC[\Lambda/d\Lambda]$. Clearly $\CC[\L/d\L]$ decomposes into the direct sum of $\Ind_{W_{x}}^{W}(\triv)$ where $x$ runs over $W$-orbits on $\L/d\L$ and $W_{x}$ is the stabilizer of $x$ under $W$. In the following we give a recipe for the subgroups $W_{x}$ that can occur. 

We denote by $\D$ the set of simple roots of $G$ (with respect to a chosen $B$ and $T$), and $\Delta^{\aff}=\D\cup \{\a_{0}\}$ the set of affine simple roots for $G\lr{t}$. Let $\L\subset \XX_{*}(T)$ be the coroot lattice of $G$. We view $\a\in \D$ as a linear function on $\L\ot\RR$ and $\a_{0}$ as an affine function on $\L\ot\RR$ defined by $\a_{0}(x)=1-\vartheta(x)$. Let $\{n_{\a}\}_{\a\in \D^{\aff}}$ be coprime positive integers satisfying $\sum_{\a\in \D^{\aff}} n_{\a} \alpha =1$ as functions on $\L\ot\RR$.  Then $n_{\a_{0}}=1$ and for $\a\in \D$, $n_{\a}$ is the coefficient of $\alpha$ in the highest root $\vartheta$. 

For a proper subset $J\subset \D^{\aff}$, let $W^{\sh}_{J}$ be the subgroup of the affine Weyl group $W^{\aff}$ generated by reflections across those $\alpha\in J$. Let $W_{J}$ be the image of the injection $W^{\sh}_{J}\subset W^{\aff}\to W$. Note that in general, $W_J$ may be not be a parabolic subgroup of $W$.

\begin{defn}\label{d:d allowable} Let $d\in\NN$. A proper subset $J\subset \D^{\aff}$ is called {\em $d$-allowable} if there exist positive integers $k_{\a}$ for $\a\in \D^{\aff}\setminus J$ such that $\sum_{\a\in\D^{\aff}\setminus J} k_{\a} n_{\a} = d$. 
\end{defn}

\begin{lemma}\label{l:stab d allowable} Let $d\in\NN$. Then the stabilizers $W_{x}$ of $x\in \L/d\L$ are $W$-conjugate to $W_{J}$  for a $d$-allowable subset $J\subset \D^{\aff}$ (see Definition \ref{d:d allowable}). Conversely, for any  $d$-allowable subset $J\subset \D^{\aff}$, $W_{J}=W_{x}$ for some $x\in \L/d\L$. 

\end{lemma}
\begin{proof} Note that $\L/d\L \cong \frac{1}{d}\L/\L$, therefore $W$-orbits on $\L/d\L$ are in natural bijection with $W^{\aff}$-orbits on $\frac{1}{d}\L$. The set 
\begin{equation*}
D=\{ x \in \frac{1}{d}\L \mid  \a(x)\ge0 \mbox{ for all } \a\in \D^{\aff}\}
\end{equation*}

is a fundamental domain for the $W^{\aff}$-action on $\frac{1}{d}\L$. 
The lemma follows from the following statement: the stabilizers $W^{\aff}_{x}$ of $x\in D$ under $W^{\aff}$ are exactly $W_{J}^{\sh}$ for $d$-allowable $J$.

For $x\in D$, let $J$ be the set of $\a\in \D^{\aff}$ such that $\a(x)=0$. Then $W^{\aff}_{x}=W^{\sh}_{J}$. Let $k_{\a}=\a(x)\cdot d\in \ZZ$, then $k_{\a}$ are positive integers for $\a\notin J$. We have $\sum_{\a\notin J}n_{\a}k_{\a}/d=\sum_{\a\in \D^{\aff}}n_{\a}k_{\a}/d=\sum_{\a\in \D^{\aff}}n_{\a}\a(x)=1$. This shows that $J$ is $d$-allowable.

Conversely, assume $J\subset \D^{\aff}$ is $d$-allowable, so there are positive integers $k_{\a}$ (for $\a\in \D^{\aff}\setminus J$) such that $\sum_{\a\notin J} k_{\a} n_{\a} = d$. Define $k_{\a}=0$ if $\a\in J$. Then there is a unique $x\in D$ with $\a(x)= k_\a/d$ for all $\a\in \D^{\aff}$. We have $W^{\aff}_{x}=W^{\sh}_{J}$.
\end{proof}

Also recall the following result of Sommers.

\begin{prop}[Sommers {\cite[Proposition 4.1.]{Som}}]\label{p:stab par} If $d$ is good (i.e., $d$ is prime to all $n_{\a}$, $\a\in \D^{\aff}$), then the stabilizers $W_{x}$ for $x\in \L/d\L$ are  parabolic subgroups of $W$.
\end{prop}

\begin{cor}\label{cor:commonsummand} Assume $d$ is very good for $G$. Let $\chi\in \Irr(W)$, and assume $\chi_{L_{1/m}(\triv)} = \sum_{i\in I} k_{i} \chi_i$ for some positive integers $\{k_i\}_{i\in I}$, is the decomposition of the $W$-module $L_{1/m}(\triv)$ into irreducible characters. Then $\chi$ occurs in $L_{d/m}(\triv)$ if and only if there is a $d$-allowable subset $J\subset \D$ and an $i\in I$ such that
\begin{equation}\label{ResWJ}
\langle \Res_{W_{J}}^{W} \chi , \Res_{W_{J}}^{W} \chi_{i} \rangle > 0,
\end{equation}
where $\langle - , - \rangle$ denotes the natural inner product for class functions (i.e., when restricted to $W_J$, $\chi$ and $\chi_i$ have a common irreducible summand). 
\end{cor}
\begin{proof} By Corollary \ref{cor:charform} and Lemma \ref{l:stab d allowable}, $L_{d/m}(\triv) $ is a direct sum of $\Ind_{W_{J}}^{W}(\triv)\ot L_{1/m}(\triv)$ for $d$-allowable $J\subset \D^{\aff}$ (possibly with multiplicities more than one). By Proposition \ref{p:stab par}, we only need to consider those $d$-allowable $J$ contained in $\D$. Therefore, $\chi$ appears in $L_{d/m}(\triv)$ if and only if $\j{\chi, \Ind_{W_{J}}^{W}(\triv)\ot\chi_{i}}>0$ for some $i\in I$ and some $d$-allowable $J\subset \D$. By the projection formula and Frobenius reciprocity,

\[
\j{\chi, \Ind_{W_{J}}^{W}(\triv)\ot\chi_{i}}=\langle \chi, \Ind_{W_J}^{W} \Res_{W_J}^{W} \chi_{i} \rangle = \langle \Res_{W_J}^{W} \chi, \Res_{W_J}^{W} \chi_{i} \rangle \]
and the claim follows.
\end{proof}

\begin{remark} \label{rem:minpara} Let $J'\subset J\subset \D$. Then if \eqref{ResWJ} holds for $W_{J}$, it also holds for $W_{J'}$ (wiht the same $\chi$ and $\chi_{i}$).  Therefore, to apply the criterion in Corollary \ref{cor:commonsummand}, it suffices to check for {\em minimal} $d$-allowable $J\subset \D$.

\end{remark}

For $J\subset \D$, let $\cO_{J,\reg}$ be the nilpotent orbit of $\frg$ containing a regular nilpotent element of the Levi group of type $J$.

\begin{lemma}\label{lem:parabolictype} Let $J\subset \D$. Let $P_{J}\subset G$ be the standard parabolic subgroup whose Levi $L_{J}$ has simple roots $ J$. Denote by $\frl_{J}$ the Lie algebra of $L_{J}$ and by $\frn^{J}$ the Lie algebra of the unipotent radical of $P_{J}$. Let $\cO$ be a nilpotent orbit of $\frg$. 
The following are equivalent:
\begin{enumerate}
\item  $E_{\cO}^{W_{J}}\neq 0$. 
\item  $\cO$ contains an element of the form $e_{J,\reg}+e'$ where $e_{J,\reg}$ is a regular nilpotent element in $\frl_{J}$ and $e'\in \frn^{J}$.
\item $\cO_{J,\reg} \preccurlyeq\cO$. 
\end{enumerate}
\end{lemma}
\begin{proof}
We abbreviate $P_{J}, L_{J}, \frp_{J}$ and $\frl_{J}$ by $P,L,\frp$ and $\frl$. Write $\frn_{P}=\frn^{J}$.

We first prove equivalence of the first two statements. 
 
Let $\pi:\wt\cN\to \cN$ and $\pi_{L}: \wt\cN_{L}\to \cN_{L}$ be the Springer resolutions for $G$ and $L$ respectively. Let  
\[ \widetilde{\cN}_{P} = \{(gP,e) \in G/P \times \cN \mid \Ad(g^{-1})e\in \frp \}. \]
We have a diagram in which the square is Cartesian
\begin{equation*}
\xymatrix{ \widetilde{\cN}\ar[d]^{\nu_{P}} \ar[r]& [\widetilde{\cN}_{L} / L ] \ar[d]^{\pi_{L}} \\
\widetilde{\cN}_{P} \ar[d]^{\pi_{P}} \ar[r]^-{\ev}& [\cN_{L} / L ] \\
\cN & 
}
\end{equation*}
Here, $\nu_{P}$ is induced by the projection $G/B \to G/P$, $\pi_{P}$ is the sends  $(gP,e)$ to $e\in \cN$,  and $\ev(gP,e)=\Ad(g^{-1})e \mod \frn_{P}$. By Springer theory, $E_{\cO}$ is the multiplicity space of $\IC(\ov{\cO})$ in the perverse sheaf $R\pi_*\CC[\dim\cN]$. Therefore, since $(R\pi_{*}\CC)^{W_{P}} = R\pi_{P,*}\CC$, we know that $E_{\cO}^{W_{P}} \neq 0$ if and only if $\IC(\ov{\cO})$ appears in $R\pi_{P,*}\CC[\dim\cN]$. The map $\pi_{P}$ is semismall and $\widetilde{\cN}_{P}$ is rationally smooth, hence this is equivalent to $\cO$ being a relevant stratum for $\pi_{P}$ in the sense that $\dim \pi_{P}^{-1}(e) = \frac{1}{2}\codim_{\cN}\cO$ for any $e\in \cO$. 

Choose $e\in \cO$ and let $\cP_{e}=\pi_{P}^{-1}(e), \cB_{e}=\pi^{-1}(e)$. We prove the following claim: $\dim \cP_{e} = \frac{1}{2}\codim\, \cO$ if and only if the image of $\ev:\cP_{e}\to [\cN_{L} / L ]$ intersects the regular locus $\cN_{L}^{\reg}$ in $\cN_{L}$. 

For the only if direction, note that the fiber of $\nu_{P,e}: \cB_{e} \to \cP_{e}$ above $h\in \cP_{e}$ is isomorphic to the Springer fiber $\cB^{L}_{\ev(h)}$ of $L$ at $\ev(h)\in\cN_{L}/L$. Therefore, if $\dim \cP_{e} = \frac{1}{2}\codim_{\cN}\cO = \dim \cB_{e}$,  then for $x$ in a non-empty open subset of $\cP_{e}$ , the fiber $\nu_{P,e}^{-1}(x)$ must be zero-dimensional, which happens only if $\ev(x)$ is regular in $\cN_{L}$.

Let us prove the if direction. If $\Im(\ev)$ meets the regular locus, then $\cP_{e}^{\reg} := \cP_{e}\cap\ev^{-1}(\cN_{L}^{\reg})$ is open in $\cP_{e}$ and non-empty. Now, $\nu_{P,e}^{\reg}: \nu_{P}^{-1}(\cP_{e}^{\reg})\to \cP^{\reg}_{e}$ is an isomorphism (since the fibers are regular Springer fibers in $L$), and $\nu_{P}^{-1}(\cP_{e}^{\reg})$ is open in the Springer fiber $\cB_{e}$. Since Springer fibers are equidimensional, hence 
\begin{equation*}
\dim\cP_{e} \ge \dim\cP_{e}^{\reg} =\dim \cB_{e}.
\end{equation*}
Since $\dim\cP_{e} \le \dim \cB_{e}$ anyway, we conclude that $\dim\cP_{e} = \dim \cB_{e}$. This proves the claim. 

Finally, this finishes the proof of $(1)\iff(2)$, since $\Im(\ev)$ meets the regular locus if and only if there is $g\in G$ such that $\Ad(g^{-1})e \mod \frn_{P} \in \cN_{L}^{\reg}$. 

Now (2) implies (3) is clear. To show (3) implies (1),  by Lemma \ref{l:Spr}, it suffices to show that $\cohog{*}{\cB_{e_{J, \reg}}}\cong \Ind^{W}_{W_{J}}(\triv)$ as $W$-modules. By \cite[Theorem 1.3]{L-ind}, we have
\[\Ind_{W_{J}}^{W} (\triv) = \Ind_{W_{J}}^{W} \cohog{*}{\cB^{L}_{e_{J,\reg}}}\cong \cohog{*}{\cB_{e_{J,\reg}} }.\]
as (ungraded) $W$-modules. This finishes the proof.
\end{proof}

\begin{cor}\label{c:crit h} Let $d\in \NN$ be coprime to $h$ and $\cO$ be a nilpotent orbit of $\frg$. Then $E_{\cO}$ appears in $L_{d/h}(\triv)$ if and only if $\cO_{J,\reg}\preccurlyeq \cO$ for some $d$-allowable subset $J\subset \D$. 
\end{cor}
\begin{proof} If $d$ is coprime to $h$, then $d$ is very good by a case-by-case checking. Since $L_{1/h}(\triv)$ is the trivial representation of $W$ (see \cite[\S 2]{BEG}), Corollary \ref{cor:commonsummand} implies that $E_{\cO}$ appears in $L_{d/h}(\triv)$ if and only if $E_{\cO}^{W_{J}}\ne0$ for some $d$-allowable $J\subset \D$. Then Lemma \ref{lem:parabolictype} implies that $E_{\cO}^{W_{J}}\ne0$ if and only if $\cO_{J,\reg}\preccurlyeq \cO$.
\end{proof}

\subsection{Type $A_{n-1}$} \label{ss:typeA}
Let $\Part(n)$ be the set of partitions of $n$. We identify nilpotent orbits of 
$\sl(n)$ with $\Part(n)$ by taking the sizes of Jordan blocks. For each $d\in\NN$, there is a unique partition $\l^{n,d}$ with at most $d$ parts, such that the parts are distributed as evenly as possible, i.e., each part is either $\lfloor n/d\rfloor$ or $\lceil n/d\rceil$.

To apply Corollary \ref{c:crit h}, we need to list all $d$-allowable $J\subset \D$ and compute the nilpotent orbits $\cO_{J,\reg}$. The subgroups $W_{J}$ are up to conjugacy the stabilizers of $x\in \L/d\L$ under $S_{n}$. For $\gcd(d,n)=1$, this is the same as stabilizers of the permutation action of  $S_{n}$ on $(\ZZ/d\ZZ)^{n}$. Therefore for $d$-allowable subsets $J\subset \D$, $W_{J}$ are exactly those standard parabolic subgroups of $W=S_{n}$ that are isomorphic to $S_{n_{1}}\times\cdots\times S_{n_{\ell}}$ for $\ell\le d$. The corresponding $\cO_{J,\reg}$ are exactly those with at most $d$ Jordan blocks. Among such nilpotent orbits,  the smallest one is $\cO_{\nu}$ that corresponds to $\l^{n,d}$. This proves Theorem \ref{th:h main} in type $A$.

\subsection{Type $B_n$}\label{ss:typeB} Let $h=2n$ be the Coxeter number and $d$ be coprime to $2n$. Nilpotent orbits in $\so(2n+1)$ are in bijection with the subset $\Part_{B}(2n+1)\subset \Part(2n+1)$ consisting of partitions in which even parts occur with even multiplicity. To apply Corollary \ref{c:crit h}, we determine $d$-allowable $J\subset \D$, which is the same thing as determining stabilizers $W_{x}$ for $x\in \L/d\L\cong (\ZZ/d\ZZ)^{n}$ by Lemma \ref{l:stab d allowable} and Proposition \ref{p:stab par}.

Let us describe the stabilizer $W_{x}$ for $x\in (\ZZ / d\ZZ)^n$. For $0\le i\le \frac{d-1}{2}$,  let $n_i$ be the number of coordinates of $x$ equal to $\pm i\in \ZZ/ d\ZZ$. Then the stabilizer of $x$ is 
\[W_{x} = S_{n_1} \times \cdots \times S_{n_{\frac{d-1}{2}}} \times W(B_{n_0}) \]
where $W(B_{n_0})$ is a Weyl group of type $B_{n_0}$ and $S_{n_i}$ denotes a symmetric group. Since $x\in (\ZZ / d\ZZ)^n$, the group $W_x$ has at most $\frac{d+1}{2}$ factors. The parabolic subgroup $W_x$ of $W$ corresponds to a parabolic subgroup $P$ of $\SO(2n+1)$ whose Levi factor is isomorphic to
\[\GL(n_1) \times \cdots\times \GL(n_{\frac{d-1}{2}}) \times \SO(2n_{0}+1). \]
The regular nilpotent orbit in this Levi factor corresponds to the partition 
\[(n_1,n_1, \cdots , n_{\frac{d-1}{2}}, n_{\frac{d-1}{2}}, 2n_0+1 ) \]
of $2n+1$. Such partitions in $\Part_{B}(2n+1)$ are characterized as follows: they have at most $d$ parts, and, except for one odd part, all parts appear an even number of times. Denote the set of such partitions by $\Part_{B, \nu}(2n+1)$.

Let $2n+1=kd+d'$ with $0 \le d' < d$. Since $d$ is odd, the partition
\[\lambda^{2n+1,d}=(\underbrace{k+1,\cdots,k+1}_{d'},\underbrace{k,\cdots, k}_{d-d'}), \]
lies in $\Part_{B,\nu}(2n+1)$.  It therefore corresponds to a nilpotent orbit $\cO_{\nu}$ in $\so(2n+1)$ which is $\cO_{J,\reg}$ for a $d$-allowable $J\subset \D$. Since $\l^{2n+1,d}$ is the minimal partition (under the dominance order) with at most $d$ parts, $\cO_{\nu}$ is the minimal nilpotent orbit among those whose Jordan type is in $\Part_{B, \nu}(2n+1)$, i.e., $\cO_{\nu}$ is minimal among those $\cO_{J,\reg}$ for $d$-allowable $J\subset \D$. This proves Theorem \ref{th:h main} in type $B$.

\subsection{Type $C_n$} The argument is very similar to the type $B$ case. Let $\nu = d/2n$ with $\gcd(d,2n)=1$. Nilpotent orbits of $\sp(2n)$ are in bijection with the subset $\Part_{C}(2n)\subset \Part(2n)$ consisting of partitions in which odd parts appear with even multiplicity.  As in \S\ref{ss:typeB}, for $x\in (\ZZ/d\ZZ)^{n}$ with $W_{x}$ conjugate to $W_{J}$, the Jordan type of $\cO_{J,\reg}$ has at most $d$ parts. Writing $2n=kd+d'$ with $ 0 \le d' < d$ we obtain the partition
\[\lambda^{2n,d}=(\underbrace{k+1,\cdots,k+1}_{d'},\underbrace{k,\cdots,k}_{d-d'}), \]
Again, since $d$ is odd, this automatically satisfies the parity constraint so it defines a nilpotent orbit $\cO_{\nu}$ in $\sp(2n)$. Now, $\lambda^{2n,d}$ is the unique minimal partition of $2n$ with at most $d$ parts. Therefore $\cO_{\nu}$ is the minimal orbit among $\cO_{J,\reg}$ for $d$-allowable $J\subset \D$. This proves Theorem \ref{th:h main} in type $C$.

\subsection{Type $D_n$} \label{ss:typeD}The argument is similar to types $B$ and $C$, with a minor modification. Let $\nu = d/(2n-2)$ with $\gcd(d,2(n-1))=1$. There is a surjection from the set of nilpotent orbits in $\so(2n)$ to the subset $\Part_{D}(2n)\subset \Part(2n)$ consisting of partitions for which even parts occur with even multiplicity. This map is $1$ to $1$ except over very even partitions (having only even parts), in which case there are exactly two nilpotent orbits mapping to the same very even partition.

For $x\in \L/d\L=(\ZZ/d\ZZ)^{n}$ (since $d$ is odd), using the same notation as in \S\ref{ss:typeB}, $W_{x}$ is 
\[W_{x} \cong S_{n_1} \times \cdots\times S_{n_{\frac{d-1}{2}}} \times W(D_{n_0}). \]
The corresponding standard Levi subgroup is isomorphic to
\begin{equation*}
 \GL(n_1)\times \cdots\times \GL(n_{\frac{d-1}{2}}) \times \SO(2n_{0}).
 \end{equation*}
Note that  the regular nilpotent orbit in $\so(2n_0)$ is subregular in $\sl(2n_0)$, hence corresponds to the partition $(2n_0-1,1)$. Therefore, $\cO_{J,\reg}$ for $d$-allowable $J\subset \D$ has Jordan type  (up to reordering and removing zeros)
\[\mu = \begin{cases}  
(n_1,n_1,n_2,n_2,\cdots,n_{\frac{d-1}{2}},n_{\frac{d-1}{2}},2n_0-1,1), & n_0 \ge 1; \\
(n_1,n_1,n_2,n_2,\cdots,n_{\frac{d-1}{2}},n_{\frac{d-1}{2}}), & n_0 = 0.
\end{cases}
\]

Such partitions form a subset $\Part_{D,\nu}(2n)\subset \Part_{D}(2n)$: it consists of partitions $\l\in \Part_{D}(2n)$ satisfying {\em one of the following conditions}
\begin{itemize}
\item Either $\l$ has at most $d-1$ parts and all parts appear with even multiplicities.
\item Or $\l$ has at most $d+1$ parts with the smallest part being $1$, and all parts appear with even multiplicities.
\item Or $\l$ has at most $d+1$ parts with the smallest part being $1$,  and, except for two distinct odd parts, one of them being $1$, all other parts of $\l$ appear with even multiplicities.
\end{itemize}

\begin{lemma}\label{l:partitionD} Let $2n-1 = kd+d'$ with $0 \le d' \le d$. Consider the partition of $2n$
\begin{equation*}
\l_{\nu} = (\underbrace{k+1,\cdots,k+1}_{d'},\underbrace{k,\cdots,k}_{d-d'},1).
\end{equation*}
Then $\l_{\nu}\in \Part_{D,\nu}(2n)$ and it is the unique minimal element in $\Part_{D,\nu}(2n)$ under the dominance order. 
\end{lemma}
\begin{proof}  
First we check $\l_{\nu}\in \Part_{D,\nu}(2n)$.  Note that since $d$ is odd, if $k$ is even (resp. odd), then $d'$ is odd (resp. even), and even parts in $\l^{\nu}$ occur with even multiplicity. 

Let $\mu\in \Part_{D,\nu}(2n)$, and let $\ell(\mu)$ be the number of parts of $\mu$. We first show that we may assume that $\ell(\mu)=d+1$. If $\ell(\mu) < d+1$, then we already need to have $\ell(\mu) \le d-1$. Let $u$ be the largest part of $\mu$. Now $u \ge 2$. If $u$ is even, it occurs with even multiplicity. Then, we obtain a strictly smaller partition $\mu'$ by replacing two instances of $u$ by $(u-1,1)$. If $u$ is odd, we have $u \ge 3$. Now there are two cases. Either, $u$ occurs at least twice, in which case we argue as in the even case, or $u$ occurs only once. In this case, the last part of $\mu$ has to be $1$, and we obtain a strictly smaller partition $\mu'$ by replacing both $u$ and the last part of $\mu$ by $(u-1)/2$, reordering, and adding two parts equal to $1$ at the end. 

We are thus reduced to the case that $\mu$ has $d+1$ parts, hence its smallest part is equal to $1$. Denote by $\tilde{\mu}$ the partition of $2n-1$ obtained by removing the last part. Then minimizing $\mu$ over partitions of $2n$ with  $d+1$ parts whose last part is $1$ is equivalent to minimizing $\tilde{\mu}$ over partitions of $2n-1$ with $d$ parts. Since $\l^{2n-1,d}$ is minimal among partitions of $2n-1$ with $d$ parts, we conclude that $\l_{\nu}=\l^{2n-1,d}\cup\{1\}$ is minimal in $\Part_{D,\nu}(2n)$. This finishes the proof of the lemma.
\end{proof}

Since $\l_{\nu}$ is not very even, it corresponds to a unique nilpotent orbit $\cO_{\nu}$ of $\so(2n)$, which by Lemma \ref{l:partitionD} is the minimal orbit among $\cO_{J,\reg}$ for $d$-allowable $J\subset \D$. This finishes the proof of Theorem \ref{th:h main} in type $D$.

\subsection{Exceptional types}
By Corollary \ref{c:crit h}, we need to identify the nilpotent orbits $\cO_{J,\reg}$ for $d$-allowable $J\subset \D$, and find the minimal ones among them. In Bala-Carter's labelling of nilpotent orbits \cite[\S 13.1.]{Carter}, if an orbit is labelled by a subdiagram (possibly reducible) of the Dynkin diagram of $G$ with vertices $J$,  it means this orbit is $\cO_{J,\reg}$. We list the minimal $d$-allowable $J$'s and find that the corresponding set of $\cO_{J,\reg}$'s has a unique minimal element $\cO_{\nu}$.

The computations to list the relevant parabolic subgroups are straightforward, and we leave them to the reader. Note that in type $E_7$, one encounters the following subtlety. For $E_7$ and $d=5$ one needs to check that out of the orbits $(A_5)'$ and $(A_5)''$ only the first one appears. Similarly, for $d=7$ one needs to rule out $(A_3+A_1)''$ and for $d=11$ one needs to rule out $(3A_1)''$. This can be done by noting that the $(-)'$-versions are contained in the subgroup of type $E_6$, while the $(-)''$-versions are not.

\begin{remark} 
In Tables \ref{t:clCox} and \ref{t:excCox} we list the quantity $\frac{1}{2}(\nu|\Phi|-\dim C(\cO_{\nu}))$, where $\Phi$ is the set of roots of $G$ and $C(\cO_{\nu})$ is the centralizer of any $e\in \cO_{\nu}$. This number measures how far a solution is from being cohomologically rigid by the discussion in \S \ref{ss:coho rig}, in particular Proposition \ref{p:num rig}. The equality $\frac{1}{2}(\nu|\Phi|-\dim C(\cO_{\nu}))=0$ is equivalent to saying that the solutions to the corresponding Deligne-Simpson problem is cohomologically rigid.  Note that $\nu|\Phi|=(d/h)\cdot hr=dr$. The dimension of $C(\cO_{\nu})$ can be found in \cite[\S 13.1.]{Carter}.
\end{remark}

By inspection of the results in all types, we arrive at the following observation.
\begin{prop} For any simple Lie algebra $\frg$, we have $\cO_{1-1/h}=\cO_{\min}$, the minimal nilpotent orbit in $\frg$. In other words, for an adjoint orbit $\cO$ (not necessarily nilpotent),  $DS(\nu,\cO)$ has an affirmative answer if and only if $\cO\ne\{0\}$.

\end{prop}

\section{Complete solutions for classical groups}\label{s:cl}

In this section we use the criterion in Theorem \ref{th:DS Gr} and the idea of ``skeleton'' introduced in \cite{Yun-KL} to give complete and explicit answers to $DS(\nu,\cO)$ when $G$ is a classical group. The main result can be summarized as follows.

\begin{theorem}\label{th:cl} Let $G$ be an almost simple classical group, and $\nu=d/m>0$ with $m$ a regular number for $W$ and $d\ge1$ is prime to $m$. Then there is a (necessarily unique) nilpotent orbit $\cO_{\nu}$ in $\frg$ such that $DS(\nu,\cO)$ has an affirmative answer if and only if $\cO_{\nu}\cle \cO^{\nil}$. 

When $\nu\ge1$ we have $\cO_{\nu}=\{0\}$. When $\nu<1$, the Jordan types of $\cO_{\nu}$ are given by Table \ref{t:completecl}. \footnote{Note in type $D$ none of the Jordan types that appear in Table \ref{t:completecl} is very even, therefore the Jordan types determine the nilpotent orbits $\cO_{\nu}$.}

\begin{table}
\begin{center}
\begin{tabular}[t]{|C|C|}
\hline
\mbox{type} & \cO_{\nu} 	\\ 	
\hline

A_{n-1} & \begin{cases}\l^{n,n\nu}, & m|n;\\
\l^{n-1,(n-1)\nu}\cup\{1\}, & m|n-1.\end{cases}\\
\hline
B_{n} & \begin{cases}\l^{2n+1,2n\nu}, & 2n\nu \mbox{ odd};\\
\l^{2n,2n\nu}\cup\{1\}, & 2n\nu \mbox{ even}, d>1 \mbox{ or $d=1$ and $m$ odd};\\
(m+1, m,\cdots, m, m-1,1), & 2n\nu \mbox{ even}, d=1 \mbox{ and $m$ even}.\end{cases}\\
\hline
C_{n}  & \l^{2n,2n\nu} \\
\hline
D_{n}  & \begin{cases}\l^{2n,2n\nu}, & m|n, d>1 \mbox{ or $d=1$ and $m$ odd};\\
(m+1, m,\cdots, m, m-1), & m|n,  d=1 \mbox{ and $m$ even};\\
\l^{2n-1,(2n-2)\nu}\cup\{1\}, & m|2n-2, (2n-2)\nu \mbox{ odd}; \\
\l^{2n-2, (2n-2)\nu}\cup\{1,1\}, & m|n-1, d>1 \mbox{ or $d=1$ and $m$ odd};\\
(m+1,m,\cdots, m, m-1, 1,1)& m|n-1, d=1 \mbox{ and $m$ even}.\end{cases}\\
\hline
 \end{tabular} 
\end{center}
\caption{Complete solutions in the classical types}
\label{t:completecl}
\end{table}

\end{theorem}

We make the condition $\cO_{\nu} \preccurlyeq \cO^{\nil}$ more explicit at the end of this section in Theorem \ref{th:hq main}, answering a question of D. Sage \cite[Conjecture 5.7.]{Sage}.

\subsection{Strategy of proof}
Let $\psi$ be a homogeneous element of slope $\nu$. Theorem \ref{th:DS Gr} implies that $DS(\nu,\cO)$ is affirmative if and only if  $\cO^{\nil}\cge\cO$ for some $\cO\in \RT_{\min}(\psi)$. Therefore, to prove Theorem \ref{th:cl}, we need to show that $\RT_{\min}(\psi)$ consists of a single orbit $\cO_{\nu}$ specified by Table \ref{t:completecl}. 

When $m$ is not elliptic, then we may assume $\psi$ is elliptic in some proper Levi subgroup $M\subset G$. By \cite[Lemma 3.2]{Yun-KL}, $\RT_{\min}(\psi)\subset \io_{M,G}(\RT^{M}_{\min}(\psi))$, where $\RT^{M}_{\min}(\psi)$ is the minimal reduction type of $\psi$ viewed as an element of $\fm\lr{t}$ (so it consists of nilpotent orbits of $\fm=\Lie M$), and $\io_{M,G}: \cN_{M}/M\to \cN/G$ is the natural map induced by the embedding $M\incl G$.  For classical groups $G$, $M$ is a product of groups of type $A$ and a smaller group of the same type as $G$. This observation allows us to reduce the calculation of $\RT_{\min}(\psi)$ to that of $\RT^{M}_{\min}(\psi)$. Therefore, it suffices to treat the case $m$ is elliptic.

\subsection{Minimal reduction type and skeleton}
We recall some definitions and constructions in \cite{Yun-KL}. Let $F=\CC\lr{t}$ and $\cO_{F}=\CC\tl{t}$. Let $\psi\in \frg(F)$ be a  regular semisimple element. Assume $\psi$ is compact, namely its image in $(\frg\sslash G)(F)$ lies in $(\frg\sslash G)(\cO_{F})$.

Let $\cT$ be the centralizer of $\psi$ in $G\ot\CC\lr{t}$, viewed as a torus over $F$. Let $L^{+}\cT$ be the unique parahoric subgroup of the loop group $L\cT$, see \cite[\S 7]{Yun-KL}. The {\em $\cT$-skeleton} of the affine Grassmannian $\Gr_{G}$ is the (reduced) fixed points
\begin{equation*}
\cX_{G}^{\cT}:=(\Gr^{L^{+}\cT}_{G})^{\red}\subset \Gr_{G}.
\end{equation*}
Since $\psi$ is compact, $\psi\in \Lie L^{+}\cT$, therefore $\cX_{G}^{\cT}\subset \Gr_{\psi}$.

The following simple observation allows us to reduce the calculation of $\RT_{\min}(\psi)$ to points on the skeleton.
\begin{lemma}[{\cite[Lemma 7.3]{Yun-KL}}]
For any $\cO\in \RT_{\min}(\psi)$, $\cX^{\cT}_{G}\cap \Gr_{\psi,\cO}\ne\vn$.
\end{lemma}
In other words, to compute $\RT_{\min}(\psi)$, we only need to compute the image of $\cX^{\cT}_{G}$ under the evaluation map $\ev_{\psi}$, and take the minimal orbits in the image.

Below, for each elliptic regular number $m$ for $W$, we choose a torus $\cT$ of type $[w]$ where $[w]$ is the conjugacy class in $W$ consisting of elliptic elements of order $m$,  choose a homogeneous element $\psi\in \Lie L^{+}\cT$ of slope $\nu=d/m$,  and calculate the minimal nilpotent orbits in the  image of $\ev_{\psi}|_{\cX^{\cT}_{G}}: \cX^{\cT}_{G}\to [\cN/G]$.

\subsection{Type $A_{n}$} The only elliptic regular number $m=n$. This case has been treated in \S\ref{s:ex}. It is also easy to compute $\RT_{\min}(\psi)$ using the skeleton. We work with $G=\SL_{n}$. Let $E=\CC\lr{t^{1/n}}$, viewed as an $F$-vector space of dimension $n$. Let $\Tr:E\to F$ and $\Nm:E^{\times}\to F^{\times}$ be the trace and norm maps. Then $\cT=(\Res_{E/F}\Gm)^{\Nm=1}$ is a maximal elliptic torus in $G=\SL(E)$, and $L^{+}\cT=(\cO_{E})^{\Tr=0}$.  We identify $\Gr_{G}$ with the set of $\cO_{F}$-lattices in $E$ with the same relative dimension as $\cO_{E}$. The action of $L^{+}\cT$ on $\Gr_{G}$ has a unique fixed point, namely the lattice $\cO_{E}$. The element $\psi=t^{d/n}\in (\cO_{E})^{\Tr=0}$ is homogeneous of slope $d/n$. It is easy to see that the action of $\psi=t^{d/n}$ on $\cO_{E}/t\cO_{E}$ has Jordan type $\l^{n,d}$.

\subsection{Type $C_{n}$}\label{ss:sk C}
The elliptic regular numbers are even divisors $m$ of $2n$. We use the notations from \cite[8.3]{Yun-KL}. Let $\ell=2n/m$. Let $E_{1}=E_{2}=\cdots=E_{\ell}=\CC\lr{t^{1/m}}$, and let $F_{1}=\cdots =F_{\ell}=\CC\lr{t^{2/m}}$. We view $F_{i}$ as a subfield of $E_{i}$ with trace map $\Tr_{i}: E_{i}\to F_{i}$ and norm map $\Nm_{i}: E_{i}^{\times}\to F^{\times}_{i}$. Define an $F$-torus $\cT$ to be $\prod_{i=1}^{\ell}(\Res_{E_{i}/F}\Gm)^{\Nm_{i}=1}$ (here we view $\Nm_{i}$ as a homomorphism of $F$-tori $\Res_{E_{i}/F}\Gm\to \Res_{F_{i}/F}\Gm$). It is explained in \cite[8.3]{Yun-KL} that the $2n$-dimensional $F$-vector space $V=\op_{i=1}^{\ell} E_{i}$  is equipped with a symplectic form $\j{-,-}$ (with some choices) such that $\cT$ with its natural action on $V$ (by the multiplication of $E_{i}^{\times}$ on $E_{i}$) is a maximal torus in $G=\Sp(V)$. The symplectic form $\j{-,-}$ is chosen such that $\j{E_{i},E_{j}}=0$ for $i\ne j$, and that $\cO_{E_{i}}$ is a self-dual lattice in $E_{i}$.

Note that 
\begin{equation*}
\Lie L^{+}\cT=\op_{i=1}^{\ell}(\cO_{E_{i}})^{\Tr_{i}=0}.
\end{equation*}
Let $(c_{1},\cdots, c_{\ell})\in \CC^{\times,\ell}$  be such that $c_{1}^{m},\cdots, c_{\ell}^{m}$ are distinct. Let
\begin{equation}\label{B psi}
\psi=(c_{1}t^{d/m}, \cdots, c_{\ell}t^{d/m})\in \Lie L^{+}\cT.
\end{equation}
Then $\psi$ is homogeneous of slope $\nu$ in $\frg(F)$.

It is shown in \cite[8.3]{Yun-KL} that $\cX_{G}^{\cT}$ consists of a single point, i.e., the standard lattice $\L_{0}=\op_{i}\cO_{E_{i}}\subset V$. Then  $\ev_{\psi}(\L_{0})$ is the nipotent element given by the action of $\psi$ on $\L_{0}/t\L_{0}$, or the direct sum of the actions of $c_{i}t^{d/m}$ on $\cO_{E_{i}}/t\cO_{E_{i}}=\CC\tl{t^{1/m}}/t\CC\tl{t^{1/m}}$. It is easy to see that the Jordan type of $c_{i}t^{d/m}$ on $\CC\tl{t^{1/m}}/t\CC\tl{t^{1/m}}$ is $\l^{m,d}$. Therefore the Jordan type of $\psi$ on $\L_{0}/t\L_{0}$ is $\l^{m,d}\cup \l^{m,d}\cup\cdots\cup \l^{m,d}$ ($\ell$ times), which is $\l^{2n, d\ell}=\l^{2n, 2n\nu}$.

\subsection{Type $B_{n}$}
The elliptic regular numbers are even divisors $m$ of $2n$. We use the notations from \cite[9.1]{Yun-KL}, and the notations $E_{i}, F_{i}$ and $\cT$ from \S\ref{ss:sk C}. The $2n+1$-dimensional vector space $V=F\op(\op_{i=1}^{\ell} E_{i})$ is equipped with a quadratic form (with some choices) such that the different summands are orthogonal. Moreover, the quadratic form is chosen to have the following property: the lattice $\cO_{E_{i}}$ is dual to $t^{-1/m}\cO_{E_{i}}$ in $E_{i}$; if $\ell=2n/m$  is even, then $\cO_{F}$ is self-dual in $F$, and if $\ell$ is odd, $\cO_{F}$ is dual to $t^{-1}\cO_{F}$. 

The natural action of $\cT$ on $V$ (trivially on $F$ and by multiplication on $E_{i}$) realizes $\cT$ as a maximal torus of $G=\SO(V)$ of type $[w]$, where $w$ is regular of order $m$.  We use the same $\psi$ as in \S\ref{ss:sk C}.

\sss{$\ell$ even}\label{sss:B ell even} We first consider the case where $\ell$ is even. In this case, $\cX^{\cT}_{G}$ consists of self-dual lattices $\L=\cO_{F}\op \L'$, where $\L'$ is a self-dual lattice in $\op E_{i}$ between two given lattices:
\begin{equation*}
\op_{i=1}^{\ell}\cO_{E_{i}}=\L'_{0}\subset\L'\subset t^{-1/m}\L'_{0}.
\end{equation*}
Let $Q=\L'_{0}/t^{1/m}\L'_{0}$, an $\ell$-dimensional vector space equipped with a non-degenerate quadratic form such that the lines $\cO_{E_{i}}/t^{1/m}\cO_{E_{i}}$ are orthogonal to each other for different $i$. Then $\cX^{\cT}_{G}$ can be identified with the space of Lagrangian subspaces $L\subset Q$ (hence it has two components). Such a Lagrangian $L$ corresponds to the lattice $\L=\cO_{F}\op\L'$ where $\L'=\L'_{0}+t^{-1/m}L$. 

We compute the Jordan type of $\psi$ on $\L'/t\L'=(\L'_{0}+t^{-1/m}L)/(t\L'_{0}+t^{(m-1)/m)}L)$. There is a grading by valuation on $\L'/t\L'$ with pieces
\begin{equation}\label{decomp even}
(\L'/t\L')_{i}=\begin{cases}t^{-1/m}L, & i=-1;\\
t^{i/m}Q, & 0\le i\le m-2;\\
t^{(m-1)/m}(Q/L), & i=m-1.\end{cases}
\end{equation}

Then $\psi$ sends the degree $i$ piece to degree $i+d$ piece. Moreover, the map $\psi_{i}: (\L'/t\L')_{i}\to (\L'/t\L')_{i+d}$ is an isomorphism if $0\le i\le m-2-d$, injective for $i=-1$ and surjective for $m-1-d$. From this we see that the possible lengths of Jordan blocks are between $\lfloor(m-1)/d\rfloor$ and $\lceil(m+1)/d\rceil$.

When $d>1$, then $m/d\notin\ZZ$, which implies $\lceil(m+1)/d\rceil-\lfloor(m-1)/d\rfloor=1$. Therefore, the lengths of Jordan blocks differ at most by $1$. Moreover, there are $\ell d$ Jordan blocks since we can arrange that each Jordan block contains a unique line in $t^{i/m}Q$ for $0\le i\le d-1$. Therefore the Jordan type of $\psi$ on $\L'/t\L'$ is $\l^{2n, \ell d}=\l^{2n,2n\nu}$. The Jordan type of $\psi$ on $\L/t\L$ is then $\l^{2n,2n\nu}\cup\{1\}$. 
We conclude that when $d>1$, $\RT_{\min}(\psi)$ consists of a single nilpotent orbit with Jordan type $\l^{2n,2n\nu}\cup\{1\}$.

When $d=1$. Let $s$ be the rank of $\psi^{m}: (\L'/t\L')_{-1}=t^{-1/m}L\to (\L'/t\L')_{m-1}=t^{(m-1)/m}(Q/L)$. Then the action of $\psi$ on $\L'/t\L'$ has Jordan type
\begin{equation*}
(\underbrace{m+1,\cdots, m+1}_{s}, \underbrace{m,\cdots, m}_{\ell-2s},\underbrace{m-1,\cdots, m-1}_{s}).
\end{equation*}
So the minimal reduction type of $\psi$ corresponds to the minimal value of $s$. Now $c=\psi^{m}/t=\diag (c_{1}^{m},\cdots, c_{\ell}^{m}): Q\to Q$ is regular semisimple and diagonal with respect to an orthogonal basis of $Q$, hence $c$ is self-adjoint with respect to the symmetric bilinear form $(-,-)$ on $Q$ corresponding to its quadratic form. 

\begin{lemma}\label{l:min s}
Let $Q$ be a quadratic space over $\CC$ with even dimension $\ell$. Let $c:Q\to Q$ be regular semisimple and self-adjoint. Let $L\subset Q$ be a Lagrangian and consider the composition $c_{L}:  L\subset Q\xr{c}Q\surj Q/L$. Then the minimal possible rank of $c_{L}$ is $1$.
\end{lemma}
\begin{proof}
First, $c_{L}$ cannot be zero for otherwise $cL\subset L$ which implies $c$ has isotropic eigenvectors.  

Now we construct a Lagrangian $L$ such that $c_{L}$ has rank one.
Let $x\in Q$ be a nonzero vector satisfying the system of $\ell-1$ homogeneous quadratic equations $(x,c^{i}x)=0$ for $0\le i\le \ell-2$. Solutions to such $x$ in $\PP(Q)=\PP^{\ell-1}$ is the intersection of $\ell-1$ quadrics in $\PP^{\ell-1}$ so it is non-empty. Let $L$ be the span of $x,cx,\cdots, c^{\ell/2-1}x$. Then $L$ is isotropic. Also $(x,c^{\ell-1}x)\ne 0$ for otherwise all coordinates of $x$ are forced to be zero by Vandermonde determinant. We claim $\dim L=\ell/2$. Indeed, consider the $\ell/2\times \ell/2$ matrix $A$ with entries $A_{ij}=(c^{i}x, c^{j+\ell/2}x)=(x,c^{i+j+\ell/2}x)$ for $0\le i,j\le \ell/2-1$. We see that $A$ has zero entries above the anti-diagonal, and the anti-diagonal entries are $A_{i,\ell/2-i-1}=(x,c^{\ell-1}x)\ne0$. This implies $x,cx,\cdots, c^{\ell/2-1}x$ are linear independent, hence $L$ is a Lagrangian in $Q$. Since $c(L)=\Span\{cx,\cdots, c^{\ell/2}x\}$, its image in $Q/L$ is one-dimensional, hence $c_{L}: L\to Q/L$ has rank one. 
\end{proof}

By this lemma, the minimal possible $s$ is $1$.  We conclude that when $d=1$, $\RT_{\min}(\psi)$ consists of a single nilpotent orbit with Jordan type $(m+1,m,\cdots, m, m-1,1)$.

\sss{$\ell$ odd} The calculation is similar to the $\ell$ even case so we only give a sketch. In this case, $\cX^{\cT}_{G}$ consists of self-dual lattices $\L$ between two given lattices:
\begin{equation*}
\cO_{F}\op (\op_{i=1}^{\ell}\cO_{E_{i}})=:\L_{0}\subset\L\subset \L_{1}:=t^{-1}\cO_{F}\op t^{-1/m}(\op_{i=1}^{\ell}\cO_{E_{i}}).
\end{equation*}
Let $Q_{0}=\cO_{F}/t\cO_{F}$ and $Q_{i}=\cO_{E_{i}}/t^{1/m}\cO_{E_{i}}$. Let $Q^{0}=\op_{i=1}^{\ell}Q_{i}$. Then $Q:=\L_{1}/\L_{0}=t^{-1}Q_{0}\op t^{-1/m}Q^{0}$ is a $(\ell+1)$-dimensional vector space with a non-degenerate quadratic form such the lines $t^{-1}Q_{0},t^{-1/m}Q_{1},\cdots, t^{-1/m}Q_{\ell}$ are orthogonal to each other. We can identify $\cX^{\cT}_{G}$ with the space of Lagrangian subspaces  $L\subset Q$: each such $L$ corresponds to the lattice $\L=\L_{0}\op  L$. Similar to the $\ell$ even case, $\L/t\L$ has a grading $\L/t\L=\op_{i=-1}^{m-1}(\L/t\L)_{i}$ with
\begin{equation}\label{decomp odd}
(\L/t\L)_{i}=\begin{cases} L, & i=-1;\\
t^{i/m}Q^{0}, & 0\le i\le m-2;\\
t(Q/L), & i=m-1.\end{cases}
\end{equation}
Now $\psi$ sends $(\L/t\L)_{i}$ to $(\L/t\L)_{i+d}$. Note $\psi_{-1}: (\L/t\L)_{-1}=L\to Q^{(d-1)/m}Q^{0}$ is still injective because the only possible kernel is $t^{-1}Q_{0}$, which is not contained in $L$ since $L$ is  isotropic. Similarly, $\psi_{m-1-d}: (\L/t\L)_{m-1-d}=t^{(m-1-d)/m}Q^{0}\to (\L/t\L)_{m-1}=t(Q/L)$ is still surjective. We conclude that the sizes of the  Jordan blocks of $\psi$ on $\L/t\L$ are again between $\lfloor(m-1)/d\rfloor$ and $\lceil(m+1)/d\rceil$. 

When $d>1$, $\lceil(m+1)/d\rceil-\lfloor(m-1)/d\rfloor=1$. We conclude that  $\RT_{\min}(\psi)$ consists of a single nilpotent orbit with Jordan type $\l^{2n+1,2n\nu}$.

When $d=1$, let $s$ be the rank of $\psi^{m}/t: L\to Q/L$. Then $\psi$ on $\L/t\L$ has Jordan type
\begin{equation*}
(\underbrace{m+1,\cdots, m+1}_{s}, \underbrace{m,\cdots ,m}_{\ell+1-2s}, \underbrace{m-1,\cdots, m-1}_{s-1}).
\end{equation*}
The minimal possible $s$ is $1$ by Lemma \ref{l:min s}. Therefore in this case $\RT_{\min}(\psi)$ consists of a single nilpotent orbit with Jordan type $(m+1,m,\cdots, m)=\l^{2n+1, 2n\nu}$.

\subsection{Type $D_{n}$}
There are two cases where $m$ is an elliptic regular number for the Weyl group $W$ of type $D_{n}$. We consider them separately.

\sss{$m|n$ and $m$ is even}\label{sss:D m|n} Let $\ell=2n/m$ which is even. The torus $\cT$ of type $[w]$ (where $w$ is regular of order $m$) has the same description as in the type $B$ case, and we freely use notations from \S\ref{sss:B ell even}. The skeleton $\cX^{\cT}_{G}$ in this case consists of self-dual lattices between two given lattices:
\begin{equation*}
\op_{i=1}^{\ell}\cO_{E_{i}}=\L_{0}\subset \L\subset \L_{1}=t^{-1/m}\L_{0}.
\end{equation*}
Again $\cX^{\cT}_{G}$ can be identified with Lagrangians $L\subset t^{-1/m}Q$, with $Q=\L_{0}/t^{1/m}\L_{0}$, i.e.,  $\L=\L_{0}+t^{-1/m}L$. We have a decomposition $\L/t\L=\op_{i=-1}^{m-1}(\L/t\L)_{i}$ as in \eqref{decomp even}. We take the same element $\psi$ as given in \eqref{B psi}. Then $\psi$ sends $(\L/t\L)_{i}\to (\L/t\L)_{i+d}$. The sizes of Jordan blocks of $\psi$ on $\L/t\L$ are between $\lfloor (m-1)/d\rfloor$ and $\lceil(m+1)/d\rceil$. When $d>1$, $\lceil(m+1)/d\rceil-\lfloor (m-1)/d\rfloor=1$ and we conclude that the Jordan type  of $\psi$ on $\L/t\L$ is $\l^{2n,\ell d}=\l^{2n,2n\nu}$. When $d=1$, let $s$ be the rank of $\psi^{m}/t: L\to Q/L$. Then the Jordan type of $\psi$ on $\L/t\L$ is 
\begin{equation*}
(\underbrace{m+1,\cdots, m+1}_{s}, \underbrace{m,\cdots, m}_{\ell-2s}, \underbrace{m-1,\cdots, m-1}_{s}).
\end{equation*}
The minimal possible $s$ is $1$ by Lemma \ref{l:min s}. Therefore $\RT_{\min}(\psi)$ consists of a single orbit with Jordan type $(m+1,m,\cdots, m, m-1)$.

\sss{$m|2n-2$ and $\frac{2n-2}{m}$ is odd} The case $m=2$ is a special case considered in \S\ref{sss:D m|n}. Below we assume $m>2$.

Let $\ell=(2n-2)/m$ which is odd by assumption. Let $E_{0}=\CC\lr{t^{1/2}}$ with norm map $\Nm_{0}: E_{0}^{\times}\to F^{\times}$. Then $\cT=\prod_{i=0}^{\ell}(\Res_{E_{i}/F}\Gm)^{\Nm_{i}=1}$. Let $V=\op_{i=0}^{\ell}E_{i}$, then under a suitable quadratic form on $V$, $E_{i}$ are orthogonal to each other, and $\L_{0}=\op \cO_{E_{i}}$ has dual lattice $\L_{1}=t^{-1/2}\cO_{E_{0}}\op(\op_{i=1}^{\ell}t^{-1/m}\cO_{E_{i}})$. The torus $\cT$ is realized as a maximal torus of $G=\SO(V)$ by the multiplication action of $E^{\times}_{i}$ on $E_{i}$. In this case $\cX^{\cT}_{G}$ consists of self-dual lattices $\L$ satisfying
\begin{equation*}
\op_{i=0}^{\ell} \cO_{E_{i}}=\L_{0}\subset \L\subset \L_{1}=t^{-1/2}\cO_{E_{0}}\op(\op_{i=1}^{\ell}t^{-1/m}\cO_{E_{i}}).
\end{equation*}
Let $Q_{0}=\cO_{E_{0}}/t^{1/2}\cO_{E_{0}}$, $Q_{i}=\cO_{E_{i}}/t^{1/m}\cO_{E_{i}}$ for $1\le i\le \ell$, and $Q^{0}=\op_{i=1}^{\ell}Q_{i}$. Then $Q:=\L_{1}/\L_{0}=t^{-1/2}Q_{0}\op t^{-1/m}Q^{0}$ carries a quadratic form such that the lines $t^{-1/2}Q_{0}$ and $t^{-1/m}Q_{i}$ are orthogonal to each other. The self-dual lattice $\L\in \cX^{\cT}_{G}$ corresponds to a Lagrangian $L\subset Q$, so that $\L=\L_{0}+L$. Write $\L/t\L=Q_{0}\op (\L/t\L)'$, where $ (\L/t\L)'$ has a grading with pieces
\begin{equation*}
(\L/t\L)'_{i}=\begin{cases}L, & i=-1;\\
t^{i/m}Q^{0}, & 0\le i\le m-2;\\
t(Q/L), & i=m-1.\end{cases}
\end{equation*}
We take $\psi$ to be
\begin{equation*}
\psi=(0,c_{1}t^{d/m},\cdots, c_{\ell}t^{d/m})\in \Lie L^{+}\cT=\op_{i=0}^{\ell}(\cO_{E_{i}})^{\Tr_{i}=0},
\end{equation*}
where $(c_{1}^{m},\cdots, c_{\ell}^{m})$ are distinct and all nonzero. Since the $E_{0}$ component of $\psi$ is zero, both $Q_{0}$ and $(\L/t\L)'$ are stable under $\psi$. It therefore  suffices to compute the Jordan type $\l'$ of $\psi$ on $(\L/t\L)'$ as a partition of $2n-1$, and the Jordan type of $\psi$ on $\L/t\L$ will be $\l'\cup\{1\}$.

Now $\psi$ sends $(\L/t\L)'_{i}$ to $(\L/t\L)'_{i+d}$. Note $\psi_{-1}: L\to t^{(d-1)/m}Q^{0}$ is still injective because the only possible kernel is $t^{-1/2}Q_{0}$, which is not contained in $L$ since $L$ is isotropic. Similarly, $\psi_{m-1-d}: t^{(m-1-d)/m}Q^{0}\to t(Q/L)$ is still surjective. We conclude that the sizes of the  Jordan blocks of $\psi$ on $(\L/t\L)'$ are again between $\lfloor(m-1)/d\rfloor$ and $\lceil(m+1)/d\rceil$.  When $d>1$, $\lceil(m+1)/d\rceil-\lfloor(m-1)/d\rfloor=1$. We conclude that  $\RT_{\min}(\psi)$ consists of a single nilpotent orbit with Jordan type $\l^{2n-1,2n\nu}\cup\{1\}$. When $d=1$, let $s$ be the rank of $\psi^{m}/t: L\to Q/L$. Then $\psi$ on $(\L/t\L)'$ has Jordan type
\begin{equation*}
(\underbrace{m+1,\cdots, m+1}_{s}, \underbrace{m,\cdots, m}_{\ell+1-2s}, \underbrace{m-1,\cdots, m-1}_{s-1}).
\end{equation*}
The minimal possible $s$ is $1$ by Lemma \ref{l:min s}. Therefore in this case $\RT_{\min}(\psi)$ consists of a single nilpotent orbit with Jordan type $(m+1,m,\cdots, m,1)=\l^{2n-1, 2n\nu}\cup\{1\}$.

\subsection{Comparison with \cite{KLMNS} and \cite{Sage}} Our criterion recovers previous results of M. Kulkarni, N. Livesay, J. Matherne, B. Nguyen and D. Sage \cite{KLMNS}. In addition, we partially prove D. Sage's \cite[Conjecture 5.7.]{Sage} in the classical types, and we give answers to all possible slopes whose denominators are not necessarily the Coxeter number. First, we need to set up some notation.

Denote by $\Orb = \Orb(\frg)$ the set of adjoint orbits in $\frg$. The set $\Orb$ is ordered according to orbit closure, and there is a poset decomposition 
\[\Orb(\frg) = \coprod_{q \in \frt\sslash W} \Orb_{q}. \]
 Here $\frt$ is a Cartan subalgebra of $\frg$, and $\Orb_{q}$ denotes the preimage of $q$ under $\frg \to \frg \sslash G \cong \frt \sslash W $. Denote by $DS(\nu,\cO,q)$ the Deligne-Simpson problem where $\cO$ ranges only over $\Orb_{q}$. In classical types, we can identify $q$ with a characteristic polynomial. They are of the form 
 \begin{align*} \begin{cases} q=\prod_{i=1}^{s} (x-a_i)^{m_i}, & \mbox{type }A_{n-1} \\
  q=x^{2m_s+\e} \prod_{i=1}^{s-1} (x^{2}-a_{i}^{2})^{m_i},& \mbox{type }B_n, C_n, D_n.  \end{cases}
 \end{align*}
Here $\e=1$ in type $B_{n}$ and $\e=0$ in types $C_{n}$ and $D_{n}$. The $a_i$ are pairwise distinct in type $A$. In types $B,C$ and $D$ each non-zero eigenvalue appears together with its negative, and we require $a_i \neq \pm a_j$ for all $i,j\in\{1,\cdots, s-1\}$. 
 
Each $q$ determines a Levi subgroup $L_{q}\cong\prod_{i=1}^{s}G_{i}\subset G$ up to conjugacy as 
the centralizer of a semisimple element $x\in \frg$ with characteristic polynomial $q$. When $G$ is of type $A_{n-1}$, all factors $G_{i}$ are of type $A_{m_{i}-1}$, and we call all of them the linear factors of $L_{q}$. When $G$ is of type $B_{n},C_{n}$ or $D_{n}$, the factors $G_{i}$ ($1\le i\le s-1$) are of type $A_{m_{i}-1}$, and the factor $G_{s}$ is of the same type as $G$ and has rank $m_{s}$. We call $G_{1},\cdots, G_{s-1}$ the linear factors of $L_{q}$ in these cases.

Specifying an adjoint orbit $\cO$ in $\Orb_{q}$ is the same as specifying a nilpotent orbit for each factor of $L_q$, up to permutation of the linear factors of the same size;  i.e. $\cO$ is given by a collection of partitions $(\l^{1},\cdots, \l^{s})\in \prod_{i=1}^{s}\Part(m_{i})$ up to possible permutations of all factors in type $A$, and by a collection $(\l^{1},\cdots, \l^{s})\in \prod_{i=1}^{s-1}\Part(m_{i}) \times \Part(2m_{s}+\e)$ up to possible permutations of all but the last factor in types $B,C$ and $D$, with $\e$ as above. 

In order to state the next result, we introduce some notation. Recall for positive integers $n$ and $r$, there is a unique minimal partition $\l^{n,r}$ of $n$ with at most $r$ parts. Write $n=kr+r'$ with $0\le r' <r$. Then $\l^{n,r} = (\underbrace{k+1,...,k+1}_{r'},\underbrace{k,...,k}_{r-r'}),$ spread as evenly as possible. We define 
$$\tilde{\l}^{n,r} = (\underbrace{k+1,...,k+1}_{r'+1},\underbrace{k,...,k}_{r-r'},k-1).$$ 
This partition is the unique smallest partition with $r$ parts which is not $\l^{n,r}$. 

 \begin{theorem} \label{th:hq main}
Let $G$ be an almost simple classical group, and $\nu=d/m>0$ with $m$ a regular number for $W$ and $d\ge1$ is prime to $m$, and let $q\in \frt\sslash W$. We keep the above notation for $q$. Then there are explicit orbits $\{\cO_{\nu}^{q,j}\}_{j\in J} \subset \Orb_{q}$ ($J$ is some finite index set) such that $DS(\nu,\cO,q)$ is affirmative if and only if $\cO_{\nu}^{q,j} \preccurlyeq \cO$ for some $j\in J$. These orbits are given in Table \ref{t:clq}. In the second column we give conditions on $\nu=d/m$ (compare with Theorem \ref{th:cl}). In the final column we give the Jordan types of $\cO_{\nu}^{q,j}\in\Orb_{q}$. If there is no dependence on $j$, the orbit is unique. 
\begin{table}

\begin{tabular}[t]{|C|L|L|}
\hline
\mbox{type} & \mbox{conditions on }d,m &  \cO_{\nu}^{q,j}	\\ 	
\hline

A_{n-1} 	&  m|n 													&  (\l^{m_1,n\nu}, ..., \l^{m_s,n\nu}) \\
		\cline{2-3}
		& m|n-1 													& (\l^{m_1,(n-1)\nu},..., \l^{m_j-1,(n-1)\nu} \cup \{1\} ,..., \l^{m_s,(n-1)\nu}), j=1,...,s \\
\hline
B_{n} 	& m |2n, 2n\nu \mbox{ odd} 									&  (\l^{m_1,2n\nu}, \dots, \l^{2m_s+1,2n\nu})\\
\cline{2-3}
		& m |2n, 2n\nu \mbox{ even} & (\l^{m_1,2n\nu}, \dots, \l^{2m_s,2n\nu}\cup\{1\})  \\
		& d>1 \mbox{ or $d=1$ and $m$ odd} 	&	\\
		\cline{2-3}
		& m |2n, d=1 \mbox{ and $m$ even}, 
		& 
		(\l^{m_1,\ell},...,\l^{m_{s-1},\ell},\l^{2m_{s},\ell}\cup \{1\}) 
		\mbox{ if } \ell\nmid\gcd(m_{1},\cdots, m_{s-1},2m_{s}),\\
		& \ell:=2n/m \mbox{ even}	&
		\mbox{change any one of the $\l$'s above to $\wt\l$ if $\ell|\gcd(m_{1},\cdots, m_{s-1},2m_{s})$}\\

\hline
C_{n}  	& m| 2n 													&  (\l^{m_1,2n\nu}, \dots, \l^{2m_s,2n\nu})  \\
\hline
D_{n}  	& m|n & (\l^{m_1,2n\nu}, \dots, \l^{2m_s,2n\nu}) \\
		& d>1 \mbox{ or $d=1$ and $m$ odd}& \\
\cline{2-3}
		& m|n, d=1 \mbox{ and $m$ even}				& 
		(\l^{m_1,\ell},...,\l^{m_{s-1},\ell},\l^{2m_{s},\ell})  \mbox{ if } \ell\nmid\gcd(m_{1},\cdots, m_{s-1},2m_{s}) \\
		& \ell:=2n/m &
		\mbox{change any one of the $\l$'s above to $\wt\l$ if $\ell|\gcd(m_{1},\cdots, m_{s-1},2m_{s})$} \\
		\cline{2-3}
		& m|2n-2, (2n-2)\nu \mbox{ odd} 								& (\l^{m_1,(2n-2)\nu}, \dots, \l^{m_s-1,(2n-2)\nu}, \l^{2m_{s-1},(2n-2)\nu}\cup\{1\})\\
		\cline{2-3}
		& m|n-1			& (\l^{m_1,(2n-2)\nu},..., \l^{m_j-1,(2n-2)\nu} \cup \{1\} ,..., \l^{2m_{s},(2n-2)\nu}), j=1,...,s-1,\\ 
		& d>1 \mbox{ or $d=1$ and $m$ odd} & \mbox{ and } (\l^{m_1,(2n-2)\nu},...,\l^{m_{s-1},(2n-2)\nu}, \l^{2m_s-2,(2n-2)\nu} \cup \{1,1\})  \\
		\cline{2-3}
		
 		& m|n-1, d=1 \mbox{ and $m$ even}			
		&  
		(\l^{m_1,\ell},..., \l^{m_j-1,\ell} \cup \{1\} ,..., \l^{2m_{s},\ell}), j=1,...,s-1, \mbox{ and }  \\	
		& \ell:=(2n-2)/m & (\l^{m_1,\ell},...,\l^{m_{s-1},\ell}, \l^{2m_s-2,\ell} \cup \{1,1\}), \mbox{ if }\ell\nmid\gcd(m_{1},\cdots, m_{s-1},2m_{s}-2) \\
		& & \mbox{change any one of the $\l$'s above to $\wt\l$ if $\ell|\gcd(m_{1},\cdots, m_{s-1},2m_{s}-2)$}\\
\hline
 \end{tabular} 
\caption{Explicit solutions in the classical types for fixed characteristic polynomial}

\label{t:clq}

\end{table}
\end{theorem}

\begin{remark} By \cite[Theorem 5.4]{KLMNS}, for $q= \prod_{i=1}^{s} (x-a_i)^{m_i}$, the Deligne-Simpson problem $DS(d/h,\cO,q)$ for isoclinic $\GL(n)$-connections has a solution if and only if $\Res(\Tr(A)) = -\sum_{i=1}^{s} m_i a_i $ and $\cO_{d/h}^{q} \preccurlyeq \cO$. This result is recovered by Theorem \ref{th:hq main}. 
\end{remark}

\begin{proof} \label{pr:hq main} We give the proof only in the case where $m=h$ is the Coxeter number. In these cases $\cO_{\nu}^{q}$ is unique. The general cases are similar, and we leave the details to the reader. Given $\cO \in \Orb_{q}$, by Theorem \ref{th:h main} we need to prove that $\cO_{d/h} \preccurlyeq \cO^{\nil}$ if and only if $\cO_{d/h}^{q} \preccurlyeq \cO$. This is a matter of explicitly computing the Lusztig-Spaltenstein induction, which we do separately in each type. 
\subsubsection*{Type $A_{n-1}$}
Let $q=\prod_{i=1}^{s} (x-a_i)^{m_i}$. Using \cite[1.7. Corollary 2]{K}, we can explicitly calculate $\cO^{\nil}$ for any $\cO \in \Orb_{q}$.  For $i=1,\dots,s$, let $\l^{i}$ be a partition of $m_i$ and write $\l^{i}=(\l^{i}_{1},\l^{i}_{2},\cdots)$ with $\l^{i}_{1}\ge \l^{i}_{2}\ge\cdots$. Then we form a new partition \[\Sigma\lambda = (\sum \l^{j}_{1}, \sum \l^{j}_{2}, \sum \l^{j}_{3}, \cdots ) \]
of $n$. If $\cO$ corresponds to $(\l^{1},\l^{2},\cdots)\in \prod_{i=1}^{s}\Part(m_{i})$, then $\cO^{\nil}$ corresponds to $\Sigma\l$. Clearly, if each $\l^{i}$ has at most $d$ parts, the same holds for $\Sigma\l$. Vice versa, if $\Sigma\l$ has at most $d$ parts, no $\l^{i}$ can have more than $d$ parts. In other words, $\cO_{d/n}\preccurlyeq \cO^{\nil}$ if and only if $\cO_{d/h}^{q} \preccurlyeq \cO  $. 
\subsubsection*{Type $B_{n}$} We assume $G=\SO(2n+1)$. Let $q=x^{2m_s+1} \prod_{i=1}^{s-1} (x^{2}-a^{2}_i)^{m_i}$. Then $L_{q}$ is of the form
\[L_{q}=\GL(m_1) \times \cdots \times \GL(m_{s-1}) \times \SO(2m_s+1) \]
for positive integers $m_i$ with $\sum m_i =n$. Define $\cO_{d/h}^{q}$ to be the orbit corresponding to the collection of partitions $(\l^{m_1,d},\cdots,\l^{m_{s-1},d},\l^{2m_s+1,d})$, where each partition is as evenly distributed as possible. We will prove that $\cO_{d/h} \preccurlyeq \cO^{\nil}$ if and only if $\cO_{d/h}^{q} \preccurlyeq \cO$. More precisely, that means the following. The orbit $\cO$ corresponds to a collection of partitions $\l = (\l^1,\cdots,\l^{m_s})$, with $\l^{i}$ a partition of $m_i$ for $1\le i \le s-1$, and $\l^{s}$ a type $B$ partition of $2m_s+1$. We have to show that the partition corresponding to $\cO^{\nil}$ has at most $d$ parts if and only if each $\l^{i}$ has at most $d$ parts. By \cite[\S 3]{K}, the Lusztig-Spaltenstein induction in type $B$ is computed explicitly as follows. Let $d'$ be a collection of partitions corresponding to a nilpotent orbit for $L$. First, induce as in type $A$ from $\GL(m_1) \times \cdots \times \GL(m_{s-1}) \times \SO(2m_s+1)$ to $\GL(m_1+\cdots+m_{s-1}) \times \SO(2m_s+1)$. Denote the resulting partition of $m_1+\cdots+m_{s-1}$ by $d$ (with parts $d_i$) and the partition of $2m_s+1$ by $f$ (with parts $f_i$). The Lusztig-Spaltenstein induction is then the $B$-collapse of 
\[  p=(2d_1+f_1,2d_2+f_2,\cdots). \]
Call an even part that occurs with odd multiplicity a $B$-violation. Note that taking the $B$-collapse will either leave the number of parts invariant (in case there is an even number of $B$-violations) or it will increase the number of parts by one, adding one part equal to $1$ at the end. 

Now one direction is clear: the Lusztig-Spaltenstein induction can only increase the number of parts, so if $\cO_{d/h} \preccurlyeq \cO^{\nil}$, then no partition occuring in $\cO$ in some factor of the Levi can have more than d parts. This shows  $\cO_{d/h} \preccurlyeq \cO^{\nil}$ implies $\cO_{d/h}^{q} \preccurlyeq \cO$. 

For the other direction, take some nilpotent orbit for $L$, corresponding (as above) to $(\l^1,\cdots,\l^{s-1},\l^{s})$, with $\l^{i}$ a partition of $m_i$ for $1\le i \le s-1$, and $\l^{s}$ a type $B$ partition of $2m_s+1$. 

Taking the induction may increase the number of parts by $1$, but we can only go beyond $d$ in the following situation: 
\begin{enumerate}
\item There is an $i$ for which $\l^i$ has exactly $d$ parts and
\item there is an odd number of $B$-violations  in the partition 
\[p = (2\sum_j \l^{j}_1+\l^{s}_1, 2\sum_j\l^{j}_2+\l^{s}_2,\cdots, 2\sum_j\l^{j}_{d_s}+\l^{s}_{d_s}, 2\sum_j\l^{j}_{d_{s}+1}, \cdots , 2\sum_j \l^{j}_d)\]

where $d_s$ is the number of parts of $\l^s$. 
\end{enumerate}

Because $\l^s$ satisfies the type $B$ parity constraint, there is an even number of $B$-violations in the first $d_s$ parts of $p$. Denote by $p'$ the last $d-d_{s}$ parts of $p$. For 2) to be satisfied, there has to be an odd number of $B$-violations in $p'$. 

But now, $d$ is odd, and $\sum \l^{s}_i$ is odd, so $d_s$ is odd. Therefore, the number of parts of $p'$ is even. But $p'$ consists only of even parts, and if their number is even, there can be no odd number of $B$-violations. 

So 1) and 2) can never hold together, and we see that $\cO_{d/h}^{q} \preccurlyeq \cO$ implies $\cO_{d/h} \preccurlyeq \cO^{\nil}$. 

\subsubsection*{Type $C_n$} The Lusztig-Spaltenstein induction is computed as in type $B_n$, replacing $\SO(2m_s+1)$ by $\Sp(2m_s)$.  In this case, taking the collapse does not change the number of parts. Thus, the claim is clear.
\subsubsection*{Type $D_n$} The argument is very similar to type $B_n$. Let $q=x^{2m_s} \prod_{i=1}^{s-1} (x^{2}-a^{2}_i)^{m_i}$, and assume  $\cO$ correspond to the collection of partitions $(\l^1,...,\l^{m_s})$. For $1\le i\le s-1$, we have that $\l^{i}$ is a partition of $m_i$ and $\l^{s}$ is a partition of $2m_s$ in which each even part occurs with even multiplicity. Let $p$ be the partition corresponding to $\cO^{\nil}$. Note that $\cO_{d/h} \preccurlyeq \cO^{\nil}$ if and only if $p$ has at most $(d+1)$-parts, and if it has exactly $(d+1)$-parts, then its last part has to be $1$. 

We prove the contrapositive of $\cO_{d/h} \preccurlyeq \cO^{\nil}$ implies $\cO_{d}^{q} \preccurlyeq \cO$.  If $\cO^{d}_{q}$ is not a lower bound for $\cO$, then there is $1\le i \le s-1$ such that $\l^{i}$ has more than $d$ parts, or $\l^{s}$ has at least $(d+1)$ parts, and if it has $(d+1)$-parts, its last part is greater than $1$. In either case, by the explicit formula for the Lusztig-Spaltenstein induction, the $(d+1)-st$ part is at least $2$. This shows that $\cO_{d/h}$ is not a lower bound for $\cO^{\nil}$. 

To prove $\cO_{d}^{q} \preccurlyeq \cO$ implies $\cO_{d/h}\preccurlyeq \cO^{\nil}$ we argue as follows. Let $d_s+1$ be the number of parts of $\l^{s}$. If $d_s=d$, then $\l^{s} = (\l^{s}_{1}, ..., \l^{s}_{d}, 1)$, and 
\[p = (2\sum_j \l^{j}_1+\l^{s}_1, 2\sum_j\l^{j}_2+\l^{s}_2,\cdots, 2\sum_j\l^{j}_{d}+\l^{s}_{d}, 1),\]
where we possibly add trivial parts, to get the correct indices. This has an even number of $D$-violations, and they appear in the first $d$ parts, so the $D$-collapse will be bounded from below by $\l^{2n-1,d}\cup\{1\}$. If $d_s \le d-1$, one argues as in type $B$.
\end{proof}

\section{More cases in $F_{4}$}\label{s:F4soln}

In this section we will give answers to $DS(\nu,\cO)$ for $G=F_{4}$ and $\nu=d/m$,  where  $d$ is coprime to $2$ and $3$ (i.e. $d$ is very good in the sense of \cite[\S 3.]{Som}).

By Corollary \ref{c:sp case}, we only need to consider $\nu < 1$. So apart from the epipelagic cases (which correspond to $d=1$) constructed in \cite{Ch} (based on the $\ell$-adic construction in \cite{Yun}), and the Coxeter cases solved in the previous section,  there are three remaining possibilities
\[\nu \in \{5/6, 5/8, 7/8\}. \]

The results are summarized as follows.

\begin{theorem}\label{th:F4} For $\nu \in \{5/6, 5/8, 7/8\}$, there is a nilpotent $\cO_{\nu}$ of $F_{4}$, tabulated in Table \ref{table:DSsolnF4}, such that $DS(\nu,\cO)$ has an affirmative answer if and only if $\cO_{\nu}\preccurlyeq \cO^{\nil}$. (For the definition of $\D_{\nu}$ in the table, see \eqref{Dnu}.)
\end{theorem}

\begin{table}[h]
\begin{center}
\begin{tabular}{|C|C|C|}
\hline
\nu & \cO_{\nu}  & \D_{\nu} \\
\hline
5/6 		& A_1		&	2 \\
5/8 		& \tilde{A}_1	&	0 \\
7/8		& A_1		&	3\\

\hline
\end{tabular}
\end{center}
\caption{Solutions to $DS(\nu,\cO)$ in certain cases for $G=F_4$}
\label{table:DSsolnF4}
\end{table}

The rest of the section is devoted to the proof of the theorem.

For convenience of the reader, in Table \ref{table:springerf4} we give the Hasse diagram for nilpotent orbits in $F_{4}$ and the respective irreducible characters of $W$ under the Springer correspondence, where we copy the ordering of nilpotent orbits for the sake of readability. Both can be found in \cite[\S 13.3., \S 13.4.]{Carter}. Here, we only consider characters coming from pairs $(\cO, \triv)$, consisting of a nilpotent orbit $\cO$ with the trivial character on $A_{\cO}=\pi_{0}(C(\cO))$. The character table for a Weyl group of type $F_4$ was calculated by Kondo and can be found in \cite[\S 13.2.]{Carter}. We follow Kondo's original labelling of irreducible characters. 

\begin{table}[h]
\begin{center}
\begin{tikzpicture}[scale=.8]
\draw (3.5,3.5) -- (3.5,-10.5);
\draw (-3,2.5) -- (10,2.5);

\node (orbits) at (0,3) {Nilpotent orbits for $F_4$};

  \node (F4) at (0,2) {$F_{4}$};
  \node (F4a1) at (0,1) {$F_{4}(a_{1})$};
  \node (F4a2) at (0,0) {$F_{4}(a_{2})$};
  \node (B3) at (-1,-1) {$B_3$};
  \node (C3) at (1,-1) {$C_3$};
  \node (F4a3) at (0,-2) {$F_4(a_3)$};
  \node (C3a1) at (0,-3) {$C_3(a_1)$}; 
  \node (A2s+A1) at (1,-4) {$\tilde{A}_2+A_1$};
  \node (B2) at (-1,-4) {$B_2$}; 
  \node (A2+A1s) at (0,-5) {$A_2+\tilde{A}_1$};
  \node (A2s) at (1,-6) {$\tilde{A}_2$};
  \node (A2) at (0,-6) {$A_2$};
  \node (A1+A1s) at (.5,-7) {$A_1+\tilde{A}_1$};
  \node (A1s) at (.5,-8) {$\tilde{A}_1$};
  \node (A1) at (.5,-9) {$A_1$};
  \node (0) at (.5,-10) {$0$};
  
\draw (F4) -- (F4a1) --(F4a2) --(B3) -- (F4a3) -- (C3a1) -- (B2) -- (A2+A1s) -- (A2) -- (A1+A1s) -- (A1s) -- (A1) -- (0) ;
\draw (F4a2) -- (C3) -- (F4a3);
\draw (C3a1) -- (A2s+A1) -- (A2+A1s);
\draw (A2s+A1)  -- (A2s) -- (A1+A1s);

\node (chars) at (7,3) {Irreducible characters of $W$};

  \node (F4) at (7,2) {$\chi_{1,1}$};
  \node (F4a1) at (7,1) {$\chi_{4,2}$};
  \node (F4a2) at (7,0) {$\chi_{9,1}$};
  \node (B3) at (6,-1) {$\chi_{8,1}$};
  \node (C3) at (8,-1) {$\chi_{8,3}$};
  \node (F4a3) at (7,-2) {$\chi_{12,1}$};
  \node (C3a1) at (7,-3) {$\chi_{16,1}$}; 
  \node (A2s+A1) at (8,-4) {$\chi_{6,1}$};
  \node (B2) at (6,-4) {$\chi_{9,2}$}; 
  \node (A2+A1s) at (7,-5) {$\chi_{4,3}$};
  \node (A2s) at (8,-6) {$\chi_{8,2}$};
  \node (A2) at (7,-6) {$\chi_{8,4}$};
  \node (A1+A1s) at (7.5,-7) {$\chi_{9,4}$};
  \node (A1s) at (7.5,-8) {$\chi_{4,5}$};
  \node (A1) at (7.5,-9) {$\chi_{2,4}$};
  \node (0) at (7.5,-10) {$\chi_{1,4}$};
  
\draw (F4) -- (F4a1) --(F4a2) --(B3) -- (F4a3) -- (C3a1) -- (B2) -- (A2+A1s) -- (A2) -- (A1+A1s) -- (A1s) -- (A1) -- (0) ;
\draw (F4a2) -- (C3) -- (F4a3);
\draw (C3a1) -- (A2s+A1) -- (A2+A1s);
\draw (A2s+A1)  -- (A2s) -- (A1+A1s);
\end{tikzpicture} 
\end{center}
\caption{Springer correspondence for $F_4$ }
 \label{table:springerf4}
\end{table}

We use the following labelling for simple reflections $s_i$ for $F_4$ as depicted in the Dynkin diagram below.
\begin{center} \dynkin[labels={s_1,s_2,s_3,s_4}, scale = 2] F4 \end{center}
Note that this convention differs from the convention in \cite{Carter}, but agrees with the one used in \cite{GP}. 

Using the discussion in \ref{ss:stab}, we need to list $d$-allowable subsets $J\subset \D=\{1,2,3,4\}$ for $d=5, 7$. The results are recorded in the first two columns of Table \ref{table:paraf4}, where parabolic subgroups printed in bold are minimal in its column.

In addition, we can find some irreducible characters which occur in $\Ind_{W_J}^{W} (\triv)$ in \cite[Table C.3.]{GP}. Among those, we list all characters from Table \ref{table:springerf4}. 

\begin{table}[h]
\begin{center}
\begin{tabular}{|C|C|C|}
\hline
d=5 & d=7 & \textup{Characters} 	\\
\hline
			&	\mathbf{12} 	& 				\\
			& 	\mathbf{13}	&	\chi_{8,4}, \chi_{9,4}			\\
\mathbf{23}	&	\mathbf{23}	& 				\\
			&	\mathbf{24}	&				\\
			&	\mathbf{34}	&	\chi_{8,2}			\\
123			&	123	&				\\
\mathbf{124}	&	124	&	\chi_{4,3}, \chi_{9,2}	\\
\mathbf{134}	&	134	&	\chi_{6,1}, \chi_{8,1},\chi_{12,1},\chi_{16,1}		\\
234			&	234	&	\chi_{4,2},\chi_{8,3},\chi_{9,1} 	\\
\hline
\end{tabular}
\end{center}
\caption{Some characters of parabolic type}
\label{table:paraf4}
\end{table}

We use results of Norton who computed the character of $L_{1/m}(\triv)$ for exceptional groups. By \cite[\S 5.]{Norton1}, we have the following character formulas: 
\begin{align}\label{N1/6} \chi_{L_{1/6}(\triv)}(w,t) &= \chi_{1,1}(t^{-2}+1+t^2)+\chi_{4,2}(t^{-1}+t)+\chi_{9,1}, \\
\label{N1/8}
\chi_{L_{1/8}(\triv)}(w,t) &= \chi_{1,1}(t^{-1}+t)+\chi_{4,2}. 
\end{align}

By Corollary \ref{cor:commonsummand}, to determine which $\chi$ from Table \ref{table:springerf4} appear in $L_{d/m}(\triv)$, we only need to test whether $\chi$ has a common component with $L_{1/m}(\triv)$ after restricting to $W_{J}$ for some $J$ in Table \ref{table:paraf4}. Moreover, by Remark \ref{rem:minpara}, it suffices to check for those $J$ printed in bold. Since the trivial character $\chi_{1,1}$ always appears in $L_{1/m}(\triv)$, those $\chi$ that appear in Table \ref{table:paraf4} already pass the test.  The remaining $\chi$'s from Table \ref{table:springerf4} are:
\begin{equation}\label{eqn:charsleft}
\chi_{1,4},\chi_{2,4},\chi_{9,4},\chi_{4,5},\chi_{8,2},\chi_{8,4}.
\end{equation}

In Table \ref{table:characterspara} we list the characters of the parabolic subgroups $W_{J}$ for bold $J$, and fix an ordering among $\Irr(W_{J})$. In the table, $\triv$ is the trivial character, $\std$ is the reflection representation, $\sgn$ is the sign character, $\varepsilon_{i}$ is the non-trivial character of $\langle s_i \rangle$, $\xi$ is the long sign character, and $\eta$ is the short sign character in type $B_2$. 
\begin{table}[h]
\begin{center}
\begin{tabular}{|C|C|C|}
\hline
J 		& \textup{Type} 	& \Irr(W_{J})											\\
\hline
12		& A_2			&\triv, \std, \sgn											\\
13 		& A_1+\tilde{A}_1	&1, \varepsilon_1, \varepsilon_3, \varepsilon_1 \varepsilon_3		\\
23		& B_2			&1, \sgn, \xi, \eta, \std								\\
24		& A_1+\tilde{A}_1 	&1, \varepsilon_2, \varepsilon_4, \varepsilon_2 \varepsilon_4		\\	
34		& \tilde{A}_2		&\triv, \std, \sgn			\\
124		& A_2+\tilde{A}_1	&\triv, \std, \sgn, \varepsilon_4\triv, \varepsilon_4\std, \varepsilon_4\sgn				\\
134		& A_1+\tilde{A}_2	&\triv, \std, \sgn, \varepsilon_1\triv, \varepsilon_1\std, \varepsilon_1\sgn			\\
\hline
\end{tabular}
\end{center}
\caption{Characters of parabolic subgroups}
\label{table:characterspara}
\end{table}

Using the character table for $F_4$ given in \cite[\S13.2.]{Carter}, it is not difficult to compute the restriction of the characters in \ref{eqn:charsleft} and those appear in \eqref{N1/6} and \eqref{N1/8} to $W_{J}$. We obtain Table \ref{table:reschars} with restrictions of these $\chi$'s to $W_{J}$,  given as tuples of multiplicities of $\Irr(W_{J})$ under the ordering fixed in Table \ref{table:characterspara}. For example, the upper left entry means $\Res^{W}_{W_{12}}\chi_{1,4}=\sgn$.

\begin{table}[h]
\begin{center}
\begin{tabular}{|c|C|C|C|C|C|C|C|}
\hline
\diagbox{$\chi$}{$J$}	& 12 		& 13 				& 23 				& 24 			& 34 		& 124 		& 134 		\\
\hline
$\chi_{1,4} $			& (0,0,1)	 & 	(0,0,0,1) 		&  (0,1,0,0,0) 		& (0,0,0,1)		&(0,0,1)	&(0,0,0,0,0,1) 	&(0,0,0,0,0,1) 	\\
$\chi_{2,4}	$		& (0,1,0)	&	(0,0,1,1)		& (0,1,1,0,0)		& (0,0,1,1)		&(0,0,2)	&(0,0,0,0,1,0)	&(0,0,1,0,0,1)	\\	
$\chi_{9,4}	$		& (0,3,3)	&	(1,2,2,4)		&(0,3,1,1,2)		&(1,2,2,4)		&(0,3,3)	&(0,1,1,0,2,2)	&(0,1,1,0,2,2)	\\
$\chi_{4,5}	$		& (0,1,2)	&	(0,1,1,2)		&(0,2,0,0,1)		&(0,1,1,2)		&(0,1,2)	&(0,0,1,0,1,1)	&(0,0,1,0,1,1)		\\	
$\chi_{8,2}	$		&(0,4,2)	&	(1,3,1,3)		&(0,2,0,2,2)		&(1,3,1,3)		&(1,3,1)	&(0,1,2,0,1,2)	&(0,1,0,1,2,1)		\\
$\chi_{8,4}	$		&(1,3,1)	&	(1,1,3,3)		&(0,2,2,0,2)		&(1,1,3,3)		&(0,4,2)	&(0,1,0,1,2,1)	&(0,1,2,0,1,2)		\\
\hline
$\chi_{1,1} 	$		&(1,0,0) 	&	(1,0,0,0)		&(1,0,0,0,0)		&(1,0,0,0)		&(1,0,0)	&(1,0,0,0,0,0)	&(1,0,0,0,0,0)			\\
$\chi_{4,2} 	$		&(2,1,0)	&	(2,1,1,0)		&(2,0,0,0,1)		&(2,1,1,0)		&(2,1,0)	&(1,1,0,1,0,0)	&(1,1,0,1,0,0)			\\
$\chi_{9,1} 	$		&(3,3,0)	&	(4,2,2,1)		&(3,0,1,1,2)		&(4,2,2,1)		&(3,3,0)	&(2,2,0,1,1,0)	&(2,2,0,1,1,0)			\\
\hline
\end{tabular}
\end{center}
\caption{Restrictions to parabolic subgroups}
\label{table:reschars}
\end{table}

From Table \ref{table:reschars}, we can simply read off which characters among \eqref{eqn:charsleft} share a common component with $L_{1/m}(\triv)$  when restricted to some $W_{J}$. This completes the proof of Theorem \ref{th:F4} using Theorem \ref{th:DS} and Corollary \ref{cor:commonsummand}.

\begin{remark} Since the character tables of all exceptional Weyl groups are known, one can in principle use the same strategy to solve the isoclinic Deligne-Simpson problem for $E_6, E_7$ and $E_8$ (assuming $\nu = d/m$, $d$ very good, $m$ regular elliptic). However, this requires a substantial amount of calculation.
\end{remark}

\section{Rigid isoclinic connections}\label{s:rig}

In this section we give some potential examples in which the isoclinic Deligne--Simpson problem has a solution that is cohomologically rigid. We assume for simplicity that $G$ is almost simple. 

\subsection{Local differential Galois group}

Let $\G_{\infty}$ be the local differential Galois group at $\infty$ with its wild inertia subgroup $\G^{+}_{\infty}$. This normal subgroup is a pro-torus, whose character group is the group of Laurent tails $\cQ = \bigcup_{k \ge 1} \tau^{-1/k} \CC[\tau^{-1/k}]$. The tame quotient $\G_{\infty}/\G_{\infty}^{+}$ is a pro-algebraic group over $\CC$, which is a product of $\Ga$ with a pro-group of multiplicative type with character group $\CC/\ZZ$.

Let $(\cE,\nb)$ be an isoclinic $G$-connection of slope $\nu$ over $D_{\infty}^{\times}$. We would like to compute the dimension of horizontal sections of the adjoint bundle $\Ad(\cE)^{\nb}$. Now $(\cE,\nb)$ gives rise to a monodromy representation of $\G_{\infty}$ into $G$ (unique up to conjugacy):
\begin{equation*}
\r_{(\cE,\nb)}: \G_{\infty}\to G.
\end{equation*}
We have $\Ad(\cE)^{\nb}=\frg^{\r_{(\cE,\nb)}(\G_{\infty})}$, the latter we abbreviate as $\frg^{\G_{\infty}}$. 

\begin{lemma}\label{l:irr at infty}
In the above situation, we have:
\begin{enumerate}
\item The irregularity of $\Ad(\cE)$ is $\Irr_{\infty}(\Ad(\cE))=\nu|\Phi|$, where $\Phi$ is the set of roots of $G$ with respect to a maximal torus $T$.
\item $\dim\Ad(\cE)^{\nb}=\dim\frg^{\G_{\infty}}=\dim\frt^{w}$, where $w\in W$ is a regular element of order $m$.
\end{enumerate}
 
\end{lemma}
\begin{proof}

In our situation, the leading term is regular semisimple after pull-back along an $m$-fold cover. Therefore, by \cite[Lemma 2.1.]{Boalch}, the formal type is diagonalizable, and in particular the formal monodromy is semisimple. Thus, the $\Ga$-factor acts trivially. 

Since $\G_{\infty}^{+}$ is a pro-torus, we may assume $\r_{(\cE,\nb)}(\G_{\infty}^{+})\subset T$ for a maximal torus $T\subset G$. Consider the root space decomposition $\frg = \frt \oplus \bigoplus_{\alpha \in \Phi} \frg_{\alpha}$, where $\frt=\Lie T$. Thus, $\G_{\infty}^{+}$ acts trivially on $\frt$ and on each $\frg_{\alpha}$ it acts with slope $\nu > 0$. This immediately implies (1). 

Since $\G_{\infty}^{+}$ is normal in $\G_{\infty}$, we get
\[\frg^{\G_{\infty}} = (\frg^{\G_{\infty}^{+}})^{\G_{\infty}/\G_{\infty}^{+}} = \frt^{\G_{\infty}/\G_{\infty}^{+}}.\]
In the proof of Lemma \ref{l:A}, we showed that the tame quotient $\G_{\infty}/\G_{\infty}^{+}$ acts on $\frt$ through the quotient $\mu_{m}=\Gal(\CC\lr{\t^{1/m}}/\CC\lr{\t})$, and a generator of $\mu_{m}$ acts through a regular element $w\in W$ of order $m$. We conclude that $\frg^{\G_{\infty}} = \frt^{w}$. 
\end{proof}

 \subsection{Cohomological rigidity}\label{ss:coho rig}
 
Recall from \cite[\S5.0.1]{Katz} and \cite[Definition 3.2.4]{Yun-CDM} that an algebraic connection $(\cE,\nb)$ on $\Gm$ is cohomologically rigid if $\dim H^1(\PP^1, j_{!*}\Ad(\cE) ) = 0$. Now $(\cE,\nb)$ gives rise to a monodromy representation of the differential Galois group $\G$ of $\Gm$
\begin{equation*}
\r_{(\cE,\nb)}:\G\to G. 
\end{equation*}
The local differential Galois groups $\G_{0}$ and $\G_{\infty}$ embed into $\G$ up to conjugacy.  We have (see \cite[p.29, after Prop.11]{FG}):
\begin{equation}\label{dimH1}
\dim H^1(\PP^1, j_{!*}\Ad(\cE) ) = \Irr_{0}(\Ad(\cE)) +\Irr_{\infty}(\Ad(\cE))  - \dim \frg^{\G_{0}}-\dim \frg^{\G_{\infty}} + 2\dim \frg^{\G}. 
\end{equation}

We say that an adjoint orbit $\cO$ is {\em non-resonant} if the eigenvalues of any $x\in \cO$ under the adjoint action on $\frg$ do not differ by a nonzero integer. Let $C(\cO)$ denote the centralizer of any $x\in \cO$ in $G$; it is a subgroup of $G$ well-defined up to conjugation.

\begin{prop}\label{p:num rig} Suppose $m$ is an elliptic regular number and $\cO$ is non-resonant.  Then an isoclinic $G$-connection $(\cE,\nb)$ on $\Gm$ of type $(A,\cO)$, where $A$ has slope $\nu=d/m$ (in lowest terms), is cohomologically rigid if and only if
\begin{equation}\label{eqn:rigidity}
\nu|\Phi| = \dim C(\cO).
\end{equation}
\end{prop}
\begin{proof}
We apply \eqref{dimH1} to $(\cE,\nb)$. Since $\cE$ is regular at $0$, $\Irr_{0}(\Ad(\cE))=0$. Moreover, $\G_{0}$ acts through its tame quotient, where a genenator acts on $\frg$ via $\exp(2\pi i x)$ for $x\in \cO$. Since $\cO$ is non-resonant, we have $\dim \frg^{\G_{0}}=\dim C(\cO)$.

The local invariants at $\infty$ are calculated in Lemma \ref{l:irr at infty}. In particular, since $w$ is elliptic, $\dim\frg^{\G_{\infty}} = \dim\frt^{w} = 0$, therefore $\dim \frg^{\G} = 0$. These calculations show that the right hand side of \eqref{dimH1} is $\nu|\Phi|-\dim C(\cO)$. Therefore the left side is zero if and only if $\nu|\Phi|=\dim C(\cO)$.
\end{proof}

\subsection{Index of rigidity}
Let $\nu=d/m$ (in lowest terms) where $m$ is a regular number of $W$.  For an adjoint orbit  $\cO$ of $\frg$ ,define 
\begin{equation*}
\D_{\nu,\cO}=\frac{1}{2}(\nu|\Phi|-\dim C(\cO)+\dim \frt^{w})
\end{equation*}
where $w$ is a regular element of order $m$, the denominator of $\nu$.  Then Proposition \ref{p:num rig} says that if an isoclinic $G$-connection $(\cE,\nu)$ has type $(\nu,\cO)$, and $m$ is elliptic and $\cO$ is non-resonant, then $\cE$ is cohomologically rigid if and only if $\D_{\nu,\cO}=0$. 

By \cite[Theorem 1.3(a)]{LS}, $\dim C(\cO)=\dim C(\cO^{\nil})$. Therefore  
$$\D_{\nu,\cO}=\D_{\nu,\cO^{\nil}}.$$

On the other hand, let $\psi\in\frg\lr{t}$ be homogeneous of slope $\nu$ and suppose $\cO$ is a nilpotent orbit such that $\Gr_{\psi, \ov\cO}\ne\vn$. Let $r=\dim\frt$ be the rank of $G$. By Bezrukavnikov \cite{Be}, we have
\begin{equation*}
\dim \Fl_{\psi}=\dim \Gr_{\psi}=\frac{1}{2}(\nu|\Phi|-r+\dim \frt^{w}).
\end{equation*}
On the other hand, for an element $e\in \cO$, 
\begin{equation*}
\dim \cB_{e}=\frac{1}{2}(\dim C(\cO)-r).
\end{equation*}
Combining these equalities, we get
\begin{equation}\label{diff FlB}
\D_{\nu,\cO}=\dim\Fl_{\psi}-\dim \cB_{e}.
\end{equation}
In particular, $\D_{\nu,\cO}\ge0$. Note that $\cB_{e}$ is the fiber of the projection $\Fl_{\psi}\to \Gr_{\psi}$ over any point of $\Gr_{\psi,\cO}$. Therefore we have 
\begin{equation*}
\D_{\nu,\cO}=\dim \Fl_{\psi}-\dim \cB_{e}\ge \dim \Gr_{\psi,\ov\cO}.
\end{equation*}
When $\D_{\nu,\cO}=0$,  we have that $\Gr_{\psi,\ov\cO}=\Gr_{\psi, \cO}$ is finite.

\subsection{Calculations of $\D_{\nu}$ for classical groups}
When $G$ is an almost simple classical group,  we know from Theorem \ref{th:cl} that $DS(\nu,\cO)$ has an affirmative answer if and only if $\cO^{\nil}\cge \cO_{\nu}$ for some nilpotent $\cO_{\nu}$. Define
\begin{equation}\label{Dnu}
\D_{\nu}=\D_{\nu,\cO_{\nu}}.
\end{equation}
By \eqref{diff FlB}, $\D_{\nu}\ge0$. Therefore for any adjoint orbit $\cO$ such that $\cO^{\nil}\cge\cO_{\nu}$, $\D_{\nu,\cO}=\D_{\nu,\cO^{\nil}}\ge\D_{\nu}\ge0$.  To look for $(\nu,\cO)$ such that $\D_{\nu,\cO}=0$, we must have $\cO=\cO_{\nu}$ and $\D_{\nu}=0$. 

The values of $\D_{\nu}$ in each classical type is listed in Table \ref{t:cl index rig}. In the table, we always set
\begin{equation*}
k=\lfloor \frac{m}{d}\rfloor, d'=m-kd.
\end{equation*}

\begin{table}\label{tab:delta_nu}
\begin{center}
\begin{tabular}[t]{|C|C|C|C|}
\hline
\mbox{type} & \nu & \D_{\nu} & \mbox{Notation} 	\\ 	
\hline

A_{n-1} & m|n & \frac{\ell(\ell-1)}{2}d'(d-d')+\frac{\ell}{2}(d'-1)(d-d'-1) & n=m\ell \\
\cline{2-4}
 & m|n-1 & \frac{\ell(\ell-1)}{2}d'(d-d')+\frac{\ell}{2}(d'-1)(d-d'-1)  & n=m\ell+1 \\
\hline

B_{n} &  m|2n, m \mbox{ even} & 
\begin{cases} \frac{\ell(\ell-1)}{4}d'(d-d')+\frac{\ell}{4}((d'-2)(d-d'-1)-2), & \ell \mbox{ even}, k \mbox{ even}, d>1\\
0, & \ell \mbox{ even}, d=1\\ 
\frac{\ell(\ell-1)}{4}d'(d-d')+\frac{\ell}{4}((d'-1)(d-d'-2)-2), & \ell \mbox{ even}, k \mbox{ odd}\\
\frac{\ell(\ell-1)}{4}d'(d-d')+ \frac{\ell}{4}d'(d-d'-1), & \ell \mbox{ odd}, k \mbox{ even}\\
\frac{\ell(\ell-1)}{4}d'(d-d')+\frac{\ell}{4}(d'+1)(d-d'-2)+\frac{\ell-1}{2}, & \ell \mbox{ odd}, k \mbox{ odd}
\end{cases}& 2n=m\ell\\
\cline{2-4}
& m|n, m \mbox{ odd} & 
\begin{cases}
\frac{\ell(\ell-1)}{4}d'(d-d')+\frac{\ell}{4}((d'-2)(d-d'-1)-1), & k \mbox{ even}\\
\frac{\ell(\ell-1)}{4}d'(d-d')+\frac{\ell}{4}((d'-1)(d-d'-2)-1), & k \mbox{ odd}
\end{cases}
& 2n=m\ell \\
\hline

C_{n}  & m|2n, m \mbox{ even} &  \begin{cases}\frac{\ell}{4}d’(\ell(d-d’)-1) &
k \mbox{ even}  \\ \frac{\ell}{4}(d-d’)(\ell d’-1),& k \mbox{ odd} \end{cases} & 2n=m\ell \\
\cline{2-4}
& m|n, m \mbox{ odd} &  \begin{cases}\frac{\ell}{4}(\ell d’(d-d’)-d’+1) &
k \mbox{ even}  \\ \frac{\ell}{4}(\ell d’(d-d’)-(d-d’)+1),& k \mbox{ odd} \end{cases} & 2n=m\ell \\

\hline
D_{n}  & m|n, m \mbox{ even} & 
\begin{cases}
\frac{\ell^{2}-1}{4}d'(d-d')+\frac{\ell}{4}((d'-2)(d-d'-1)-2), & k \mbox{ even}, d>1\\
\frac{\ell^{2}-1}{4}d'(d-d')+\frac{\ell}{4}((d'-1)(d-d'-2)-2), & k \mbox{ odd}\\
0, & d=1
\end{cases} & 2n=m\ell  \\
\cline{2-4}

& m|n, m \mbox{ odd} & 
\begin{cases}
\frac{\ell^{2}-1}{4}d'(d-d')+\frac{\ell}{4}((d'-2)(d-d'-1)-1), & k \mbox{ even}\\
\frac{\ell^{2}-1}{4}d'(d-d')+\frac{\ell}{4}((d'-1)(d-d'-2)-1), & k \mbox{ odd}
\end{cases} & 2n=m\ell \\

\cline{2-4}

& m|2n-2, m\nmid n-1 & 
\begin{cases}
\frac{\ell(\ell-1)}{4}d'(d-d')+\frac{\ell}{4}d'(d-d'-1), & k \mbox{ even}\\
\frac{\ell(\ell-1)}{4}d'(d-d')+\frac{\ell}{4}(d'+1)(d-d'-2)+\frac{\ell-1}{2}, & k \mbox{ odd}
\end{cases} & 2n-2=m\ell\\

\cline{2-4}
& m|n-1,  m \mbox{ even} & 
\begin{cases}
\frac{\ell(\ell-1)}{4}d'(d-d')+\frac{\ell}{4}((d'-2)(d-d'-1)-2), & k \mbox{ even}, d>1\\
0, & d=1\\
\frac{\ell(\ell-1)}{4}d'(d-d')+\frac{\ell}{4}((d'-1)(d-d'-2)-2), & k \mbox{ odd}
\end{cases} & 2n-2=m\ell\\

\cline{2-4}
& m|n-1,  m \mbox{ odd} & 
\begin{cases}
\frac{\ell(\ell-1)}{4}d'(d-d')+\frac{\ell}{4}((d'-2)(d-d'-1)-1), & k \mbox{ even}\\
\frac{\ell(\ell-1)}{4}d'(d-d')+\frac{\ell}{4}((d'-1)(d-d'-2)-1), & k \mbox{ odd}
\end{cases} & 2n-2=m\ell\\

\hline
 \end{tabular} 
\end{center}
\caption{Index of rigidity in classical types}
\label{t:cl index rig}
\end{table}

\subsection{Rigid cases in classical types}
For classical groups, we list in Table \ref{t:cl ell rig} those slopes $\nu$ with elliptic denominator $m$ such that an isoclinic $G$-connection of type $(\nu,\cO_{\nu})$ is cohomologically rigid, i.e., $\D_{\nu}=0$.

\begin{table}
\begin{center}
\begin{tabular}[t]{|C|C|}
\hline
\mbox{type} & \nu=d/m \\ 	
\hline

A_{n-1} & m=n, d|n\pm 1 \\
\hline

B_{n} & m=2n, d|n+1 \mbox{ or }d|2n+1 \\
\cline{2-2}
& m=n \mbox{ even}, d=3 \\
\cline{2-2}
& m|2n, m \mbox{ even}, d=1 \\
\hline

D_{n}  & m|n, m \mbox{ even}, d=1 \\
\cline{2-2}
& m=n \mbox{ even}, d=3 \\
\cline{2-2}
& m|2n-2, m\nmid n-1, d=1 \\
\cline{2-2}
& m=2n-2, d|2n \mbox{ or } d|2n-1 \\
\hline

\end{tabular} 
\end{center}
\caption{Slopes of rigid isoclinic connections in classical types ($m$ elliptic)}
\label{t:cl ell rig}
\end{table}

We observe:
\begin{itemize}
\item Most of the cases in Table \ref{t:cl ell rig} are either $m=h$ (Coxeter number) or $d=1$ (epipelagic cases), except for the following cases:
\begin{enumerate}
\item Type $B_{n}$, $m=n$ even, $d=3$. 
\item Type $D_{n}$, $m=n$ even,  $d=3$.
\end{enumerate}
It is an interesting problem to construct the $\ell$-adic analog of these cohomologically rigid connections.
\item The only cases where $\nu>1$ are $\nu=1+1/h$, which appears uniformly for any $G$ and $\cO_{\nu}=\{0\}$. This are the de Rham analogue of the Airy sheaves studied in \cite{JKY}. \end{itemize}

\begin{remark} When $m=h$, the Coxeter number of $G$, and $\cO$ is a nilpotent orbit,  the pairs $(\nu=d/h, \cO)$  satisfying $\D_{\nu,\cO}=0$ have been classified by Kamgarpour and Sage \cite[\S 4.3.]{KS}. In addition to the classical cases described before, they discovered one exceptional case, namely in type $E_7$ and $\nu=7/18$.  The cohomologically rigid connections in type $B_n$ (for $\cO$ nilpotent) with $d=3$ were first constructed in A. Azhang's thesis \cite{AA}. 

\end{remark}

\subsection{Exceptional types} 
In Table \ref{t:potigexc}, we give a full list of potential examples of rigid isoclinic connections in exceptional types satisfying the numerical criterion \eqref{eqn:rigidity}, with $\nu = d/m$ and $1<d<m$ coprime to $m$. In the final column we state whether there is such a connection.  If it is unknown we leave the entry empty. The single confirmed case in type $F_4$ ($\nu=5/8$) was obtained in \S\ref{s:F4soln}, and so was the classification in the Coxeter case. The classification in Tables \ref{t:cl ell rig} and \ref{t:potigexc} proves (and generalizes) \cite[Conjecture 5.9.]{Sage}.

The only cases with $\nu >1$ are the Airy cases $\nu=1+1/h$ described in \cite{JKY}.

\begin{table}[h]
\begin{center}
\begin{tabular}[t]{|C|C|C|c|}
\hline
$type$ 	& \nu 	& \cO^{\nil} 	& existence 	\\
\hline
G_2 		& 2/3		& A_1		&	\\
\hline
F_4		& 3/4		& B_4		& 	\\
		& 3/8		& A_2+\tilde{A}_1 &	\\
		& 5/8 	& \tilde{A}_1 	& yes\\
		\hline
E_6		& 2/9 	& A_4+A_1	&	\\
		& 4/9		&A_2+A_1	& 	\\
 		& 7/9 	& A_1 		& 	\\
		& 5/12 	& 2A_2 		& no 	\\
		\hline
E_7		& 3/14	& D_5(a_1)	&	\\
		& 3/14 	& A_4+A_2	& 	\\
		& 9/15	& 2A_1		& 	\\
		& 11/14	& A_1		&  	\\
		& 5/18 	& A_3+A_2	& no\\
		& 7/18	& 2A_2		& no	 \\
		& 7/18 	& A_2+3A_1 	& yes	\\
		\hline
 \end{tabular} 	\quad	
 \begin{tabular}[t]{|C|C|C|c|}
 \hline
 $type$ 	& \nu 	& \cO^{\nil} 	& existence 	\\
\hline
E_8		& 3/10	& D_4(a_1)+A_1 &	\\
		& 5/12	& A_3 		& 	\\	
		& 2/15	& D_6		&	\\
		& 2/15	& E_6		&	\\
		& 2/15	& D_7(a_2)	& 	\\
		& 4/15	& D_4+A_1	&	\\
		& 4/15	& D_4(a_1)+A_2 & 	\\
		& 7/15	& A_2+A_1	&  	\\
		& 3/20	& A_6+A_1	& 	\\
		& 3/20	& E_7(a_4) 	& 	\\
		& 7/20	& A_3+A_1	& 	\\
		& 13/20 	& 2A_1		& 	\\
		& 5/24 	& D_4+A_2	& 	\\
		& 5/24 	& E_6(a_3)	& 	\\
		& 7/24	& A_3+A_2	& 	\\
		& 19/24 	& A_1		& 	\\
		& 7/30	& A_4+2A_1	& no	\\
		& 17/30	& 3A_1		& no	\\
		\hline
 \end{tabular} 
\end{center} 
\caption{Numerics for potential examples of rigid isoclinic connections in exceptional types}
\label{t:potigexc}
\end{table}

\end{document}